\pdfoutput=1
\documentclass[11pt, reqno]{amsart}
\usepackage[foot]{amsaddr}
\usepackage[T1]{fontenc}
\usepackage[utf8]{inputenc}
\usepackage{amssymb,amsmath,amsthm,mathrsfs,array,graphicx,bbm,enumerate,wasysym, dsfont,xcolor,trimclip, mathtools, upgreek, enumitem, tablefootnote, threeparttable}
\usepackage{tikz}
\usetikzlibrary{trees}
\usepackage{todonotes}
\usepackage{algpseudocode}
\usepackage[]{algorithm2e}
\usepackage{setspace}
\usepackage[skip=3pt, indent]{parskip}
\usepackage{caption}
\usepackage{subcaption}
\usepackage[most]{tcolorbox}
\usepackage[backend=bibtex, sorting=nty, style=alphabetic, maxnames=10, maxalphanames=10, sortcites=true]{biblatex}
\DeclareFieldFormat[article]{title}{#1}
\renewbibmacro*{volume+number+eid}{%
  \printfield{volume}%
  \setunit*{\addcomma\space}%
  \printfield{number}%
  \setunit*{\addcomma\space}%
  \printfield{eid}%
}

\DeclareFieldFormat{pages}{#1}
\renewbibmacro{in:}{
  \ifentrytype{article}
    {}
    {\bibstring{in}
     \printunit{\intitlepunct}}
     }
\DeclareFieldFormat[article]{number}{\mkbibparens{#1}}

\usepackage[all]{xy}
\usepackage{bbm}
\usepackage{fullpage}
\usepackage{float}
\usepackage{verbatim}
\usepackage[dvipsnames]{xcolor}
\usepackage{hyperref}

\newtheorem{thm}{Theorem}[section]
\newtheorem{lemma}[thm]{Lemma}
\newtheorem{rem}[thm]{Remark}
\newtheorem{defn}[thm]{Definition}
\newtheorem{prop}[thm]{Proposition}
\newtheorem{cor}[thm]{Corollary}

\newtheorem{con}[thm]{Conjecture}

\numberwithin{equation}{section}

\newcommand{\Z}{\mathbb Z}
\newcommand{\R}{\mathbb R}
\newcommand{\C}{\mathbb C}

\newcommand{\N}{\mathbb N}
\newcommand{\de}{\delta}
\newcommand{\D}{\mathbb D}
\newcommand{\E}{\mathbb E}

\renewcommand{\P}{\mathbb P}
\newcommand{\La}{\mathcal L}
\renewcommand{\1}{\mathbf 1}
\newcommand{\A}{\mathds A}

\renewcommand{\L}{\mathcal L}
\newcommand{\F}{\mathcal F}

\renewcommand{\epsilon}{\varepsilon}
\newcommand{\norm}[1]{\left\lVert#1\right\rVert}

\newcommand{\eps}{\epsilon}

\newcommand{\CLE}{\mathrm{CLE}}
\newcommand{\SLE}{\mathrm{SLE}}

\newcommand{\Int}{\mathrm{Int}}

\newcommand{\Var}{\mathrm{Var}}
\newcommand{\Leb}{\mathrm{Leb}}
\newcommand{\curr}{\mathrm{curr}}
\newcommand{\DRC}{\mathrm{DRC}}
\newcommand{\nb}{\mathbf{n}}

\newcommand{\odd}{\mathrm{odd}}

\newcommand{\CR}{\mathrm{CR}} 

\renewcommand{\O}{\mathcal{O}}

\newcommand{\Hs}[2]{\hyperref[#1]{#2}}

\newcommand{\loc}{{loc}}

\newcommand{\acts}{\curvearrowright}
\newcommand\restr[2]{\ensuremath{\left.#1\right|_{#2}}}
\newcommand{\barcos}{\overline{\mathrm{cos}}}
\newcommand{\Pb}{\mathbf{P}}

\newcommand{\field}[1]{:\! #1 \!: }

\makeatletter
\DeclareRobustCommand{\lasymp}{\lg@asymp{<}}
\DeclareRobustCommand{\gasymp}{\lg@asymp{>}}

\newcommand{\under@asymp}[1]{\clipbox{0pt 0pt 0pt {0.5\height}}{$\m@th#1\asymp$}}
\newcommand{\lg@asymp}[1]{\mathrel{\mathpalette\lg@asymp@{#1}}}
\newcommand{\lg@asymp@}[2]{%
  \vcenter{%
    \offinterlineskip
    \m@th
    \ialign{%
      \hfil##\hfil\cr
      $#1#2$\cr
      \under@asymp{#1}\cr
    }%
  }%
}
\makeatother

\mathtoolsset{showonlyrefs} 

\begin{document}
\title{Excursion decomposition of the XOR-Ising model}
\author{Tom{\'a}s Alcalde~L{\'o}pez$^\dagger$}
\address {$\dagger$ Technische Universität Wien, Institut für Stochastik und Wirtschaftsmathematik, Wiedner Hauptstraße 8-10, Vienna, Austria}
\email{tomas.alcalde@tuwien.ac.at}
\author{Avelio Sep{\'u}lveda$^*$}
\address{$^*$ Departamento de Ingenier\'ia Matem\'atica and Centro de Modelamiento Matem\'atico (IRL-CNRS 2807)\\
  Universidad de Chile\\
  Av. Beauchef 851, Torre Norte, Piso 5\\
  Santiago\\
  Chile}
\email{lsepulveda@dim.uchile.cl}

\begin{abstract}
We study the excursion decomposition of the two-dimensional critical XOR-Ising model with either $+$ or free boundary conditions. In the first part, we construct the decomposition directly in the continuum. This construction relies on the identification of the XOR-Ising field with the cosine or sine of a Gaussian free field (GFF) $\phi$ multiplied by $\alpha = 1/\sqrt{2}$, and is obtained by an appropriate exploration of two-valued level sets of the GFF. More generally, the same construction applies to the fields $:\! \cos(\alpha \phi) \!:$ and $:\! \sin(\alpha \phi) \!:$ for any $\alpha \in (0,1)$.

In the second part, we show that the continuum excursion decomposition arises as the scaling limit of the double random current decomposition of the critical XOR-Ising model on the square lattice. To this end, we exploit the rich Markovian structure of the discrete decomposition and strengthen the convergence of the double random current height function to the continuum GFF by establishing joint convergence with its cosine and sine. We conjecture that for $\alpha \in [1/2,\sqrt{3}/2)$ the continuum excursion decompositions arise as the scaling limit of those of the Ashkin-Teller polarisation field along its critical line.
\end{abstract} 

\maketitle

\tableofcontents


\section{Introduction}
The XOR-Ising model is a particular instance of the Ashkin-Teller model \cite{AT} and can be obtained as the product of two independent Ising models. In the early 2010s, numerical simulations suggested that certain geometric properties of the critical XOR-Ising model are related to specific level sets of the Gaussian free field (GFF) \cite{Wil}. More recently, inspired by the bosonisation framework in physics, it was shown that its correlation functions are related to those of observables of the GFF \cite{cont-boson,IGMC}. In this paper, we show that this correspondence holds also at the level of excursion decompositions in the continuum.

Excursion decompositions are ubiquitous in the study of statistical physics models. They provide a way to uncover independent structure in dependent systems and, at the same time, allow one to reformulate questions about correlation functions as connectivity problems in suitable percolation-type models. 

We say that a random (generalized) function $\psi$ admits an \emph{excursion decomposition} if there exist positive measures $(\mu_k)_{k\geq1}$ and independent random signs $(\xi_k)_{k\geq1}$ such that
\begin{enumerate}[label=(\roman*)]
    \item The excursions $C_k$, that is to say the supports of the measures $\mu_k$, are connected and pairwise disjoint.
    \item The field can be recovered as
    \[
        \psi = \sum_{k\geq1} \xi_k \mu_k,
    \]
    where the sum is understood in an appropriate distributional sense.
\end{enumerate}
A comparison of several known excursion decompositions and their properties is given in Table~\ref{table:exc}.

\begin{table}[htbp]
\centering
\renewcommand{\arraystretch}{1.3}
\begin{tabular}{|l|c|c|c|}
  \hline
  \textbf{} & \textbf{Prop 1} & \textbf{Prop 2} & \textbf{Prop 3} \\ \hline
  \textbf{Brownian motion (\cite{ito}}) & No & Yes & No \\ \hline
  \textbf{Ising and FK--clusters (\cite{CME,  mag-field, ES, FK})} & Yes & No\tablefootnote{Known in the discrete. In the continuum, it is shown in \cite{CGS}.} & Yes \\ \hline
    \textbf{Ising and CDRC--clusters (\cite{ ising-GFF,curr-prob})} & Yes & No\tablefootnote{Known in the discrete. In the continuum, it is also expected to fail.} & Yes \\ \hline
  \textbf{Continuum GFF (\cite{excursion-GFF})} & Yes & Yes & No \\ \hline
  \textbf{Metric graph GFF (\cite{Lupu})} & No & Yes & No \\ \hline
   \textbf{Discrete GFF (\cite{Lupu})} & No & No & No \\ \hline
    \textbf{Percolation spin field (\cite{camia-excur})} & Yes & No & Yes \\ \hline
    \textbf{Loop-soup field (\cite{fields-loo-soups})} & Yes & Not known & No \\ \hline
  \textbf{XOR-Ising and DRC-clusters (\cite{DRC-1}, This Paper)} & Yes & No\tablefootnote{However, the excursions are a function of the underlying GFF when considering the pair of coupled XOR-Ising fields.} & No \\ \hline
\end{tabular}\vspace{5mm}
\caption{Some excursion decompositions in statistical physics, and their properties. \textbf{Prop 1:} Excursion measures are a function of the excursions. \textbf{Prop 2:} Excursions are a function of the field. \textbf{Prop 3:} Macroscopic excursions touch each other.}
\label{table:exc}
\end{table}

The first excursion decomposition appearing in the literature was constructed by It\^o \cite{ito} to describe the Brownian motion. About a decade later, the so-called Edwards--Sokal coupling \cite{FK, ES} provided an excursion decomposition of the Ising model by coupling it to the Fortuin-Kasteleyn (FK) random cluster model \cite{FK}. Over the last decade, many new excursion decompositions have been discovered, both for discrete and continuum models.

On the discrete side, three notable examples of models whose excursion decompositions were described in recent years are the discrete Gaussian free field \cite{Lupu}, the XOR-Ising model \cite{spin-perc-height, DRC-1}, and the Ising model \cite{curr-prob,ising-GFF}. Importantly, the latter two decompositions are based on a double random current (DRC) representation and can be viewed as dual to one another. 

On the continuum side, it has been shown that scaling limits of discrete excursion decompositions yield continuum ones. In \cite{mag-field, CME}, it was proved that, in two dimensions at criticality, the FK--Ising coupling converges to an excursion decomposition of the continuum magnetisation field. In \cite{ising-GFF}, the convergence of the DRC-based decomposition of the Ising model was established, defining a natural coupling between the magnetisation field and the GFF. Furthermore, in \cite{excursion-GFF}, the excursion decomposition of the continuum GFF was built and shown to be the scaling limit of the excursion decomposition of the metric graph GFF. 

The continuum decompositions in \cite{excursion-GFF} and \cite{ising-GFF} are closely related to the theory of level sets of the Gaussian free field \cite{BTLS,ALS1,ALS2, TVS}. Namely, these decompositions can be described through suitable iterative explorations of (two-valued) level sets of the Gaussian free field, and the resulting measures can be identified as the conformal Minkowski contents of their supports.

In the present paper, we focus on constructing the excursion decomposition of the scaling limit of the critical XOR-Ising model in two dimensions. Our starting point is the bosonisation of the Ising model, i.e. the identification of the continuum XOR-Ising field with a trigonometric function of the Gaussian free field \cite{boson-phys, dimer-boson, nesting-field, XOR-dimer, IGMC}. More generally, we construct a decomposition for the fields $\field{\cos(\alpha\phi)}$ and $\field{\sin(\alpha \phi)}$ for any $\alpha \in (0,1)$. We subsequently show that the discrete decomposition of the critical XOR-Ising model converges to its continuum counterpart.

We present our results in two parts. First, we state the excursion decompositions of the continuum fields. Then, we explain the scaling limit results obtained. We also devote a short section to discuss the Ashkin-Teller polarisation field.

\subsection{Continuum results} The following result states the excursion decomposition of the XOR-Ising model in the continuum. Figure \ref{fig:sets-decomp} gives a schematic drawing for the supports of the measures, yet we defer the detailed outline of the construction to Section \ref{sec:decomp}.

\begin{thm}\label{thm:cont-XOR}
    Let $D\subset\C$ be a bounded, simply connected domain. Let $\tau$ be the continuum XOR-Ising model with free/free boundary conditions in $D$. Then, there exists a collection $(\mu_k)_{k=1}^\infty$  of measures supported on $(C_k)_{k=1}^\infty$ such that 
    \begin{align*}
        \tau = \lim_{N\to \infty}\ \sum_{k=1}^{N} \xi_{k} \mu_{k},
    \end{align*}
    where $(\xi_k)_{k=1}^\infty$ are i.i.d. symmetric signs. The sum converges almost surely in the Sobolev space $H^{s}(\C)$ for $s<-1$, under \emph{any ordering} that is independent of the collection of signs. Moreover, 
    \begin{enumerate}[label=(\roman*)]
        \item The excursions $C_k$ are connected and pairwise disjoint almost surely.
        \item Each measure $\mu_k$ is a deterministic function of the respective excursion $C_k$.
        \item The set of excursions $(C_k)_{k=1}^\infty$ is locally finite. That is, almost surely for any $\rho>0$ there are  finitely many $C_k$ with diameter greater than $\rho$. 
        \item For any continuous, bounded function $f:\C\to\R$, 
        \begin{align}\label{eq:L^2_estimate}
        \E\left[\sum_{k=1}^\infty(\mu_k,f)^2\right ]<\infty,
        \end{align} 
        and thus 
        \begin{equation}
           (\tau,f) = \lim_{N\to\infty} \sum_{k=1}^N \xi_k(\mu_k,f) \quad\text{in}\quad L^2(\P).
        \end{equation}
    \end{enumerate}
    The same statement holds for\footnote{In this case, one needs to be careful to work in a domain $D$ whose boundary has small enough dimension (or, instead, restrict to the local Sobolev space $H^s_{loc}(D)$).} the XOR-Ising model with $+/+$ boundary conditions in $D$ by adding a symmetry-breaking (positive) measure $\mu_0$.
\end{thm}

As already stated, the XOR-Ising model can be obtained as a trigonometric function of a GFF times $\alpha=1/\sqrt{2}$.  In particular, Theorem \ref{thm:cont-XOR} is a direct consequence of the analogue result for the Wick sine and cosine of a GFF. We refer to Section \ref{sec:prelim-cont} for all the relevant definitions. 

\begin{thm}\label{thm:cont-general}
    Let $D\subset\C$ be a bounded, simply connected domain. Let $\phi$ be a Gaussian free field with zero boundary conditions in $D$. Then, for any $\alpha\in(0,1)$ there exists a collection $(\mu_k)_{k=1}^\infty$ of measures supported on $(C_k)_{k=1}^\infty$ such that 
            \begin{align*}
            \field{\sin(\alpha \phi )}\  = \lim_{N\to \infty}\ \sum_{k=1}^{N} \xi_{k} \mu_{k},
            \end{align*}
    where $(\xi_k)_{k=1}^\infty$ are i.i.d. symmetric signs. 
    The sum converges almost surely in the Sobolev space $H^s(\C)$ for $s<-1$, under \emph{any ordering} that is independent of the collection of signs. Moreover, properties (i)-(iv) of Theorem \ref{thm:cont-XOR} hold, along with
    \begin{enumerate}[label=(\roman*)]\setcounter{enumi}{4}
        \item The signs $(\xi_k)_{k=1}^\infty$, the measures $(\mu_k)_{k=1}^\infty$ and the excursions $(C_k)_{k=1}^\infty$ are measurable with respect to $\phi$.
    \end{enumerate}
    The same statement holds for\footnote{Under the same caveat as before.} the field $\field{\cos(\alpha\phi+u)}$ for any $u\in (-\pi/2,\pi/2)$ by adding a symmetry-breaking (positive) measure $\mu_0$.
\end{thm}

\begin{figure}[h]
\centering
\begin{subfigure}{.45\textwidth}
  \centering
  \includegraphics[width=.92\linewidth]{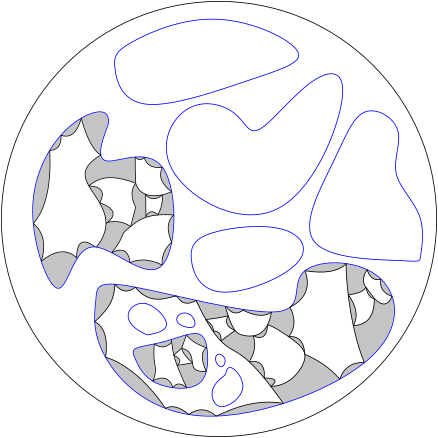}
  \label{fig:free}
\end{subfigure}%
\begin{subfigure}{.45\textwidth}
  \centering
  \includegraphics[width=.92\linewidth]{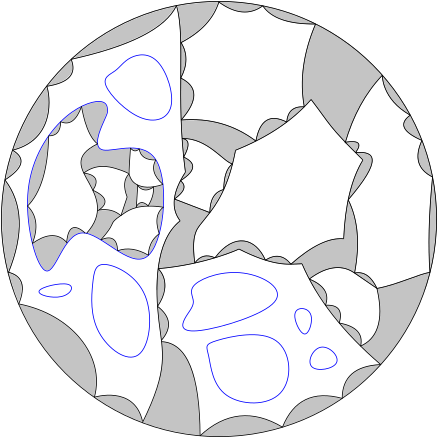}
  \label{fig:+}
\end{subfigure}
\caption{Excursions for the XOR-Ising model, with free/free and $+/+$ boundary conditions respectively, shaded in gray. Iterations are only showed inside certain loops, but should be done within all loops.}
\label{fig:sets-decomp}
\end{figure}

The starting point for the proof of Theorem \ref{thm:cont-general} is the interplay between imaginary chaos and two-valued sets of the GFF, introduced in \cite{dim-TVS}\footnote{Questions concerning the XOR-Ising model were teased to appear in a reference therein, and expected to be written by the second author of this paper. Any curious enough reader can check that this paper has never appeared in the literature (nor will it ever appear). This paper is the spiritual incarnation of the one promised there.}. In particular, each measure in Theorem \ref{thm:cont-XOR} and Theorem \ref{thm:cont-general} is constructed from the ones in {\cite[Proposition 6.1]{dim-TVS}}, and their supports are given by a certain iterative sampling of two-valued sets -- see Section \ref{sec:decomp} for this iteration. Let us give some further comments about these results:

\begin{enumerate}[label=(\roman*)]
    \item By {\cite[Theorem 1.2]{IGMC}}, the fields considered are a.s. not signed measures (i.e. the total variation of the sum of measures is infinite) and the compensation induced by the signs is vital. Hence the existence of such a decomposition is already non-trivial. 
    \item The signs appearing in the decompositions correspond to the labels of the two-valued set $\A_{-2\lambda, 2\lambda}$. As part of the proof, we must show that $\field{\sin(\alpha\phi)}$ does not charge this set with mass, and for $\alpha\in [1/2,1)$ one cannot use classical arguments. Instead, we apply a new technique that will appear in the  upcoming work \cite{CGS}.
    \item While the sum is not absolutely convergent, we prove that any rearrangement of the terms yields the same limit as long as this rearrangement is independent of the signs. This follows from the $L^2(\P)$ convergence in \eqref{eq:L^2_estimate}, as it was explained to us by Ofer Zeitouni. The same argument (see Proposition \ref{prop:sign-permute}) applies to the decompositions for the GFF \cite{excursion-GFF} and the Ising model \cite{mag-field, CME, ising-GFF}. Surprisingly, even though this argument is elementary, it was not noted before in this context.
    \item It follows from {\cite[Proposition 6.2]{dim-TVS}}, that if the conformal Minkowski content of the support of the measures $\mu_k$ exists, then it is almost surely equal to an explicit (deterministic) constant times $\mu_k$.
    \item Theorem \ref{thm:cont-general} gives a decomposition for both the real and the imaginary part of $\field{e^{i\alpha\phi}}$, but the supports of the measures $\mu_k$ in each decomposition do not pair well with one another. We discuss the challenges of proving a \emph{joint} decomposition of $\field{e^{i\alpha\phi}}$ in Remark \ref{rem:joint-decomp}. This is related to the upcoming work of Aru and Lupu \cite{AL26}. 
    \item We expect that analogous decompositions hold for $\alpha=1$, which we discuss in Remark \ref{rem:ALE-decomp}.
    \item On the other hand, we conjecture that no excursion decomposition holds in the range $\alpha\in(1, \sqrt{2})$. Indeed, this is related to the non-existence of two-valued sets with small height gaps {\cite[Proposition 2]{BTLS}}. 
    \item Note that the discrete XOR-Ising model lacks any simple Markov property. This makes the existence of an excursion decomposition even more surprising, as usually these decompositions are chosen from the zero-level set of the field and are proved using the Markov property. Another natural example of a non-Markovian model with such an excursion decomposition is the field built in \cite{fields-loo-soups} out of the (discrete) loop soup.
\end{enumerate}

To prove Theorem \ref{thm:cont-general}, it is key to understand correlation inequalities for $\field{\sin(\alpha\phi)}$ and $\field{\cos(\alpha\phi)}$. This is done in Section \ref{sec:corr-mon}. Said inequalities are the continuum analogue of classical correlation inequalities for the (XOR-)Ising model. The proof of many of them is a direct consequence of similar inequalities for the Green's function. However, it is surprisingly difficult to prove monotonicity in the domain (Theorem \ref{thm:domain-mon}): as expected, the two-point function of $\field{\sin(\alpha \phi)}$ is increasing in the domain but that of $\field{\cos(\alpha \phi)}$ is decreasing, yet they have very similar values for points in the bulk. The key input for this inequality is the classical Hadamard's variational formula for the Green's function, formulated as in \cite{LK-hadamard}.

\subsection{Convergence results}\label{sec:intro-conv}
The excursion decomposition of the continuum XOR-Ising model has a discrete analogue. As stated in {\cite[Proposition 1.2, Corollary 2.7]{spin-perc-height}} and later in {\cite[Corollary 3.3]{DRC-1}}, one can sample a (not necessarily critical) discrete XOR-Ising model by assigning signs symmetrically and independently to each cluster of a double random current (DRC) model. We refer to Section \ref{sec:prelim-discrete} for background on these models. This readily gives an excursion decomposition at the discrete level, where the measures correspond to the area measures of the DRC clusters. Let us restate this result here for completion. 

\begin{thm}[{\cite[Proposition 1.2, Corollary 2.7]{spin-perc-height}}, {\cite[Corollary 3.3]{DRC-1}}]\label{thm:ES-XOR}
    Let $D$ be a Jordan domain with discrete domain approximation $D_\delta\subset \delta \Z^2$. The XOR-Ising model $\tau_\delta$ with free/free boundary conditions in $D_\delta$ can be written as 
    \begin{align*}
        \tau_\delta = \sum_{k=1}^{N} \xi_{k}^\delta \mu_{k}^\delta,
    \end{align*}
    where $(\mu_k^\delta)_{k=1}^N$ is a collection of measures with pairwise disjoint supports $(C_k^\delta)_{k=1}^N$, and $(\xi_k)_{k=1}^N$ are i.i.d. symmetric signs. 
    
    \noindent The same statement holds under $+/+$ boundary conditions by adding a symmetry-breaking (positive) measure $\mu_0^\de$.
\end{thm}

In what follows, we consider only the critical models. Using the main results of \cite{DRC-1, DRC-2}, we show that the decomposition of the (appropriately renormalised) XOR-Ising field converges to the continuum decomposition described in Theorem \ref{thm:cont-XOR}. For the sake of simplicity, we have stated here a somewhat weaker version of the complete statement that will be proved in Section \ref{sec:conv}. We refer to Sections \ref{sec:prelim-discrete}--\ref{sec:conv} for the topologies of convergence.

\begin{thm}\label{thm:discrete-conv}
   Let $\tau_\de$ be a critical XOR-Ising model in the setup of Theorem \ref{thm:ES-XOR}. As $\delta\to 0$, 
    \begin{equation}
        \big(\delta^{-1/4}\tau_\delta,\ (\delta^{-1/4}\mu_k^\delta)_{k\geq1},\ (C_k^\delta)_{k\geq1},\ (\xi_k^\delta)_{k\geq1}\big) \overset{(d)}{\longrightarrow} \big(\tau,\ (\mu_k)_{k\geq1},\ (C_k)_{k\geq1},\ (\xi_k)_{k\geq1}\big),
    \end{equation}
    where the limiting law is that of Theorem \ref{thm:cont-XOR}. Under the obvious modifications, the same statement holds for $+/+$ boundary conditions.
\end{thm}

Let us be more precise now about certain aspects of the joint convergence above. The discrete coupling presented should be seen as a projection of the \emph{master coupling} in {\cite[Theorem 3.1]{DRC-1}}, which we now briefly explain. Let $D_\de^\dagger$ be the \emph{weak} dual graph of $D_\de$. The key idea is that one can define a discrete height function $h_\de$ in $D_\de\cup D_\delta^\dagger$ such that 
\begin{equation}\label{eq:XOR-height-intro}
    \tau_\de^{+}:=\cos(\pi h_\delta) \quad\text{and}\quad \tau_\delta^\dagger:=i\sin(\pi h_\delta)
\end{equation}
are XOR-Ising models with $+/+$ boundary conditions in $D_\de$ and free/free boundary conditions~in~$D_\de^\dagger$, respectively. Moreover, the excursions of these spin fields (i.e. the clusters of the DRC) correspond to the certain level lines of the height function. 

On the one hand, it is known that the height function $h_\de$ converges in distribution to $1/(\pi\sqrt{2})$ times a GFF $h$ {\cite[Theorem 1.4]{DRC-1}}. On the other hand, a simple joint moment computation using {\cite[Theorem 1.1]{CHI-2}} yields that
\begin{equation}
    \big(\delta^{-1/4}\tau_\de,\ \delta^{-1/4}\tau_\delta^{\dagger}\big) \overset{(d)}{\longrightarrow} \big( \field{\cos((1/\sqrt{2})h)},\ \field{\sin((1/\sqrt{2})h)}\big),
\end{equation}
where again the law of $h$ is that of a GFF with zero boundary conditions. In order for Theorem \ref{thm:discrete-conv} to follow readily from Theorem \ref{thm:cont-XOR} and the convergence statements in \cite{DRC-1, DRC-2}, one must first show that\footnote{Note that the correct parameter $\alpha=1/\sqrt{2}$ can be immediately deduced from the variance of the limiting GFF.}
\begin{equation}\label{eq:XOR-joint}
    (h_\de,\ \delta^{-1/4}\tau_\de,\ \delta^{-1/4}\tau_\delta^{\dagger}) \overset{(d)}{\longrightarrow} (1/(\pi\sqrt{2})h,\ \field{\cos((1/\sqrt{2})h)},\ \field{\sin((1/\sqrt{2})h)}).
\end{equation}
That is, the base field of the continuum chaos is precisely the limit of the discrete height function. This may look obvious at first glance, however it is not at all trivial that measurability should be preserved upon taking limits and thus that the limits commute with the the trigonometric functions.

 We prove \eqref{eq:XOR-joint} using a result of independent interest: any ``reasonable'' local set coupling of the imaginary multiplicative chaos must also be a local set coupling of the underlying GFF. We defer the precise statement to Section \ref{sec:local}. Since our arguments for the joint convergence in \eqref{eq:XOR-joint} do not rely on computing the joint moments of the triple, it can readily be applied to couplings with other fields. 
 
 In {\cite{ising-GFF}}, the master coupling above is extended in such a way that the \emph{same} height function~$h_\delta$ can be coupled to four Ising models
\begin{equation}
    (\sigma_\de, \tilde\sigma_\de, \sigma_\de^\dagger, \tilde\sigma_\de^\dagger),
\end{equation}
where $(\sigma_\de, \tilde\sigma_\de)$ are independent Ising models with $+$ boundary conditions in $D_\delta$ and $(\sigma_\de^\dagger, \tilde\sigma_\de^\dagger)$ are independent Ising models with free boundary conditions in $D_\delta^\dagger$. In fact, one defines
\begin{equation}
    \tilde\sigma_\de := \sigma_\de\tau_\de,\quad \tilde\sigma_\de^\dagger := \sigma_\de^\dagger\tau_\de^\dagger,
\end{equation}
where $(\tau_\de, \tau_\de^\dagger)$ are as in \eqref{eq:XOR-height-intro}. In \cite{ising-GFF} it is shown that, as $\de\to0$,
\begin{equation}\label{eq:ising-gff}
    \big(h_\de,\ \delta^{-1/8}\sigma_\de,\ \delta^{-1/8}\tilde\sigma_\de,\ \delta^{-1/8}\sigma_\de^\dagger,\ \delta^{-1/8}\tilde\sigma_\de^\dagger\big) \overset{(d)}{\longrightarrow} \big(1/(\pi\sqrt{2})h,\ \sigma,\ \tilde\sigma,\ \sigma^\dagger,\ \tilde\sigma^\dagger\big),
\end{equation}
where e.g. $\sigma$ denotes a (continuum) Ising magnetisation field, first defined in \cite{mag-field}. The results in this paper can be combined with those of \cite{ising-GFF} to obtain the following result.

\begin{cor}\label{cor:joint-conv}
    In the above setup, the joint law
    \begin{equation}
        \big(h_\de,\ \delta^{-1/8}\sigma_\de,\ \delta^{-1/8}\tilde\sigma_\de,\ \delta^{-1/4}\tau_\delta,\ \delta^{-1/8}\sigma_\de^\dagger,\ \delta^{-1/8}\tilde\sigma_\de^\dagger,\ \delta^{-1/4}\tau_\delta^\dagger\big)
    \end{equation}
    converges as $\delta\to0$ to
    \begin{equation}
        \big(1/(\pi\sqrt{2})h,\ \sigma,\ \tilde\sigma,\ \field{\cos((1/\sqrt{2})h)},\ \sigma^\dagger,\ \tilde\sigma^\dagger,\ \field{\sin((1/\sqrt{2})h)}\big),
    \end{equation}
    where we recall $h$ is a GFF with zero boundary conditions. Moreover, 
    \begin{equation}\label{eq:wick-prod}
        \field{\sigma\tilde\sigma}\ =\ \field{\cos(h/\sqrt{2})}\quad \text{and}\quad \field{\sigma^\dagger\tilde \sigma^\dagger}\ =\ \field{\sin(h/\sqrt{2})}.
    \end{equation}
   Here $\field{\sigma\tilde\sigma}$ denotes the Wick product of two Ising fields given by
   \begin{equation}
       (\field{\sigma\tilde\sigma}, f) = \lim_{\eps\to0}(\sigma^\eps\tilde\sigma^\eps, f),
   \end{equation}
   where $\sigma^\eps, \tilde\sigma^\eps$ are smooth $\eps$-mollifications of $\sigma, \tilde\sigma$ respectively.
\end{cor}

\noindent As was pointed out to us by Marcin Lis, the identifications in \eqref{eq:wick-prod} can be combined with \cite[Theorem 1]{GFF-IGMC} to recover the law of a GFF as a measurable function of four Ising models appropriately coupled. We refer to Remark \ref{rem:meas-relation} for further discussion on the very strict analogies between the discrete and continuum couplings in this regard.

\begin{rem}
    The interplay between the decomposition of $\tau$ and the decompositions of the dual fields $\sigma^\dagger$ and $\tilde\sigma^\dagger$ is easy to describe. Namely, the excursions of the latter pair are supported precisely in the unshaded areas of Figure \ref{fig:sets-decomp}.
\end{rem}

\subsection{A short note on the AT polarisation}
On a final note, let us briefly discuss potential generalizations to the Ashkin-Teller (AT) polarisation field, in the spirit of \cite[Section 1.3]{ising-GFF}. We recall that the Ashkin-Teller model, say with free boundary conditions in $D_\de$, is given by the probability measure
\begin{equation}
    \mathbf{P}^{AT}_{J,U}(\sigma_\de, \tilde\sigma_\de) = \frac{1}{Z_{J,U}^{AT}}\exp\left\{ \sum_{x\sim y}\left(J\big(\sigma_\de(x)\sigma_\de(y) + \tilde\sigma_\de(x)\tilde\sigma_\de(y)\big)+U\sigma_\de(x)\sigma_\de(y)\tilde\sigma_\de(x)\tilde\sigma_\de(y)\right) \right\}
\end{equation}
on pairs of interacting spin configurations $(\sigma_\de, \tilde\sigma_\de)\in\{\pm1\}^{D_\de} \times \{\pm1\}^{D_\de}$. The AT polarisation is defined as $$\tau_\de=\sigma_\de\tilde\sigma_\de,$$ which at the decoupling point $U=0$ is precisely the XOR-Ising model. Along the critical line \cite{AT-crit} given by
\begin{equation}
    \sinh(2J) = e^{-2U} \quad\text{and} \quad U\leq J,
\end{equation}
it has been predicted in the physics literature \cite{kada-brown, nienhuis, c1} that the scaling limit of the polarisation field is given by $\field{\cos(\alpha\phi)}$ and $\field{\sin(\alpha\phi)}$ for an explicit value $$\alpha=\alpha(U)\in[1/2, \sqrt{3}/2).$$ We refer to {\cite[Conjecture 1.1]{ising-GFF}} for a restatement of these conjectures in the current setup. The discrete counterpart of Theorem \ref{thm:ES-XOR} can be found implicitly in {\cite[Proposition 13]{AT-curr}}, using the Ashkin-Teller currents introduced in that same paper, and an alternative proof can be found in \cite[Section 2.3]{ising-GFF}. Moreover, the conjectural scaling limit of the Ashkin-Teller currents as stated in {\cite[Conjecture 1.2]{ising-GFF}} matches the excursions in our decompositions. Thus, one arrives at the following natural conjecture:
\begin{con}\label{conj:AT}
    The excursion decomposition of the continuum Ashkin-Teller polarisation field with parameter $U$ is given by those in Theorem \ref{thm:cont-general} for the appropriate value $\alpha=\alpha(U)$.
\end{con}

\bigskip
\smallskip

\noindent \textbf{Outline:} We begin by discussing preliminaries concerning the continuum models in Section \ref{sec:prelim-cont}. We then give a detailed description of the decomposition in Section \ref{sec:decomp}, explaining the key mechanisms that we rigorously derive later throughout Sections \ref{sec:thin}-\ref{sec:cos-TVS}. The concluding proofs of Theorems \ref{thm:cont-XOR}-\ref{thm:cont-general} are in Section \ref{sec:main-proofs}. We then  make a small aside in Section \ref{sec:local} to prove that local sets of imaginary chaos give local sets of the GFF. Finally, we devote Section \ref{sec:prelim-discrete} to preliminaries concerning the discrete models, and give the proof of Theorem \ref{thm:discrete-conv} and Corollary \ref{cor:joint-conv} in Section \ref{sec:conv}. 

\noindent \textbf{Acknowledgments:} We are grateful to Marcin Lis for getting us started on this project and many helpful discussions. We thank Ofer Zeitouni for pointing out the very helpful observation of Proposition \ref{prop:sign-permute}. We thank Fredrik Viklund for directing us to Theorem \ref{thm:hadamard}. Finally, we thank Juhan Aru for many insightful discussions and sharing with us part of the upcoming work \cite{AL26}.

\noindent This project was started during the workshop \emph{Quantum fields and probability II} at Institut Mittag-Leffler. Part of the work was carried out during and supported by the program \emph{Probabilistic methods in quantum field theory} at the Hausdorff Institute of Mathematics funded by the Deutsche Forschungsgemeinschaft (DFG, German Research Foundation) under Germany's Excellence Strategy – EXC-2047/1 – 390685813.

\noindent T.A.L. was funded by the Austrian Science Fund (FWF) grant ``Spins, loops and fields'' P36298, and by a grant of the Vienna School of Mathematics (VSM). The research of A.S. was supported by Centro de Modelamiento Matem\'{a}tico Basal Funds FB210005 from ANID-Chile, by Fondecyt Grant 1240884, and by ERC 101043450 Vortex. 


\section{Preliminaries in the continuum} \label{sec:prelim-cont}

\subsection{Gaussian free field} Let us start by defining the two-dimensional Gaussian free field (GFF) and introducing the basic properties we make use of. For further background we refer the reader to \cite{She,notes-GFF,BP}.

Let $D \subset \C$ be a simply connected domain. The Gaussian free field  $\phi$ with zero boundary conditions in $D$ is the centred Gaussian process with covariance given by the Green's function $G_D$ of the Laplacian with zero boundary conditions in $D$. Recall that
\[
G_D(z,w) = - \log |z-w| + \log \CR(w,D) + o(1)
\qquad \text{as } z \to w,
\]
where $\CR(w,D)$ denotes the conformal radius of $D$ seen from $w$ (see e.g. \cite[Theorem~1.23]{BP}). Since the Green's function is singular on the diagonal, the GFF cannot be defined pointwise and is instead a random distribution (generalised function). More precisely, it can be defined as a centred Gaussian process $\{ (\phi,f) : f \in C_c^\infty(D)\}$ 
with covariance given by
\[
\E\left[(\phi,f)(\phi,g)\right ]
= \iint_{D \times D} f(z) G_D(z,w) g(w) dz  dw.
\]
It is known (see e.g. \cite[Theorem 1.45]{BP}) that for every $s<0$ the GFF admits a version in the Sobolev space $H^{s}(D)$. If $\rho_{z,\eps}$ denotes the uniform probability measure on the circle $\partial B(z, \eps)$ of radius $\eps$ around $z$, then the circle average process
\[
\phi_\eps(z) := (\phi, \rho_{z,\eps})
\]
is well defined and continuous in $(z,\eps)$ (see \cite[Section 1.12]{BP}). Moreover, $\{\phi_\eps(z):z\in D, \eps>0\}$ is a centred Gaussian process with covariance given by
\[
\E \left [\phi_\eps(z)\phi_{\eps'}(z')\right ]\ 
\begin{cases}
= G_D(z,z') ,
& \text{if } B(z,\eps), B(z',\eps') \subset D \text{ and } |z-z'| \geq \eps+\eps',\\[0.2cm]
= -\log(\eps\vee \eps') + \log \CR(z,D) ,
& \text{if } z=z' \text{ and } B(z,\eps) \subset D,\\[0.2cm]
\leq G_D(z,z') ,
& \text{otherwise.}
\end{cases}
\]

\subsection{Imaginary multiplicative chaos}\label{ss.IMC}
We now briefly recall the construction and basic properties of imaginary multiplicative chaos, referring to \cite{IGMC, IGMC2} for more details. 

Let $D\subset\C$ be a bounded, simply connected domain. Let $\phi$ be a GFF with zero boundary conditions on $D$. For any $\epsilon>0$, consider the renormalized approximations\footnote{This normalization choice differs from that in \cite{IGMC, IGMC2} by a (domain-dependant) factor of $\CR(z,D)^{-\alpha^2/2}$.}
\begin{equation}\label{eq:IGMC-def}
\epsilon^{-\alpha^2/2}e^{i\alpha\phi_\epsilon(z)},
\end{equation}
which we can view as (well-defined) functions supported in $D$. For $0<\alpha<\sqrt{2}$, by {\cite[Proposition 3.6, Theorem 3.12]{IGMC}}, we have that, as $\eps\to0$,
\begin{equation}
    (\epsilon^{-\alpha^2/2}e^{i\alpha\phi_\epsilon(z)}, f):=\int_{D}\epsilon^{-\alpha^2/2}e^{i\alpha\phi_\epsilon(z)}f(z)dz \longrightarrow (\field{e^{i\alpha\phi}}, f)
\end{equation}
exists as a limit in $L^p(\P)$ for any $p\in[1, \infty)$ and any test function $f\in C^\infty_c(D)$. Moreover, the moments of $(\field{e^{i\alpha\phi}}, f)$ grow slow enough to uniquely characterize its law \cite[Theorem 3.12]{IGMC}, and thus one can define the imaginary multiplicative chaos field, denoted by  $\field{e^{i\alpha\phi}}$ as the unique distribution in the Sobolev space $H^{s}_{loc}(D)$ for any $s<-1$. The same renormalization procedure can be used to define the fields $\field{\cos(\alpha\phi)}$ and $\field{\sin(\alpha\phi)}$,
corresponding to the real and imaginary parts of $\field{e^{i\alpha\phi}}$ respectively. 

\begin{rem}
It is crucial that, above, all test functions are compactly supported in $D$. Indeed, the expectation
\begin{equation}\label{eq:int-blowup}
    \E[(\field{e^{i\alpha\phi}}, f)] = \int_{D}\CR(z, D)^{-\alpha^2/2}f(z)dz
\end{equation}
features a (potential) boundary blow-up that is only controlled by the compactness of the support of~$f$. Later, however, we must deal with test function that may not be compactly supported, and understand precisely which conditions on the regularity of $\partial D$ ensure the integrability of the fields along it. For details see Section \ref{sec:thin}.
\end{rem}

Let us also highlight that the underlying Gaussian structure allows one to explicitly compute all $k$-point correlation functions. For instance, 
\begin{equation}
    \E[(\field{e^{i\alpha\phi}}, f)^2] = \int_D \int_D \CR(z, D)^{-\alpha^2/2}\CR(w, D)^{-\alpha^2/2}\exp(-\alpha^2G(z,w))f(z)f(w)dzdw,
\end{equation}
\begin{equation}
    \E[|(\field{e^{i\alpha\phi}}, f)|^2] = \int_D \int_D \CR(z, D)^{-\alpha^2/2}\CR(w, D)^{-\alpha^2/2}\exp(\alpha^2G(z,w))f(z)f(w)dzdw.
\end{equation}
The ``kernel'' two-point functions in the integrands are denoted by $$\langle\field{e^{i\alpha\phi(z)}}\field{e^{i\alpha\phi(w)}}\rangle, \quad {\langle \field{e^{i\alpha\phi(z)}}\field{e^{-i\alpha\phi(w)}}\rangle}$$ respectively. Similarly, we have that
\begin{equation}\label{eq:cos-2pt}
    \langle \field{\cos(\alpha\phi(z))}\field{\cos(\alpha\phi(w))}\rangle = \CR(z, D)^{-\alpha^2/2}\CR(w, D)^{-\alpha^2/2}\cosh(\alpha^2G(z,w)),
\end{equation}
\begin{equation}\label{eq:sin-2pt}
    \langle \field{\sin(\alpha\phi(z))}\field{\sin(\alpha\phi(w))}\rangle = \CR(z, D)^{-\alpha^2/2}\CR(w, D)^{-\alpha^2/2}\sinh(\alpha^2G(z,w)).
\end{equation}

\subsection{Level sets of the GFF} We recall now the construction of the \emph{two-valued sets} of the GFF first defined in \cite{BTLS}. First, we present the definition of a local set coupling of the GFF as introduced in \cite{SchSh}.

\begin{defn}[Local sets]
A coupling $(\phi,A,\phi_A)$ between a GFF $\phi$ in $D$, a random closed set $A \subset \bar{D}$ and a distribution $\phi_A$ is called a \emph{local set (coupling)} if
\begin{itemize}
    \item the restriction of $\phi_A$ to $D\setminus A$ is almost surely harmonic,
    \item conditionally on $\F_A := \sigma(A,\phi_A)$, the field
    \[
        \phi^A := \phi - \phi_A
    \]
    is a GFF with zero boundary conditions in $D \setminus A$.
\end{itemize}
\end{defn}

Let us remark that, as explained in e.g. \cite{BTLS}, the property of being a local set depends only on the joint law of $(\phi,A)$ since $\phi_A$ is almost surely determined by $(\phi,A)$. Local sets are natural analogues of stopping times for Brownian motion and, in settings where filtrations are available, they are closely related to stopping sets -- see \cite{Aru,ASZ} for further discussion. One of the most important families of local sets arising in the study of the geometry of the GFF are the \emph{two-valued sets}.

\begin{defn}[Two-valued sets]
Let $a,b\in \R$ and let $u$ be a harmonic function in $D$. Let $\phi$ be a GFF with boundary condition $u$ in $D$.  We say that a random closed set $\A_{-a,b}^u$ is a two-valued set with boundary condition $u$ if
\begin{itemize}
    \item $(\phi,\A_{-a,b}^u)$ is a local set,
    \item there exists a random function $h_{\A_{-a,b}^u}$ such that 
    \[
        (\phi_{\A_{-a,b}^u},f)=\int_{D\backslash \A_{-a,b}^u} h_{\A_{-a,b}^u}(z) f(z)dz
        \quad \text{a.s.}
    \]
    for all $f \in C_c^\infty(D)$,
    \item the function $h_{\A_{-a,b}^u}-u$ takes values in the set $\{-a,b\}$.
\end{itemize}
\end{defn}

\noindent In this work, we only need to consider boundary conditions $u$ that are constant on each connected component\footnote{For this part, it is useful to consider $D$ that are not connected. In this case, the GFF are just independent GFFs in each connected component of $D$.} of $ D$. Observe that, if $D$ is connected and $u$ is constant, then $$\A_{-a,b}^u = \A_{-a-u,b-u}.$$

The following theorem ensures the existence and uniqueness of two-valued sets. The case of a simply connected domain and zero boundary conditions is proved in \cite[Proposition 2]{TVS}, while the general case of multiply connected domains and non-constant boundary conditions is treated in \cite[Theorem 3.21]{ALS2}. 

\begin{thm}[Existence and uniqueness of TVS]\label{thm.E!TVS}
Let $a,b>0$ satisfy $a+b \geq 2\lambda = \pi$, and let $u$ be a harmonic function in $D$ whose boundary values lie in $[-a,b]$. Let $\phi$ be a GFF with boundary condition $u$ in $D$. Then,  almost surely there exists a unique two-valued set $\A_{-a,b}^u$. Moreover, $\A_{-a,b}^u$ is a measurable function of $\phi$.
\end{thm}

\noindent The boundaries of the connected components of $D\setminus \A_{-a,b}^u$ define a countable collection of simple loops, i.e. non-self-touching closed curves. We often denote this collection by $\La_{-a,b}^u$. The following result ensures that there are only finitely many macroscopic loops.

\begin{prop}[{\cite[Proposition 4.17]{ALS1}}]\label{prop:TVS-local-finite}
The collection of loops $\La_{-a,b}^u$ is locally finite. That is, almost surely for every $\rho>0$ there are finitely many connected components of $D\backslash \A_{-a,b}^u$ with diameter greater than or equal to $\rho$.
\end{prop}

\noindent Every loop $\ell\in\La_{-a,b}^u$ can be associated a \emph{label} depending on the value taken by the harmonic function inside the loop. To be precise, denote by $\mathcal{O}(\ell)$ the domain encircled by $\ell$ and define its label to be
\begin{equation}
    c(\ell) := \begin{cases}
        +1\quad \ \text{if}\quad h_{\A_{-a,b}^u}-u=b\quad \text{in}\quad \mathcal{O}(\ell), \\
         -1\quad \ \text{if}\quad h_{\A_{-a,b}^u}-u=-a \quad \text{in}\quad \mathcal{O}(\ell).
    \end{cases}
\end{equation}

In this work, as in many others, the two-valued set $\A_{-2\lambda,2\lambda}$ plays a special role due to the symmetry of its labels. The following result was originally announced by Miller and Sheffield, and first written down in \cite[Section 3]{TVS}.

\begin{thm} Let $D$ be a simply connected domain and let $\phi$ be GFF with zero boundary conditions in $D$. The law of $\A_{-2\lambda,2\lambda}$ is that of the carpet of a $\CLE_4$. Moreover, conditionally on $\A_{-2\lambda,2\lambda}$, the labels are independent and symmetric.
\end{thm}

We will also use the following fact, which can be seen as a monotonicity statement for local sets. 

\begin{prop}\label{prop:TVS-unique}
Let $(\phi,A)$ be a local set such that almost surely
\[
|\phi_A(z)| \geq a 
\]
for all $z \in D \setminus A$. Then,
\[
\A_{-a,a} \subseteq A \quad \text{a.s.}
\]
\end{prop}

\begin{proof}
This follows directly from the proof of uniqueness in \cite[Section 6]{TVS}. Indeed, the set $\A_{-a,a}$ can be constructed using only level lines with boundary values in $[-a,a]$, and such level lines cannot enter any connected component of $D \setminus A$ by \cite[Lemma 16]{TVS}.
\end{proof}

On a final note, for any $a\geq\lambda$, we define
\[
\alpha_c = \alpha_c(a) := \frac{\lambda}{a} \in (0,1].
\]
This is a crucial quantity that will be used to control the fractal dimension of two-valued sets, and plays the role of a critical parameter in the relationship between the imaginary chaos and two-valued sets (which we introduce in the next subsection).

\begin{prop}[\cite{BTLS,dim-TVS}] \label{prop:TVS-one-point}
Let $a \geq \lambda$ and $u \in \mathbb{R}$. The Minkowski dimension of the two-valued set $\A_{-a,a}^u$ satisfies
\[
\dim_M(\A_{-a,a}^u)
\leq 2 - \frac{\alpha_c(a)^2}{2}
\quad \text{a.s.}.
\]
\end{prop}

\noindent We note that almost sure equality for the Hausdorff dimension was established in \cite{dim-TVS}, but for our purposes the  upper bound above, already present \cite[Section 6.3]{BTLS}, is sufficient.

\subsection{TVS and the imaginary chaos}
The main reason to define the parameter $\alpha_c$ comes from studying the conditional law of the imaginary chaos $\field{e^{i\alpha\phi}}$ given a two-valued set $\A_{-a,a}$. Let us first note that the conditional expectation is itself a well-defined random variable in the space $H^{s}_\loc(D)$, since
\begin{equation}
    \lvert\lvert\E[\field{e^{i\alpha\phi}}\mid\F]\rvert\rvert_{H^s_\loc(D)} \leq \E[\lvert\lvert{\field{e^{i\alpha\phi}}}\rvert\rvert_{H^s_\loc(D)} \mid\F]<\infty.
\end{equation}
Moreover, it is easy to check (see e.g. {\cite[Lemma 3.5]{miller-schoug}}) that
\begin{equation}
    \E[(\field{e^{i\alpha\phi}}, f)\mid\F] = (\E[\field{e^{i\alpha\phi}}\mid\F], f)\quad \text{a.s.}
\end{equation}
for any $\F$-measurable test function $f\in C_c^\infty(D)$. A direct computation of the conditional expectation given a \emph{subcritical} two-valued set gives the next result. 

\begin{prop}[{\cite[Proposition 3.1]{dim-TVS}}] \label{prop:subcrit-CE} 
    Let $a\geq\lambda$ and $\alpha<\alpha_c(a)$. Then, for any $f\in C_c^\infty(D)$, 
    \begin{align*}
    & \E[(\field{e^{i\alpha\phi}},f)\mid\F_{\A_{-a,a}}] = \int_D  e^{i\alpha h_{\A_{-a,a}}(z)}\CR(z,D\setminus\A_{-a,a})^{-\alpha^2/2}f(z)dz.
    \end{align*}
\end{prop}

\noindent Already at this stage, one can identify a very drastic change in the behaviour between the cosine and the sine, which will be recurrent in the sequel. First, note that we can equivalently integrate over $D\setminus\A_{-a,a}$ since $\Leb(\A_{-a,a})=0$ almost surely. Moreover, $h_{\A_{-a,a}}(z)\in\{\pm a\}$ for any $z\in D\setminus\A_{-a,a}$ and taking the real part in Proposition \ref{prop:subcrit-CE} yields
\begin{equation}\label{eq:cos-condTVS}
    \E[(\field{\cos(\alpha\phi)},f)\mid\F_{\A_{-a,a}}] = \cos(\alpha a)\int_D \CR(z,D\setminus\A_{-a,a})^{-\alpha^2/2}f(z)dz,
\end{equation}
where we have used the fact that $\cos(\pm x)=\cos(x)$. In contrast, when taking the imaginary part, since $\sin(\pm x)=\pm \sin(x)$, one cannot reduce the expression further than 
\begin{equation}
    \E[(\field{\sin(\alpha\phi)},f)\mid\F_{\A_{-a,a}}] = \sum_{\ell\in\mathcal L_{-a,a}} c(\ell)\sin(\alpha a)\int_{\mathcal O(\ell)}\CR(z, D\setminus\Int(\ell))^{-\alpha^2/2}f(z)dz,
\end{equation}

Finally, we introduce the measures we use to build our excursion decomposition. Such measures were first defined in \cite{dim-TVS} and correspond to the conditional expectation of $\field{\cos(\alpha\phi)}$ given the \emph{critical} two-valued set $\A_{-a_c, a_c}$. In particular, they correspond to the limit $\alpha\nearrow\alpha_c$ in \eqref{eq:cos-condTVS}. When taking such limit, the blow-up of the integral is prevented by the fact that $\cos(\alpha a)\to0$ as $\alpha\to\alpha_c$, which is essentially the content of the following result. We highlight, however, that the existence of the analogous limit for $\field{\sin(\alpha\phi)}$ is not known, as the necessary cancellations should be induced by the (non-i.i.d.) labels of the two-valued set.

\begin{prop}[{\cite[Proposition 6.1]{dim-TVS}}] \label{prop:TVS-meas}
    Fix $a\geq\lambda$, $u\in\R$. Denote $\A=\A_{-a,a}^u$. For $\eta>0$, define the random measures
    \begin{equation}
        \mu_\eta(dz) = \eta\CR(z, D\setminus\A)^{-(\alpha_c-\eta)^2/2}dz.
    \end{equation}
    Then, for any $f\in C_c^\infty(D)$, 
    \begin{equation}
        \lim_{\eta\to0}(\mu_\eta, f) = (\mu, f)\in[0, \infty)
    \end{equation}
   exists as a limit in $L^1(\P)$ and
    \begin{equation}
       a(\mu, f) = \E[(\field{\cos(\alpha_c\phi)}, f)\mid\F_\A].
   \end{equation}
   In particular, $(\mu_\eta)_{\eta>0}$ converge in law with respect to the weak topology of measures to a random measure $\mu$, which is measurable with respect to $\A$ and supported in $\A$. 
\end{prop}


\section{Construction of the decomposition} \label{sec:decomp}

The purpose of this section is to describe the excursion decompositions considered throughout the paper. The key idea is to understand how a given two-valued set ``acts'' on  the fields $\field{\cos(\alpha\phi)}$ or $\field{\sin(\alpha\phi)}$ by describing the conditional law of the field (or, equivalently, the restriction of the field to the two-valued set). For the sake of exposition, we will write for example
\begin{equation}\label{eq:TVS-act}
    \A_{-a,a}\acts\ \field{\cos(\alpha\phi)}
\end{equation}
to denote the ``action'' of exploring the values of $\field{\cos(\alpha\phi)}$ on top of $\A_{-a,a}$. To keep track of the signs appearing after such actions, 
we reserve the notation $\xi(\ell)$ for the labels of the loops $\ell\in\L_{-2\lambda, 2\lambda}$ to emphasize that they are i.i.d. coin tosses. Slightly abusing notation, we write any of the following interchangeably
\begin{equation*}
    \field{\cos(\alpha(\phi^\A\pm a))}\ \equiv\ \field{\cos(\alpha(\phi^\A+c(\ell) a))}\ \equiv \sum_{\ell\in\mathcal{L}_{-a,a}}\field{\cos(\alpha(\phi^{\ell}+c(\ell)a))},
\end{equation*}
where $\phi^{\ell}$ are independent GFFs with zero boundary conditions in the domain $\mathcal O(\ell)$ encircled by $\ell$. 

\subsection{Properties of the \textit{action}}
Let us describe the key features (and challenges) presented by actions of the form \eqref{eq:TVS-act}, and how to use them in order to derive an excursion decomposition for the fields. These properties will be proved throughout Sections \ref{sec:thin}-\ref{sec:cos-TVS}, so the following can be thought of as a checklist to be completed in the upcoming sections. Without loss of generality, we will only consider the case $u=0$ in Theorems \ref{thm:cont-general}.

\begin{enumerate}[label=(\textbf{\Roman*})]
    \item If $a<a_c(\alpha)=\lambda/\alpha$, the action $\A_{-a, a}\acts\ \field{\cos(\alpha\phi)}$ is given by
    \begin{align}
        &\field{\cos(\alpha\phi)}\ =\ \field{\cos(\alpha(\phi^\A+c(\ell) a))},
    \end{align}
   and similarly for the action $\A_{-a, a}\acts\ \field{\sin(\alpha\phi)}$.
   This statement should be understood as saying that any subcritical two-valued set is thin for either field, i.e. too small to charge them with mass. Rigorously, we will show -- see Proposition \ref{prop:thin-chaos} -- that
    \begin{equation}
        (\field{\cos(\alpha\phi)}, f\1_{[\A_{-a,a}]_n})\overset{\P}{\longrightarrow}0 \quad\text{as}\quad n\to\infty,
    \end{equation}
    where $[\A]_n$ denotes the union of all dyadic boxes of level $n$ that intersect $\A$.
    \item If $a=a_c(\alpha)=\lambda/\alpha$, the action $\A_{-a_c, a_c}\acts\ \field{\cos(\alpha\phi)}$ is given by
    \begin{equation}
    \field{\cos(\alpha\phi)}\ = \mu-c(\ell)\field{\sin(\alpha\phi^\A)},
    \end{equation}
    where $\mu\equiv\mu({a_c})$ is the measure supported on $\A_{-a_c, a_c}$ defined in Proposition \ref{prop:TVS-meas}. More generally, for any $u\in(-\pi/2, \pi/2)$ the action $\A^u_{-a_c, a_c}\acts\ \field{\cos(\alpha(\phi+u))}$ is given by 
    \begin{equation}
    \field{\cos(\alpha(\phi+u))}\ = \mu-c(\ell)\field{\sin(\alpha\phi^\A)}.
    \end{equation}
    Rigorously, we will show -- see Proposition \ref{prop:cos-TVS} -- that
    \begin{equation}\label{eq:ii}
        (\field{\cos(\alpha(\phi+u))}, f\1_{\left [\A^u_{-a_c,a_c}\right ]_n})\overset{\P}{\longrightarrow}(\mu,f) \quad\text{as}\quad n\to\infty.
    \end{equation}
    This statement should understood as saying that one can ``take the conditional expectation out'' and view the measure as the restriction of the field to the (critical) two-valued set. 
    \item For any $\alpha\in(0,1)$, even if $a_c \geq 2\lambda$, the action $\A_{-2\lambda, 2\lambda}\acts\ \field{\sin(\alpha\phi)}$ is given by 
    \begin{equation}
        \field{\sin(\alpha\phi)}\ =\lim_{N\to \infty} \sum_{k=1}^N\xi_k\field{\cos(\alpha(\phi^\A+\xi_k(\pi-a_c))},
    \end{equation}
    where the limit is taken in $L^2(\P)$ and the order of the sum is indpendent of the labels. The rigorous statement is deferred to Section \ref{sec:thin} -- see Lemma \ref{lem:CLE4-thin}. For the case $\alpha \in [1/2,1)$, this statement cannot be obtained using the classical techniques of \cite{thin-GFF}. Instead, we adapt a new technique developed in \cite{CGS}. 
\end{enumerate}

At first glance, the reader might be surprised by the absence of the action $\A_{-a_c, a_c}\acts\ \field{\sin(\alpha\phi)}$ analogous to (\textbf{II}). To see why, we note that \eqref{eq:ii} crucially relies on the fact that the renewed field $\field{\sin(\alpha\phi^\A)}$ can be defined regardless of the dimension of the boundary of the domain. This is roughly the content of Proposition \ref{prop:bdy-sin}. However, for the action $\A_{-a_c, a_c}\acts\ \field{\sin(\alpha\phi)}$ the renewed field would become $\field{\cos(\alpha\phi^\A)}$, which can only be defined when the two-valued set is subcritical. 

In fact, the correct counterpart to (\textbf{II}) should correspond to a thinness statement. That is, the restriction of $\field{\sin(\alpha\phi)}$ to the critical two-valued set should \emph{not} yield a measure. To prove such a statement, one should rely on sign cancellations coming from the labels $c(\ell)$ of the critical-two valued set (in other words, the renewed field \emph{does} have mass inside the microscopic loops of the two-valued set, is simply ought to cancel out). Property (\textbf{III}) is precisely a statement of this kind, but its proof heavily realies on the fact that the labels of $\A_{-2\lambda, 2\lambda}$ are i.i.d. and symmetric.

\begin{rem}
    The discussion above suggests a natural distinction between \emph{absolutely thin sets} and \emph{conditionally thin sets}, which is already hinted in \cite{thin-GFF}. This distinction makes reference to whether for a field $\psi$ and a set $A$ either
    \begin{align*}
    (\psi, f\1_{[A]_n}) \to 0 \ \ \text{ as }n\to \infty,
    \end{align*}
    for the case of conditionally thin sets, or
    \begin{align*}
\sum_{\square\subset [A]_n} |(\psi, f\1_{\square})| \1_{\{A\cap \square\neq \emptyset\}} \to 0 \ \ \text{ as }n\to \infty,
    \end{align*}
    for the case of absolutely thin sets.  
    
    \noindent The  reason why this distinction is useful is that basic operations between thin sets (such as union, intersection or containment) do not always result in thin sets. However, if one of the sets is absolutely thin, then everything works as expected. We stress that all known proofs of thinness starting from an assumption on the dimension prove absolute thinness.
\end{rem}

\subsection{Sine decomposition for $\alpha\in(0,1)$}\label{sec:sin1}
The key idea of all decompositions is to obtain the i.i.d. signs from the labels of properly chosen $\A_{-2\lambda, 2\lambda}$. Regardless of the value $\alpha\in(0,1)$, this can be done by directly sampling $\A_{-2\lambda, 2\lambda}$ thanks to $(\textbf{III})$, and summing up all measures in the next step using $(\textbf{II})$. The schematic diagram of the iteration procedure is given in Figure \ref{fig:sin-CLE-decomp}. The algorithm for the decomposition, initialized at $n=1$ and under the convention that $A_0\equiv \partial D$, is as follows: 

\SetKw{KwGet}{to get}
\SetKw{KwDo}{do}
\SetKw{KwRepeat}{repeat}
\begin{algorithm}[h!]
    \TitleOfAlgo{\textbf{SineDecomposition}($\alpha$)}\vspace{2mm}
    $A_n= A_{n-1}\cup \A_{-2\lambda,2\lambda}(\phi^{A_{n-1}})$ \;
    \KwDo $A_{n}\acts\ \field{\sin(\alpha\phi^{A_{n-1}})}$ \KwGet $\field{\sin(\alpha(\phi^{A_{n}}+\xi(\ell)\pi))}\ \equiv\xi(\ell)\field{\cos\left (\alpha(\phi^{A_n}+\xi(\ell)(\pi-a_c))\right )}$  \; 
    $n=n+1$ \;
    \For{all loops $\ell$ with $\xi(\ell)=+1$}{
        $A_n= A_{n-1}\cup\A_{-2\lambda,2a_c- 2\lambda}(\phi^{A_{n-1}})$ \;
        \KwDo $A_n\acts\ \field{\cos(\alpha(\phi^{A_{n-1}}+(\pi-a_c)))}$ \KwGet $\xi(\ell)\left( \mu - c(\tilde \ell) \field{\sin(\alpha\phi^{A_{n}})}\right )$\tcp*{Inside $\ell$}
    }
    \For{all loops $\ell$ with $\xi(\ell)=-1$}{
        $A_n=A_{n-1}\cup \A_{2\lambda-2a_c, 2\lambda}(\phi^{A_{n-1}})$ \;
        \KwDo $A_{n}\acts-\field{\cos(\alpha(\phi^{A_{n-1}}-(\pi-a_c)))}$ \KwGet $ \xi(\ell)\left ( \mu -c(\tilde \ell)\field{\sin(\alpha\phi^{A_{n}})}\right )$\tcp*{Inside $\ell$}
    }
    $n=n+1$ \;
    \KwRepeat inside each loop of $A_{n-1}$ with diameter bigger than some fixed cutoff $\rho>0$.
\end{algorithm}

\begin{figure}[h]
    \centering
    \includegraphics[width=110mm]{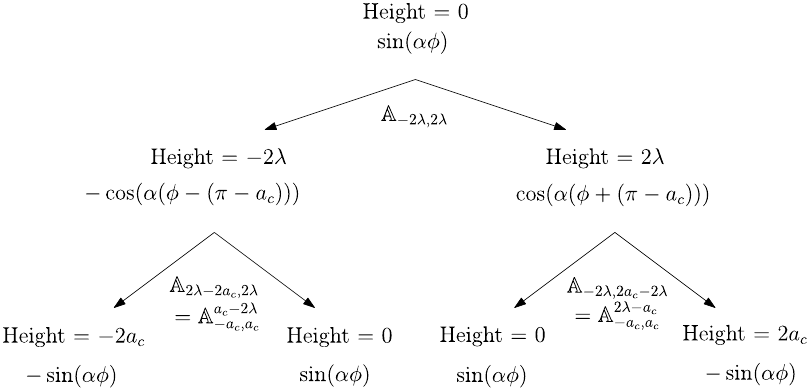}
    \caption{Schematic diagram of the excursion decomposition for $\field{\sin(\alpha\phi)}$. Within all loops of macroscopic diameter (i.e. bigger than some fixed $\rho>0$), we iterate the same procedure.}
    \label{fig:sin-CLE-decomp}
\end{figure}

\noindent Note that the signs appearing in front of the measure $\mu$ are indeed the labels $\xi(\ell)$ of $\A_{-2\lambda, 2\lambda}$. Observe also that, while subsequent layers carry over the labels $c(\tilde\ell)$, this does not break the claimed independence in the overall decomposition since the signs picked in each layer are independent of all those picked in the previous ones, and symmetry is preserved. 

\subsection{Cosine decomposition for $\alpha\in(0, 1)$}

To treat the decomposition of the cosine, it suffices to invoke (\textbf{II}) and continue as above. The schematic diagram is in Figure \ref{fig:cos-decomp}. The precise steps, initializing from $n=1$, are as follows: 

\SetKw{KwUse}{use}
\begin{algorithm}[h!]
    \TitleOfAlgo{\textbf{CosineDecomposition}($\alpha$)}\vspace{2mm}
    $A_n = \A_{-a_c, a_c}(\phi^{A_{n-1}})$\;
    \KwDo $A_n\acts\ \field{\cos(\alpha\phi^{A_{n-1}})}$ \KwGet $\mu - c(\ell)\field{\sin(\alpha\phi^{A_n})}$ \;
    \KwUse\textbf{SineDecomposition($\alpha$)}.
\end{algorithm}

\begin{figure}[h]
    \centering
    \includegraphics[width=70mm]{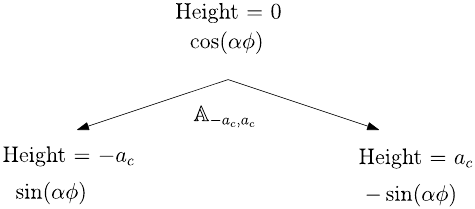}
    \caption{Schematic diagram of the excursion decomposition for $\field{\cos(\alpha\phi)}$. Within all loops of macroscopic diameter (i.e. bigger than some fixed $\rho>0$), we iterate the procedure in Figure \ref{fig:sin-CLE-decomp}.}
    \label{fig:cos-decomp}
\end{figure}

\begin{rem}\label{rem:joint-decomp}
        We stress that the excursion sets used in the decompositions for $\field{\sin(\alpha\phi)}$ and $\field{\cos(\alpha\phi)}$ are different. To obtain a joint decomposition of $\field{e^{i\alpha\phi}}$, the same excursions would be needed. However, this requires proving the analogue of (\textbf{II}) for $\field{\sin(\alpha\phi)}$, as already discussed. Note that for $\alpha=1/2$, since $a_c=2\lambda$, the joint decomposition does hold and is given by always iterating $\A_{-2\lambda, 2\lambda}$. This is closely related to the upcoming work of Aru and Lupu \cite{AL26}.
\end{rem}
\begin{rem}\label{rem:ALE-decomp}
        For $\alpha=1$, since the smallest two-valued set $\A_{-\lambda, \lambda}$ is already critical and (\textbf{I}) does not hold for the boundary of the loops of $\A_{-2\lambda,2\lambda}$, we are not able to prove (\textbf{III}). Nonetheless, we expect Theorem \ref{thm:cont-general} to also hold for $\alpha=1$, but instead one should obtain thinness of the $\SLE_4$-type lines using the equality between the left- and right-sided (conformal) Minkowski contents of $\SLE_4$ \cite{quant-SLE}.
\end{rem}
    

\section{Subcritical TVS and $\CLE_4$: Thinness} \label{sec:thin}

The goal of this section is to rigorously prove the thinness results (\textbf{I}) and (\textbf{III}) from Section \ref{sec:decomp}. 

\subsection{Correlation inequalities}\label{sec:corr-mon}

We begin by stating some straightforward correlation inequalities that we use extensively. For convenience, we write
\begin{equation}
    \field{\barcos(\alpha\phi)}\ =\  \field{\cos(\alpha\phi)}-\langle\field{\cos(\alpha\phi)}\rangle
\end{equation}
for the ``truncated'' cosine. 

\begin{prop}\label{prop:funct-ineq}
Let $\alpha\in(0,\sqrt{2})$.  
    \begin{enumerate}[label=(\roman*)]
        \item\label{prop-1} For any $n\in\N$ and any $z_1\ldots z_n\in D$, $$\left\langle \prod_{k=1}^n\field{\sin(\alpha\phi(z_k))}\right\rangle\leq\left\langle \prod_{k=1}^n\field{\cos(\alpha\phi(z_k))}\right\rangle.$$
        \item\label{prop-2} For any $z,w\in D$,
        \begin{align*}
            \langle\field{\barcos(\alpha\phi(z))}\field{\barcos(\alpha\phi(w))}\rangle\leq\langle\field{\sin(\alpha\phi(z))}\field{\sin(\alpha\phi(w))}\rangle.
        \end{align*}
        \item\label{prop-3} For any $u\in[-\pi/2,\pi/2]$ and any $z,w\in D$, 
        \begin{align*}
            \langle\field{\barcos(\alpha\phi(z)+u)}\field{\barcos(\alpha\phi(w)+u)}\rangle \leq\langle\field{\sin(\alpha\phi(z))}\field{\sin(\alpha\phi(w))}\rangle.
        \end{align*}
    \end{enumerate}
\end{prop}
\begin{proof}\
    \begin{enumerate}[label=(\roman*)]
        \item For odd $n$ the left-hand side vanishes by symmetry. Otherwise, we can write the sine (resp. cosine) as a difference (resp. sum) of imaginary exponentials, yielding the claim. Note that the case $n=2$ corresponds to the fact that $\sinh(x)\leq\cosh(x)$ for all $x\in\R$.
        \item This follows from a direct computation and the fact that $\cosh(x)-1\leq\sinh(x)$ for all $x\geq0$.
        \item Again, this follows from a direct computation using $(ii)$ and $\cos(u)^2+\sin(u)^2=1$.
    \end{enumerate}
\end{proof}

\begin{rem}
    For $\alpha=1/\sqrt{2}$, the second (resp. third) inequality states that the truncated two-point function of XOR-Ising with $+$ (resp. ``intermediate'') boundary conditions is dominated by the two-point function of XOR-Ising with free boundary conditions. In the discrete, the analogous statement for the Ising model (which is stronger) follows from the GHS inequality, or the more general \cite{DSS}.
\end{rem}

We now turn to study a different feature of the two-point correlation functions: monotonicity in the domain. Surprisingly, proving this monotonicity is subtle, needing to first study the behaviour of the two-point functions under infinitesimal deformations of the domain. This monotonicity will be a crucial input in the sequel. In particular, \eqref{eq:ineq_sin} allows us to bound the two-point correlation function of $\field{\sin(\alpha\phi)}$ in an arbitrary domain $D$ by that on the whole plane $\C$. As a consequence, $\field{\sin(\alpha\phi)}$ can be defined in any domain (see Proposition \ref{prop:bdy-sin}), and its $L^2(\P)$-norm can be uniformly bounded.

\begin{thm}\label{thm:domain-mon}
    Let $D, \tilde D\subseteq\C$ be two simply connected domains such that $D\subseteq\tilde D$. Let $\phi$ and $\tilde\phi$ be two GFFs in the respective domains. We have that:
    \begin{enumerate}[label=(\roman*)]
        \item For any $\alpha\in(0,\sqrt{2})$ and any $z,w\in D$, 
        \begin{equation}
             \langle \field{\cos(\alpha\phi(z))}\field{\cos(\alpha\phi(w))}\rangle \geq \langle \field{\cos(\alpha\tilde\phi(z))}\field{\cos(\alpha\tilde\phi(w))}\rangle. 
        \end{equation}
        \item For any $\alpha\in(0,1]$ and any $z,w\in D$, 
        \begin{equation}\label{eq:ineq_sin}
             \langle \field{\sin(\alpha\phi(z))}\field{\sin(\alpha\phi(w))}\rangle \leq \langle \field{\sin(\alpha\tilde\phi(z))}\field{\sin(\alpha\tilde\phi(w))}\rangle.
        \end{equation}
    \end{enumerate}
\end{thm}

Some comments about this result are due. For $\alpha=1/\sqrt{2}$, the statements directly correspond to the known monotonicity of the Ising two-point function: it is decreasing (resp. increasing) under $+$ (resp. free) boundary conditions. In the continuum, however, this statement is purely a statement about the Green's function. Namely, the two-point function of the cosine (resp. sine) consists of a decreasing term, coming from the conformal radii, times an increasing term, given by $\cosh(\alpha^2G(z,w))$ (resp. $\sinh(\alpha^2G(z,w))$). Thus the monotonicity is determined by whether the increasing term wins over the decreasing term, and the subtlety of the statement becomes apparent: for $x,y\in D$ close to each other, $\cosh(\alpha^2G(z,w))\sim\sinh(\alpha^2G(z,w))$. Finally, we highlight the (rather surprising) range of values of $\alpha$ for which the statement holds, needing to restrict to $\alpha\in(0,1]$ for $(ii)$ to be true. 

The following general result concerning variations of the Green's function is the central piece in the proof of Theorem \ref{thm:domain-mon}. Under smoothness assumptions and less general growth of the domain, the result is classical, see e.g. {\cite[Chapter 6, Problem 15]{evans}}. For our purposes, we use (a special case of) the statement as presented in {\cite[Lemma 3.1]{LK-hadamard}}, see also {\cite[Theorem 4]{other-hadamard}}. To do so, we briefly recall the definition of a radial Loewner chain. We refer to \cite{lawler, ber-SLE} for further background.

Let $w_t:[0, \infty)\to \R$ be a continuous function and set $\zeta_t=e^{iw_t}:[0, \infty)\to\partial\D$. For each $z\in\D$, the radial Loewner differential equation driven by $\zeta_t$ is 
\begin{equation} \label{eq:lowner}
    \partial_tf_t(z) = -zf'_t(z)\frac{\zeta_t+z}{\zeta_t-z}, \quad f_0(z)=z.
\end{equation}
The unique solution to \eqref{eq:lowner} gives a collection $(f_t)_{t\geq0}$ of conformal transformations $f_t:\D\to f(\D)\equiv D_t$ such that $f_t(0)=0$ and $f'_t(0)=e^{-t}>0$. In particular, the family $(D_t)_{t\geq0}$ is decreasing. For each $z\in\D$, let $\tau(z)=\sup\{t\geq0: z\in D_t\}$.

\begin{thm}[Hadamard's Variational Formula -- {\cite[Lemma 3.1]{LK-hadamard}}]\label{thm:hadamard}
    Consider the radial Loewner equation driven by $\zeta_t$. For a.e. $t<\tau(z)\wedge\tau(w)$ and all $z,w\in D_t$, the map $t\mapsto G_t(z,w)$ is absolutely continuous and 
    \begin{equation}
       -\partial_t G_t(z,w) = P_{\D}(g_t(z), \zeta_t)P_\D(g_t(w), \zeta_t),
    \end{equation} 
    where $g_t=f^{-1}_t:D_t\to\D$ and $P_\D(z, \zeta)=\frac{1-|z|^2}{|z-\zeta|^2}$ is the Poisson kernel in $\D$.
\end{thm}

Inverting the setup of the previous discussion is also standard, and will be very useful in the proof of Theorem \ref{thm:domain-mon}. Suppose now we are given a simple curve $\gamma:[0, \infty)\to\bar\D$ with $\gamma(0)\in\partial\D$ and $\gamma(t)\in\D$ for all $t>0$. Setting $D_t=\D\setminus\gamma([0, t])$, under an appropriate time-reparametrization of $\gamma$, there exists a driving function $\zeta$ and maps $(f_t)_{t\geq0}$ as above that satisfy the radial Loewner equation \eqref{eq:lowner}. We recall that the same claim holds for more general classes of decreasing domains, and that not every Loewner chain is generated by a curve, yet this shall suffice for our purposes.

\begin{proof}[Proof of Theorem \ref{thm:domain-mon}] \textbf{Step 0 -- Full-plane:} The case $\tilde D=\C$ follows from that of bounded domains by fixing $z, w\in D$, comparing to the ball $B(0, R)$ once it contains both points, and taking $R\nearrow\infty$.

\noindent \textbf{Step 1 -- Reduction to Loewner chains:} Up to a conformal mapping, it suffices to let $\tilde D=\D$ and let $D\subset \D$ be any simply connected domain containing the origin. By (the proof of) {\cite[Lemma 4.23]{lawler}}, the space of such domains under the Carathéodory topology has a dense subset given by \emph{slit domains}. That is, domains of the form $\D\setminus\gamma([0,t])$ for a curve curve $\gamma:[0,\infty)\to\bar\D$ with $\gamma(0)=1$, $\gamma(\infty)=0$ and $0\notin\gamma((0,\infty))\subset\D$. Let $\zeta$ be the driving function associated to the curve $\gamma$ used to approximate $D$. Since the Green's function is continuous under the convergence of domains in the Carathéodory topology\footnote{For us, since the sequence of domains is decreasing, this follows from {\cite[Chapter VII, Section 6]{doob}}. The general claim may be verified directly by combining Harnack's inequality for harmonic functions with a Beurling estimate. For further references see {\cite[Lemma 4.2]{ALS2}}, {\cite[Section 3]{CS}}, \cite{Cara}.}, it suffices to show that, in the same setup of Theorem \ref{thm:hadamard},
\begin{align}
     & -\partial_t C^{\alpha}_t(z,w)\leq0, \\
     & -\partial_t S^{\alpha}_t(z,w)\geq0,
\end{align}
where $C^{\alpha}_t(z,w), S^{\alpha}_t(z,w)$ are the two-point functions of the cosine and the sine, respectively. 

\noindent\textbf{Step 2-- Compute the time-derivatives:} We write $H_t(z,w)$ for the regular part of the Green's function, i.e.
\begin{equation}
    G_t(z, w) = -\log|z-w| + H_t (z,w).
\end{equation}
A direct computation gives that
\begin{align}\label{eq:cos-mon}
    &-\partial _t C^\alpha_t(z,w) = Z_t^\alpha(z,w)\left[X_t(z,w)\sinh(\alpha^2G_t(z,w))-Y_t(z,w)\cosh(\alpha^2G_t(z,w))\right], \\ \label{eq:sin-mon}
    & -\partial_t S^\alpha_t(z,w) = Z_t^\alpha(z,w)\left[X_t(z,w)\cosh(\alpha^2G_t(z,w))-Y_t(z,w)\sinh(\alpha^2G_t(z,w))\right],
\end{align}
where 
\begin{align}
    & X_t(z, w) = -\partial_t G_t(z,w)\geq0, \\
    & Y_t(z,w) = -\frac{1}{2}(\partial_t H_t(z,z) + \partial_t H_t(w,w))\geq0, \\
    & Z_t^\alpha(z,w) = \alpha^2\exp\left(-\alpha^2/2(H_t(z,z) + H_t(w,w))\right)\geq0.
\end{align}
Since $\partial_t G_t(z,w) = \partial_t H_t(z,w)$, Theorem \ref{thm:hadamard} holds whenever $z=w$ and
\begin{align}
    R_t(z,w) &:= Y_t(z,w)/X_t(z,w) \\
    & = \frac{1}{2}\frac{\partial_tH_t(z,z) + \partial_tH_t(w,w)}{\partial_tH_t(z,w)} \\
    & = \frac{1}{2}\frac{P_{\D}(g_t(z),\zeta_t)^2+P_\D(g_t(w), \zeta_t)^2}{P_{\D}(g_t(z), \zeta_t)P_\D(g_t(w), \zeta_t)}.
\end{align}
We immediately see that $R_t(z,w)\geq1$ for all $z,w\in D_t$, and plugging this back into \eqref{eq:cos-mon} yields that
\begin{equation}
    -\partial_t C_t^\alpha(z,w) \leq Z^\alpha_t(z,w)X_t(z,w)(\sinh(\alpha^2G(z,w))-\cosh(\alpha^2G(z,w))) \leq 0.
\end{equation}
Note that the value of $\alpha$ has not yet played a role. 

\noindent To show that $-\partial_tS^\alpha_t(z,w)\geq0$, it suffices to show that
\begin{equation}
    \cosh(\alpha^2G_t(z,w))-R_t(z,w)\sinh(\alpha^2G_t(z,w))\geq0.
\end{equation}
Equivalently, noting that $G_t(z,w)=G_\D(g_t(z), g_t(w=)$, it suffices to show that
\begin{equation}\label{eq:sin-mon-2}
    \cosh(\alpha^2G_\D(g_t(z),g_t(w)))-\frac{1}{2}\left(\frac{P_\D(g_t(z), \zeta_t)}{P_\D(g_t(w), \zeta_t)}+\frac{P_\D(g_t(w), \zeta_t)}{P_\D(g_t(z), \zeta_t)}\right)\sinh(\alpha^2G_\D(g_t(z),g_t(w)))\geq0.
\end{equation}
In particular, this is now a statement purely about explicit functions on $\D$. To prove \eqref{eq:sin-mon-2}, we use the hyperbolic metric $\mathfrak d$ on $\D$. First, recall that
\begin{equation}\label{eq:metric-hyper}
    \tanh(\mathfrak d(z,w)/2) = \exp(-G_\D(z,w)).
\end{equation}
Moreover, recall that the Busemann function is given by
\begin{equation}
    B_\zeta(z) := \lim_{s\to\infty} (\mathfrak{d}(\eta(s), z)-s) = -\log(P_\D(z, \zeta)),
\end{equation}
where $\eta$ is the hyperbolic geodesic from the origin to $\zeta\in\partial\D$. In particular,
\begin{equation}\label{eq:busemann}
    B_\zeta(z)-B_\zeta(w) = -\log\left(\frac{P_\D(z,\zeta)}{P_\D(w,\zeta)}\right).
\end{equation}
Substituting \eqref{eq:metric-hyper} and \eqref{eq:busemann} into \eqref{eq:sin-mon-2}, and rearranging, gives the equivalent statement 
\begin{equation}
    \tanh(\mathfrak{d}(z,w)/2)^{2\alpha^2}\geq\frac{\cosh(B_\zeta(z)-B_\zeta(w))+1}{\cosh( B_\zeta(z)-B_\zeta(w))-1}=\tanh^2(B_\zeta(z)-B_\zeta(w)).
\end{equation}
By the triangle inequality, $B_\zeta(z)-B_\zeta(w)\leq\mathfrak{d}(z,w)$. And the claim follows precisely when $\alpha\leq1$, noting that $\tanh(x)\leq1$ for all $x\in\R$. \smallskip
\end{proof}

\subsection{Boundary integrability}
The main goal for this section is to identify how small a local set must be to ensure it is thin for the imaginary chaos. Our approach towards thinness statements will be to consider the limits of dyadic approximations for the relevant quantities. This approach matches the one in e.g. \cite{thin-GFF, notes-GFF}. 

In what follows, for any given closed set $A\subset \D$, we use $[A]_n$ to denote the union of all dyadic squares of level $n$ that intersect $A$. A nice upshot of such approximations is that, as they can take only finitely many values, we can define the random variables $(\phi, f\1_{[A]_n})$ or $(\phi, f\1_{D\setminus[A]_n})$. Indeed, one can simultaneously define $(\phi, f\1_{a})$ for any possible value $a$ of $[A]_n$ and see that 
\begin{equation}\label{eq:def-restr-gff}
    (\phi, f\1_{[A]_n}) = \sum_a(\phi, f\1_{a})\1_{\{[A]_n=a\}} \quad \text{a.s.}
\end{equation}
This is especially useful for \emph{random} closed sets, possibly with Lebesgue measure zero. While this choice of approximation scheme is very practical (especially to derive thinness criteria in terms of the size of $A$), it is somewhat arbitrary. 
We refer to {\cite[Section 5]{thin-GFF}} for further discussion on discrete approximation schemes of the same flavour. For our dyadic approximation scheme, as in \cite{thin-GFF}, we consider the following definition. 

\begin{defn}
    A set $A$ is a thin for the field $\field{e^{i\alpha\phi}}$ if for any continuous, bounded function $f:\C\to\R$, 
    \begin{equation*}
        \lim_{n\to\infty}(\field{e^{i\alpha\phi}}, f\1_{[A]_n})= 0,
    \end{equation*}
    where the limit is in probability. The same definition applies to the fields $\field{\cos(\alpha\phi)}$ and $\field{\sin(\alpha\phi)}$.
\end{defn}

The first step to identify thin (two-valued) local sets for $\field{e^{i\alpha\phi}}$ is to study the integrability of the field along a deterministic fractal boundary. This is why, as already discussed in Section \ref{sec:prelim-cont}, it is crucial for us to be able to test our fields against functions that intersect the boundary of the domain. It will also be necessary when proving convergence of the discrete XOR-Ising decomposition in Section \ref{sec:conv}. We start with a key consequence of the monotonicity in Theorem \ref{thm:domain-mon}.

\begin{prop}[Boundary Mass -- Sine] \label{prop:bdy-sin}
Let $\alpha\in(0,1]$. Let $D$ be a bounded, simply connected domain. Fix any $A\subseteq \partial D$ to be a (deterministic) closed subset of the boundary such that $\Leb(A)=0$. For any continuous, bounded  function $f:\C\to\R$,
\begin{align}
    \lim_{n\to \infty} (\field{\sin(\alpha \phi)}, f\1_{[A]_n}) = 0, 
\end{align}
where the limit is almost sure and in $L^2(\P)$.
In particular, for any $u\in[-\pi/2,\pi/2]$,
\begin{align}\label{eq:correlacion_cos_thin}
    \lim_{n\to \infty} (\field{\barcos(\alpha\phi+u)}, f\1_{[A]_n}) = 0,
\end{align}
where the limit is, again, almost sure and in $L^2(\P)$.
\end{prop}
\begin{proof}
    The first claim follows by Theorem \ref{thm:domain-mon}. Indeed, the two-point function is dominated by its full-plane version, thus removing any boundary issues. In particular, we have the following very useful estimate
    \begin{align}
        \int_{[A]_n}\int_{[A]_n}\langle\field{\sin(\alpha\phi(z))}\field{\sin(\alpha\phi(w))}\rangle dzdw & \leq \int_{[A]_n}\int_{[A]_n} \sinh(\alpha^2 G_\C(z,w))dzdw\\
        & \leq \int_{[A]_n}\int_{[A]_n} |z-w|^{-\alpha^2}dzdw \\ \label{eq:sin-mon-leb}
        & \lesssim \Leb([A]_n)^{2-\alpha^2/2} \longrightarrow0.
    \end{align}
    Finally, \eqref{eq:correlacion_cos_thin} follows from Proposition \ref{prop:funct-ineq}\ref{prop-3}. In both cases, a.s. convergence follows from the summability of the sequence of variances.
\end{proof}

Given the result above, one immediately identifies the expectation of $\field{e^{i\alpha\phi}}$ as the \emph{only} problematic factor when integrating along the boundary. In other words, vanishing expectation implies vanishing second moment for all of our fields. While the next result holds for any $\alpha\in(0, \sqrt{2})$, we prefer to give a succinct proof using a trick that will be exploited extensively in later proofs, yet relies on the domain monotonicity of $\alpha\in(0,1]$.

\begin{prop}\label{prop:bdy-expect}
    Let $\alpha\in(0,1]$. Let $D$ be a bounded, simply connected domain. Fix any $A\subseteq \partial D$ to be a (deterministic) closed subset of the boundary such that $\dim_M(A)<2-\alpha^2/2$. 
    For any continuous, bounded  function $f:\C\to\R$,
    \begin{equation}\label{eq:expec-IGMC}
        \lim_{n\to\infty}\E[(\field{e^{i\alpha\phi}}, f\1_{[A]_n})]=0.
    \end{equation}
    As a consequence,
    \begin{equation}\label{eq:L2-IGMC}
        \lim_{n\to\infty}(\field{e^{i\alpha\phi}}, f\1_{[A]_n})=0,
    \end{equation}
    where the limit is almost sure and in $L^2(\P)$.
\end{prop}
\begin{proof}
    We have that 
    \begin{align}
        \E[(\field{e^{i\alpha\phi}}, f\1_{[A]_n})] = \int_{[A]_n}\CR(z, D)^{-\alpha^2/2}f(z)dz \leq \norm{f}_\infty\int_{[A]_n} d(z, \partial D)^{-\alpha^2/2}dz.
    \end{align}
    Our assumption on the dimension of $A$ now guarantees that this integral vanishes in the limit (i.e. we have ensured that there are not too many boxes to sum over). To be precise, letting $\epsilon_n=2^{-n}$, we see that
    \begin{align}
        \int_{[A]_n} d(z, \partial D)^{-\alpha^2/2}dz &\leq \sum_{k=n}^\infty\int_{[A]_k\setminus[A]_{k+1}}d(z,\partial D)^{-\alpha^2/2}dz\\
        &\lesssim \sum_{k=n}^\infty\int_{[A]_k\setminus[A]_{k+1}}\epsilon_k^{-\alpha^2/2}dz\\
        &\lesssim\sum_{k=n}^\infty\epsilon_k^{-\alpha^2/2}\epsilon_k^{2-\dim_M(A)}\\
        &\lesssim \epsilon_n^{2-\alpha^2/2-\dim_M(A)}\longrightarrow0.
    \end{align}
    To prove \eqref{eq:L2-IGMC}, we add and subtract the expectation to get 
    \begin{align}
        &\E[|(\field{e^{i\alpha\phi}}, f\1_{[A]_n})|^2] = \E[(\field{\barcos(\alpha\phi)}, f\1_{[A]_n})^2] + \E[(\field{\sin(\alpha\phi)}, f\1_{[A]_n})^2]        + \E[(\field{e^{i\alpha\phi}}, f\1_{[A]_n})]^2,
    \end{align}
    thus the claim follows from Proposition \ref{prop:bdy-sin} and \eqref{eq:expec-IGMC}. 
\end{proof}

\begin{rem}
    Almost sure convergence is a consequence of the choice of a dyadic approximation. For general approximation schemes, one can only hope for convergence in probability. See \cite{thin-GFF}.
\end{rem}

Considering only the second moments above is sufficient for the proofs of our continuum statements Theorem \ref{thm:cont-XOR}-\ref{thm:cont-general}. However, we require control of all higher moments, particularly inside fractal domains, to prove the convergence statement Theorem \ref{thm:discrete-conv} in Section \ref{sec:conv}. This is precisely the content of the next two results. The proofs consist simply on revisiting the statements in \cite{IGMC} without the assumption that the test function is compactly supported inside the domain. 

\begin{prop}\label{prop:onsager}
    Let $D$ be a bounded, simply connected domain. For any continuous, bounded  function $f:\C\to\R$,
    \begin{equation}\label{eq:onsager}
        \E[|(\field{e^{i\alpha\phi}}, f)|^{2k}]\leq\norm{f}_\infty^{2k}C^{2k}\int_{D^{2k}}\exp\left( \frac{\alpha^2 }{2 }\sum_{n=1}^{2k}\log(1/r(z_n))\right) dz_1\cdots dz_{2k},
    \end{equation}
    where 
    \begin{equation}\label{eq:min-onsager}
        r(z_n) := \frac{1}{2}(\min_{m\neq n}|z_n-z_m|\wedge d(z_n, \partial D)),
    \end{equation}
    and $C<\infty$ is some constant that only depends on the domain $D$.
\end{prop}
\begin{proof}
    Revisiting {\cite[Proposition 3.6]{IGMC}} without the assumption of compact support, one obtains that for any $k\geq1$, any $q_1\ldots,q_k\in\{\pm1\}$, and any $z_1,\ldots,z_k\in D$,
    \begin{equation}
        -\sum_{1\leq n,m\leq k}q_nq_mG(z_n, z_m) \leq \frac{1}{2}\sum_{n=1}^k\left(\log(1/r(z_n))+\log\CR(z_n, D)\right) + Ck,
    \end{equation}
    for some constant $C<\infty$ that only depends on $D$. A direct Gaussian computation as in {\cite[Theorem 3.12]{IGMC}}, yet minding our different normalization, gives the claim. 
\end{proof}

As explained in \cite{IGMC}, the main purpose of Proposition \ref{prop:onsager} is to argue that the moments of ${(\field{e^{i\alpha\phi}}, f)}$ uniquely characterize its law. To this end, it suffices to show that the right-hand  side of \eqref{eq:onsager} grows slow enough in $k$, which remains true without the compactness assumption on the support, provided the dimension of the boundary meets the usual requirement.

\begin{prop}\label{prop:unique}
    Let $\alpha\in(0, \sqrt{2})$. Let $D$ be a bounded, simply connected domain such that $\dim_M(D)<2-\alpha^2/2$. The moments $$\E[(\field{e^{i\alpha\phi}}, f)^n(\field{e^{-i\alpha\phi}}, f)^m]$$ exist for all $n,m\geq0$ and determine the distribution of the field as a random linear functional acting on continuous, bounded functions in $\C$.
\end{prop}
\begin{proof}
     As in {\cite[Lemma A.2]{IGMC}}, the moments above are bounded by $\E[|(\field{e^{i\alpha\phi}}, f)|^{2N}]$ for $N=n\wedge m$. In the bulk, the necessary bounds thereafter are in {\cite[Lemma 3.10]{IGMC}}. Here, we split the integral \eqref{eq:onsager} into different regions depending on the minimum attained in \eqref{eq:min-onsager}, and argue similarly to the expectation bound in Proposition \ref{prop:bdy-expect}.
\end{proof}

\subsection{Thinness of subcritical TVS}\label{sec:thin-TVS}

Let us now explain how to prove that small enough two-valued sets are thin for $\field{e^{i\alpha\phi}}$. Since we have already identified the ``correct'' assumptions on the dimension of the domain boundary $\partial D$, we only consider \emph{valid} domains, as per the following definition, for the remainder of the paper.

\begin{defn}\label{def:valid}
    A bounded, simply connected domain $D$ is \emph{valid} if
    \begin{enumerate}[label=(\roman*)]
        \item $\dim_M(\partial D)<2-\alpha^2/2$, when considering $\field{e^{i\alpha\phi}}$ or $\cos(\alpha\phi)$;
        \item $\Leb(\partial D)=0$, when considering $\field{\sin(\alpha\phi)}$.
    \end{enumerate}
\end{defn}

The following result is implicit in {\cite[Proposition 3.1]{dim-TVS}}, we write the proof here for completeness. Observe that we must now restrict ourselves to $\alpha\in(0,1)$.

\begin{prop}\label{prop:thin-chaos}
    Let $\alpha\in(0,1)$. For any $a<a_c(\alpha)$, let $\A_{-a,a}$ be a subcritical two-valued set in a valid domain $D$. For any continuous, bounded  function $f:\C\to\R$, 
    \begin{align}
    \lim_{n\to \infty} (\field{e^{i\alpha \phi}}, f\1_{[\A_{-a, a}]_n}) =0, 
    \end{align}
    where the limit is in probability. 
\end{prop}
\begin{proof}
We write $\A\equiv\A_{-a,a}$. It suffices to show that\footnote{For any sequence $X_n\geq0$ of random variables, one can easily check that if $\E[X_n\mid\F]\to0$ in probability for some $\sigma$-algebra $\F$, then $X_n\to0$ in probability.} 
\begin{equation}
    \E[|(\field{e^{i\alpha\phi}}, f\1_{[\A]_n})|^2\mid\F_\A] \longrightarrow0
\end{equation}
in probability as $n\to\infty$. The direct computation in {\cite[Proposition 3.1]{dim-TVS}} yields that
\begin{align}\label{eq:exp-2nd-mom}
    \E[&|(\field{e^{i\alpha\phi}}, g)|^2\mid\F_\A]=\int_{\D\setminus\A}\int_{\D\setminus\A}e^{i\alpha(h_\A(z)-h_\A(w))}\langle\field{e^{i\alpha\phi^\A(z)}}\field{e^{-i\alpha\phi^\A(w)}}\rangle g(z)g(w)dzdw, 
\end{align}
for any continuous, bounded function $g:\C\to\R$. Moreover, this result can easily be seen to hold whenever $g$ is a measurable function of $\F_\A$, and thus we can apply it to $g=f\1_{[\A]_n}$. The claim follows as in Proposition \ref{prop:bdy-expect}, recalling Proposition \ref{prop:TVS-one-point}.
\end{proof}

\begin{rem}
An alternative proof of thinness can be obtained by following the same argument as the one given in \cite{thin-GFF} for the GFF. Indeed, it suffices to control the size of the field in each dyadic box, which can be done using {\cite[Theorem 1.4]{IGMC}} to obtain exponential tails for the mass in each box. See
{\cite[Corollary A.2]{plasma}} and {\cite[Appendix B]{noise-IGMC}} for the derivation of such tails in the full plane case.  
\end{rem}

\begin{rem}
    Any deterministic closed set $A\subseteq D$ with $\Leb(A)=0$ is a.s. thin for $\field{e^{i\alpha \phi}}$. Indeed, this follows from the same argument in the proof of Proposition \ref{prop:bdy-expect}, using the comparison in Proposition \ref{prop:funct-ineq}\ref{prop-2} and that the conformal radius is bounded away from zero on $A$. \end{rem}
    
\subsection{Strong thinness for $\A_{-2\lambda,2\lambda}$}\label{ss:strong_thinnes}
We conclude the section by showing that, even for $\alpha\geq 1/2$, the two-valued set $\A_{-2\lambda, 2\lambda}$ is thin for $\field{\sin(\alpha\phi)}$. In fact, we prove a stronger version of thinness as stated in property (\textbf{III}) of Section \ref{sec:decomp}. Namely, that the convergence of the infinite sums holds in $L^2(\P)$. The following simple observation should be seen as the key motivation to look for such a statement. 

\begin{prop}\label{prop:sign-permute}
    Let $(\xi_k)_{k\geq1}$ be a sequence of i.i.d. symmetric signs. Let $(X_k)_{k\geq1}$ be any sequence of random variables independent of $(\xi_k)_{k\geq 1}$ such that 
    \begin{equation}\label{eq:L2-bound-std}
        \sum_{k=1}^\infty \E[X_k^2] <\infty.
    \end{equation}
    Then, 
    \begin{equation}\label{eq:L2-conv-std}
        S = \lim_{N\to\infty}\sum_{k=1}^N\xi_kX_k
    \end{equation}
    exists as a limit in $L^2(\P)$, and also
    \begin{equation}\label{eq:L2-conv-order}
        S = \lim_{N\to\infty}\sum_{k=1}^N\xi_{\pi(k)}X_{\pi(k)}
    \end{equation}
    for any (random) permutation $\pi:\N\to\N$ independent of the signs $(\xi_k)_{k\geq1}$, that may possibly depend on the values of $(X_k)_{k\geq1}$.
\end{prop}
\begin{proof}
    The convergence in \eqref{eq:L2-conv-std} is elementary. Indeed, denoting by $S_N$ the partial sums,
    \begin{equation}
        \E[(S_N-S_M)^2] = \sum_{k=N+1}^M\E[X_k^2]
    \end{equation}
    for any $M\geq N$, which converges to zero as $M\to\infty$ and then $N\to \infty$ by \eqref{eq:L2-bound-std}. Writing $S_N^\pi$ for the partial sum under the $\pi$-rearrangement,
    \begin{equation}
        S_N-S_N^\pi = \sum_{k=1}^N\xi_kX_k\1_{\{k\not\in\{\pi(1),\ldots\pi(N)\}\}} - \sum_{k=N+1}^\infty \xi_kX_k\1_{\{k\in\{\pi(1),\ldots\pi(N)\}\}}, 
    \end{equation}
   Crucially using the independence between $\pi$ and $\xi$, it follows that
    \begin{equation}
        \E[(S_N-S_N^\pi)^2] = \sum_{k=1}^\infty\E[X_k^2\1_{\{k\in\{1,\ldots,N\}\Delta\{\pi(1), \ldots, \pi(N)\}}],
    \end{equation}
    and we conclude the proof invoking  \eqref{eq:L2-bound-std} once more.
\end{proof}

The following result is the key input in the proofs of Theorems \ref{thm:cont-XOR}-\ref{thm:cont-general}. This strong notion of thinness is already non-trivial for the easier case $\alpha\in(0,1/2)$, where we have the standard thinness of Proposition \ref{prop:thin-chaos}. We stress once more that convergence in $L^2(\P)$ is absolutely crucial for us to freely rearrange the sums, and as such the proof of the upcoming result must deal carefully with different types of convergence. For $\alpha\in[1/2, 1)$, we remark once more that our approach is based on adapting a new technique developed in \cite{CGS}.

\begin{lemma}\label{lem:CLE4-thin}
    Let $\alpha\in(0, 1)$. Consider $\A_{-2\lambda, 2\lambda}$ in a valid domain $D$ and index its loops by decreasing size of diameter. For any continuous, bounded  function $f:\C\to\R$,
    \begin{equation}\label{eq:CLE4-thin}
        (\field{\sin(\alpha\phi)}, f) = \lim_{N\to\infty}\sum_{k=1}^N\xi_k(\field{\cos(\alpha(\phi_k^\A+\xi_k(\pi-a_c)))}, f),
    \end{equation}
    where the limit is in $L^2(\P)$,  $(\xi_k)_{k\geq1}$ are the i.i.d. labels of the loops, and $(\phi^\A_k)_{k\geq1}$ are independent GFFs inside each loop.
\end{lemma}

\begin{proof} We write $\A\equiv \A_{-2\lambda, 2\lambda}$.

\noindent \textbf{Case 1 -- $\alpha\in(0,1/2)$:} On sampling $\A$, we can consider the split
\begin{equation}\label{eq:main-split}
     (\field{\sin(\alpha\phi)}, f) =  (\field{\sin(\alpha\phi)}, f\1_{[\A]_n}) +  (\field{\sin(\alpha\phi)}, f\1_{[\A]_n^c}).
\end{equation}
By (the proof of) Proposition \ref{prop:thin-chaos}, the first term is such that
\begin{equation}\label{eq:first-term-zero}
    \E[ (\field{\sin(\alpha\phi)}, f\1_{[\A]_n})^2\mid\A]\longrightarrow0
\end{equation}
in probability as $n\to\infty$. We focus then on the second term in \eqref{eq:main-split}. For each $n\geq1$, we are far enough from the boundary of $\A$ so that
\begin{align*}
     (\field{\sin(\alpha\phi)}, f\1_{[\A]_n^c}) & = (\field{\cos(\alpha(\phi^\A+h_\A-a_c))}, f\1_{[\A]_n^c})\\
    & = (\field{\barcos(\alpha(\phi^\A+h_\A-a_c))}, f\1_{[\A]_n^c}) + (\langle\field{\cos(\alpha(\phi^\A+h_\A-a_c))}\rangle, f\1_{[\A]_n^c}).
\end{align*}
Observe both terms on the right-hand side can be written as sums over macroscopic loops of $\A$, e.g.
\begin{align}
    (\field{\cos(\alpha(\phi^\A+h_\A-a_c))}, f\1_{[\A]_n^c}) & = \sum_{k=1}^N(\field{\cos(\alpha(\phi^\A_k+\xi_k\pi-a_c))}, f\1_{[\A]_n^c}) \\ \label{eq:expand-sum-labels}
    & = \sum_{k=1}^N\xi_k(\field{\cos(\alpha(\phi^\A_k+\xi_k(\pi-a_c)))}, f\1_{[\A]_n^c}),
\end{align}
where $N=N(n)$ is some cuttoff such that $N\to\infty$ almost surely as $n\to\infty$. By Proposition \ref{prop:bdy-sin}, and in particular the bound \eqref{eq:sin-mon-leb}, we have that
\begin{equation}
    (\field{\barcos(\alpha(\phi^\A+h_\A-a_c))}, f\1_{[\A]_n^c}) \longrightarrow (\field{\barcos(\alpha(\phi^\A+h_\A-a_c))}, f)
\end{equation}
in $L^2(\P)$ as $n\to\infty$. Combining this with \eqref{eq:first-term-zero}, it follows that
\begin{align}\label{eq:key-thin-proof}
    \E[(\langle\field{\cos(\alpha(\phi^\A+h_\A-a_c))}\rangle, f\1_{[\A]_n^c})^2\mid\A] \longrightarrow\ & \E[ (\field{\sin(\alpha\phi)}, f)^2\mid\A] \\
    &\ - \E[(\field{\barcos(\alpha(\phi^\A+h_\A-a_c))}, f)^2\mid \A]
\end{align}
in probability as $n\to\infty$. The key observation is that, due to the orthogonality of the signs and the fact that we are only conditioning on $\A$, the left-hand side above can be written as 
\begin{equation}
    \E[(\langle\field{\cos(\alpha(\phi^\A+h_\A-a_c))}\rangle, f\1_{[\A]_n^c})^2\mid\A] = \sum_{k=1}^N(\langle\field{\cos(\alpha(\phi_k^\A+\xi_k(\pi-a_c)))}\rangle, f\1_{[\A]_n^c})^2,
\end{equation}
which is non-negative and increasing in $n$. Namely, it has an a.s. limit given by 
\begin{equation}
    \sum_{k=1}^\infty(\langle\field{\cos(\alpha(\phi_k^\A+\xi_k(\pi-a_c)))}\rangle, f)^2,
\end{equation}
which by \eqref{eq:key-thin-proof} must have 
\begin{equation}
    \sum_{k=1}^\infty\E[(\langle\field{\cos(\alpha(\phi_k^\A+\xi_k(\pi-a_c)))}\rangle, f)^2] 
    < \infty.
\end{equation}
By Proposition \ref{prop:sign-permute}, the sum
\begin{equation}
    \sum_{k=1}^\infty\xi_k(\langle\field{\cos(\alpha(\phi_k^\A+\xi_k(\pi-a_c)))}\rangle, f)
\end{equation}
exists as a limit in $L^2(\P)$ under any ordering independent of the signs $(\xi_k)_{k\geq1}$. Combining this with the expressions of the form \eqref{eq:expand-sum-labels}, the claimed convergence in \eqref{eq:CLE4-thin} holds.

\noindent \textbf{Case 2 -- $\alpha\in[1/2, 1)$:} Since $\A$ is no longer thin, we must choose an appropriate approximation. Fix some $\eps>0$ and let $\A^\eps$ be the closure of the union of all the loops of $\A$ that intersect the grid $\eps\Z^2$. By e.g. {\cite[Proposition 12]{excursion-GFF}}, this is a local set where the harmonic function $h_{\A^\eps}$ is equal to $\pm2\lambda$ in the interior of any loop intersecting $\eps\Z^2$ and equal to zero in the complement of the closure of the union of such interiors, which we denote by $\mathbb{O}^\eps$. Moreover, all loops with diameter bigger than $\sqrt{2}\eps$ have to appear in $\A^\eps$. See Figure \ref{fig:grid}.

\begin{figure}[h]
    \centering
    \includegraphics[width=0.4\linewidth]{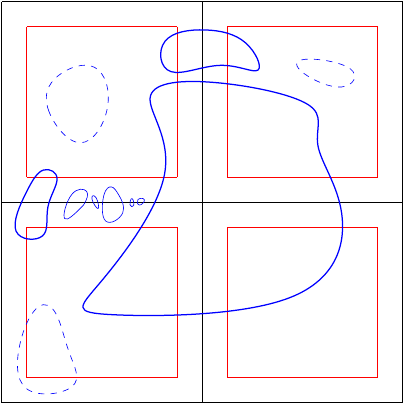}
    \caption{The $\eps$-grid in black, where the interior of each red square marks the region in which $f_\delta$ is supported. The loops of $\A_{-2\lambda, 2\lambda}$ are drawn in blue. The dashed loops are those \emph{not} in $\A_{-2\lambda, 2\lambda}^\eps$, thus they all have diameter smaller than $\sqrt{2}\epsilon$. The solid loops can have any diameter, but those intersecting the support of $f_\delta$, marked in bold, must have diameter greater than $\delta$.}
    \label{fig:grid}
\end{figure}

\noindent Let $(f_\delta)_{\delta>0}$ be the sequence of continuous functions, compactly supported in $D\setminus \eps\Z^2$, given by
\begin{equation}
    f_\delta(z):=f(z)\1_{\{d(z, \eps\Z^2)\geq \delta\}}. 
\end{equation}
By Proposition \ref{prop:bdy-sin}, a direct computations gives that
\begin{equation}\label{eq:delta-sin}
    (\field{\sin(\alpha\phi)}, f-f_\delta) \longrightarrow 0
\end{equation}
in $L^2(\P)$ as $\de\to0$, and similarly
\begin{equation}\label{eq:delta-barcos}
    (\field{\barcos(\alpha\phi^{\A^\eps})}, f-f_\delta) \longrightarrow 0.
\end{equation}

\noindent On sampling $\A^\eps$, just as in the previous case, we consider the split
\begin{equation}\label{eq:split-2-case}
    (\field{\sin(\alpha\phi)}, f_\delta) = (\field{\sin(\alpha\phi)}, f_\delta\1_{[\A^\eps]_n}) + (\field{\sin(\alpha\phi)}, f_\delta\1_{[\A^\eps]_n^c}).
\end{equation}
The key upshot of this new approximation that
\begin{equation}\label{eq:A_epsilon_thin}
    \E[(\field{\sin(\alpha\phi)}, f_\delta\1_{[\A^\eps]_n})^2\mid \A]\longrightarrow0
\end{equation}
in probability as $n\to\infty$, i.e. the approximation is thin. To see this, one argues as in the proof of Proposition \ref{prop:thin-chaos}, effectively reducing the claim to proving that
\begin{equation}
    \int_{[\A^\eps]_n} \CR(z, D\setminus\A^\eps)^{-\alpha^2/2}f_\de(z)dz\longrightarrow 0
\end{equation}
in probability as $n\to\infty$, with $\eps>\de>0$ fixed. The dyadic boxes tiling $\A^\eps$ and intersecting the support of $f_\delta$ cover only portions of loops all of which have diameter greater than $\de>0$. By the local finiteness of $\A$, this is a finite collection of loops whose Minkowski dimension is only $3/2$, and thus thinness follows as now standard.

\noindent As in the case $\alpha\in(0,1/2)$, but recalling that the harmonic function of $\A^\eps$ is now three-valued, we continue by expanding the second term in \eqref{eq:split-2-case} as
\begin{align}
    (\field{\sin(\alpha\phi)}, f_\delta\1_{[\A^\eps]_n^c}) & = (\field{\sin(\alpha\phi^{\mathbb{O}^\eps})}, f_\delta\1_{[\A^\eps]_n^c})\\ 
    & \quad + \sum_{k=1}^N\1_{\{\ell_k\in\A^\eps\}}\xi_k(\field{\barcos(\alpha(\phi^{\A^\eps}_k+\xi_k(\pi-a_c)))},f_\delta\1_{[\A^\eps]_n^c}) \\
    & \quad + \sum_{k=1}^N\1_{\{\ell_k\in\A^\eps\}}\xi_k(\langle\field{\cos(\alpha(\phi^{\A^\eps}_k+\xi_k(\pi-a_c)))}\rangle,f_\delta\1_{[\A^\eps]_n^c}),
\end{align}
where $\ell_k\in\A^\eps$ means that the $k$th largest loop of $\A$ intersects the $\eps$-grid. As before, we can use \eqref{eq:sin-mon-leb} to handle the first two terms in $L^2(\P)$. Moreover, this control can be combined with \eqref{eq:A_epsilon_thin} to see that
\begin{align}
    \E\bigg[\bigg(\sum_{k=1}^\infty&\1_{\{\ell_k\in\A^\eps\}}\xi_k(\langle\field{\cos(\alpha(\phi^{\A^\eps}_k+\xi_k(\pi-a_c)))}\rangle,f_\delta\1_{[\A^\eps]_n^c})\bigg)^2\Big\lvert \A\bigg] \\ 
    &\longrightarrow \E\left[\field{\sin(\alpha\phi)}, f_\delta)^2\mid\A\right]
    - \E\left[(\field{\sin(\alpha\phi^{\mathbb{O}^\eps})}, f_\delta)^2\mid\A\right] \\
    &\hspace{6mm} - \E\left[\left(\sum_{k=1}^N\1_{\{\ell_k\in\A^\eps\}}\xi_k(\field{\barcos(\alpha(\phi^{\A^\eps}_k+\xi_k(\pi-a_c)))},f_\delta)\right)^2\Big\lvert\A\right].
\end{align}
This is the analogue of \eqref{eq:key-thin-proof} in the previous case, which readily implies that 
\begin{align}
     (\field{\sin(\alpha\phi)}, f_\delta) & = (\field{\sin(\alpha\phi^{\mathbb{O}^\eps})}, f_\delta)\\ 
    & \quad + \sum_{k=1}^\infty\1_{\{\ell_k\in\A^\eps\}}\xi_k(\field{\barcos(\alpha(\phi^{\A^\eps}_k+\xi_k(\pi-a_c)))},f_\delta) \\ \label{eq:sum-with-delta}
    & \quad + \sum_{k=1}^\infty\1_{\{\ell_k\in\A^\eps\}}\xi_k(\langle\field{\cos(\alpha(\phi^{\A^\eps}_k+\xi_k(\pi-a_c)))}\rangle,f_\delta),
\end{align}
where both sums converge in $L^2(\P)$.

\noindent By the convergence in \eqref{eq:delta-sin}-\eqref{eq:delta-barcos}, we can take $\delta\to0$ and obtain that
\begin{equation}
    \sum_{k=1}^\infty\E[\1_{\{\ell_k\in\A^\eps\}}(\langle\field{\cos(\alpha(\phi^{\A^\eps}_k+\xi_k(\pi-a_c)))}\rangle,f) ^2]<\infty,
\end{equation}
which in turn implies that
\begin{equation}
    \sum_{k=1}^\infty\1_{\{\ell_k\in\A^\eps\}}\xi_k(\langle\field{\cos(\alpha(\phi^{\A^\eps}_k+\xi_k(\pi-a_c)))}\rangle,f_\delta) \longrightarrow \sum_{k=1}^\infty\1_{\{\ell_k\in\A^\eps\}}\xi_k(\langle\field{\cos(\alpha(\phi^{\A^\eps}_k+\xi_k(\pi-a_c)))}\rangle,f) 
\end{equation}
in $L^2(\P)$ as $\de\to0$, noting that each term on the left-hand side is monotone in $\delta$ and converges to the corresponding term on the right-hand side. Altogether, we have shown that
\begin{align}
     (\field{\sin(\alpha\phi)}, f) & = (\field{\sin(\alpha\phi^{\mathbb{O}^\eps})}, f)\\ 
    & \quad + \sum_{k=1}^\infty\1_{\{\ell_k\in\A^\eps\}}\xi_k(\field{\barcos(\alpha(\phi^{\A^\eps}_k+\xi_k(\pi-a_c)))},f) \\
    & \quad + \sum_{k=1}^\infty\1_{\{\ell_k\in\A^\eps\}}\xi_k(\langle\field{\cos(\alpha(\phi^{\A^\eps}_k+\xi_k(\pi-a_c)))}\rangle,f),
\end{align}
and both sums converge in $L^2(\P)$.

\noindent The proof is completed following the very same argument as $\eps\to0$, but now using that
\begin{align}
    \sum_{k=1}^\infty\1_{\{\ell_k\in\A^\eps\}}\xi_k(\field{\barcos(\alpha(\phi^{\A^\eps}_k+\xi_k(\pi-a_c)))},f) \longrightarrow \sum_{k=1}^\infty\xi_k(\field{\barcos(\alpha(\phi^{\A}_k+\xi_k(\pi-a_c)))},f)
\end{align}
and that
\begin{align}
     (\field{\sin(\alpha\phi^{\mathbb{O}^\eps})}, f)\longrightarrow 0
\end{align}
both in $L^2(\P)$, which follow as usual from the bound \eqref{eq:sin-mon-leb}. 
\end{proof}

\begin{rem}
    While the orthogonality of the labels of $\A_{-2\lambda, 2\lambda}$ has been used extensively, we have also crucially relied on the fact that this is the \emph{only} two-valued set for which the boundary condition outside the loops is zero. Indeed, this makes it so that we only see the \emph{pure} sine $\field{\sin(\alpha\phi^{\mathbb{O}^\eps)}}$ outside, and we can handle this in $L^2(\P)$ with no complications. Interestingly, this boundary condition feature is precisely the reason why the labels are i.i.d., as explained in \cite{TVS}.
\end{rem}


\section{Critical TVS: Measures} \label{sec:cos-TVS}

The purpose of this section is to prove the remaining property (\textbf{II}) from Section \ref{sec:decomp}. The proof is almost immediate given Proposition \ref{prop:bdy-sin}. One can understand the following result as ``taking the conditional expectation out'' in the  measures of Proposition \ref{prop:TVS-meas}. 

\begin{prop}\label{prop:cos-TVS}
Let $\alpha\in(0,1)$. For $a_c=a_c(\alpha)=\lambda/\alpha$, let $\A_{-a_c, a_c}$ be the critical two-valued set in a valid domain $D$. For any continuous, bounded  function $f:\C\to\R$,
\begin{align}\label{eq:cos-TVS}
    \lim_{n\to \infty}(\field{\cos( \alpha \phi )}, f\1_{[\A_{-a_c, a_c}]_n}) = \E\left[ (\field{\cos( \alpha \phi )}, f) \mid \F_{\A_{-a_c, a_c}} \right ], 
\end{align}
where the limit is in probability. In particular, the left-hand side converges to $(\mu, f)$ where $\mu\equiv\mu(a_c)\equiv\mu(\alpha)$ is a random measure which is measurable with respect to $\A_{-a_c, a_c}$.
\end{prop}
\begin{proof} The second part of the claim follows from\footnote{Note, however, that the measures differ by a factor of $a_c$.} Proposition \ref{prop:TVS-meas}. We write $\A\equiv\A_{-a_c, a_c}$. As before, it suffices to show that
    \begin{equation}\label{eq:cos-TVS-goal}
        \E\left[\left((\field{\cos( \alpha \phi )}, f\1_{[\A]_n})-\E\left[ (\field{\cos( \alpha \phi )}, f) \mid \F_\A \right ]\right)^2\mid\F_\A\right]\longrightarrow 0
    \end{equation}
    in probability. For each $n\geq1$, by the symmetry of the renewed sine, we have that
    \begin{equation}
        \E\left[ (\field{\cos( \alpha \phi )}, f\1_{[\A]_n}) \mid \F_\A \right ]=\E\left[ (\field{\cos( \alpha \phi )}, f) \mid \F_\A \right ] \quad\text{a.s.}
    \end{equation}
    The claim follows from Proposition \ref{prop:bdy-sin} (cf. the $\field{\barcos}$  trick in Proposition \ref{prop:bdy-expect} and Lemma \ref{lem:CLE4-thin}).
\end{proof}

\begin{rem}\label{rem:crit-second-mom}
    Using very similar arguments, one can show that
    \begin{equation}\label{eq:crit-second-mom}
        \E[(\field{\cos(\alpha\phi)}, f)^2 \mid \F_{\A_{-a_c, a_c}}] = (\mu, f)^2 + \E[(\field{\sin(\alpha\phi^\A)}, f)^2 \mid \F_{\A_{-a_c, a_c}}].
    \end{equation}
    Indeed, the usual trick yields
    \begin{equation}
        \E[(\field{\cos(\alpha\phi)}, f)^2 \mid \F_\A] = (\mu, f)^2 + \E[\big((\field{\cos(\alpha\phi)}, f)-\E[(\field{\cos(\alpha\phi)}, f)\mid\F_\A]\big)^2\mid \F_\A].
    \end{equation}
    The claim follows noting that
    \begin{equation*}
        (\field{\cos(\alpha\phi)}, f\1_{[\A]_n^c})-\E[(\field{\cos(\alpha\phi)}, f\1_{[\A]_n^c})\mid\F_\A] = (\field{\sin(\alpha\phi^\A)}, f\1_{[\A]_n^c}) \longrightarrow(\field{\sin(\alpha\phi^\A)}, f)
    \end{equation*}
    in $L^2(\P)$ and that 
    \begin{equation*}
        (\field{\cos(\alpha\phi)}, f\1_{[\A]_n})-\E[(\field{\cos(\alpha\phi)}, f\1_{[\A]_n})\mid\F_\A] \longrightarrow0
    \end{equation*}
    also in  $L^2(\P)$.
\end{rem}


\section{Proof of Theorems \ref{thm:cont-XOR}-\ref{thm:cont-general}} \label{sec:main-proofs}

\subsection{Local finiteness of the decomposition}
Before proving the main theorems, we handle one leftover result we need in the proofs. Namely, we must show that the iterative steps described in Section \ref{sec:decomp} define a locally finite collection of loops. Note that, while each two-valued set is locally finite, \emph{a priori} the successive sampling of them need not be.

\begin{prop}    
\label{prop:local-finite-tree}
    Consider the iteration of two-valued sets given in Section \ref{sec:decomp}. For any $\rho>0$, the number of loops with diameter greater than $\rho$ is finite almost surely. In particular, the iteration stops after an almost surely finite number $N\equiv N(\rho)$ of steps.
\end{prop}

The key observation to prove this result is that in each iterative sampling of $\A_{-2\lambda, 2\lambda}=\CLE_4$ there is a positive chance that all loops are microscopic, i.e. have diameter smaller than $\rho$.

\begin{lemma}\label{lem:uniform_Bound_diam}
    For any $\rho>0$ there exists $c>0$ such that
    \begin{align*}
\inf_D\P(\text{every loop of } \CLE_4(D) \text{ has diameter less than or equal to }  \rho) \geq c,    \end{align*}
where the infimum is taken over all open sets $D\subseteq \D$ such that $D$ is a countable union of pairwise disjoint simply connected domains. 
\end{lemma}

\noindent To prove this statement, we assume the reader is familiar with the loop soup construction of CLE$_4$ in \cite{ShW}. In fact, we only need these four properties:
\begin{enumerate}
    \item[$\bullet$]\textbf{Monotonicity:} If $D\subseteq D'$ are two domains as in Lemma \ref{lem:uniform_Bound_diam}, then CLE$_4(D')$ and CLE$_4(D)$ can be coupled such that every loop of CLE$_4(D)$ is surrounded by a loop of CLE$_4(D')$. 
    \item[$\bullet$] \textbf{Restriction property:} Let $D\subseteq D'$. Define $D^*$ to be the (random) domain obtained by removing from $D$ all the loops of CLE$_4(D')$ not contained in $D$ together with their interiors. Then, conditionally on $D^*$, the law of all the loops of CLE$_4(D')$ contained $D$ is that of CLE$_4(D^*)$. See Figure \ref{fig:local-finite}.
    \item[$\bullet$] \textbf{No loop hits the boundary:} No loop of CLE$_4(D)$ intersects $\partial D$.
    \item[$\bullet$] \textbf{Scale invariance:} For any $\lambda>0$ and $D\subseteq \C$, the law of CLE$_4(\lambda D)$ is that of $\lambda$CLE$_4(D)$. 
\end{enumerate}
We highlight that the uniformity in the domain claimed in Lemma \ref{lem:uniform_Bound_diam} is obtained as a direct consequence of the monotonicity property. This uniformity is crucial in the proof of Proposition~\ref{prop:local-finite-tree}.

\begin{proof}[Proof of Lemma \ref{lem:uniform_Bound_diam}]
By monotonicity, it suffices to show that for any $\rho>0$ there exists $c>0$ such that
\begin{align}\label{eq:pos-prob}
    \P(\text{every loop of } \CLE_4(\D) \text{ has diameter less than or equal to }\rho) \geq c.
\end{align}
Define
\begin{align*}
B:=\{\rho>0: \P(\text{every loop of } \CLE_4(\D) \text{ has diameter less than or equal to } \rho)=0\}.
\end{align*}
We want to show that $B=\emptyset$. Note that $B$ must be an interval, as if $\rho\in B$ then $\rho'\in B$ for all $\rho'<\rho$. Since we know that $1\notin B$, it is enough to show that $B$ is both open and closed in $(0,1]$.

\noindent We first prove that it is closed. 
Since CLE$_4(\D)$ is locally finite, if $\rho_n\nearrow \rho$ then, up to an event of zero probability,
\begin{align*}&
\{\text{every loop of } \CLE_4(\D) \text{ has diameter less than or equal to } \rho_n\} \\
&\hspace{0.2\textwidth}\nearrow  \{\text{every loop of } \CLE_4(\D) \text{ has diameter less than or equal to } \rho\},
\end{align*}
which yields the claim by the monotone convergence theorem.

\noindent To prove that $B$ is open, assume that $\rho \in B$. We need to find $r>0$ such that $\rho+r \in B$. 
By local finiteness and the fact that no loop of $\mathrm{CLE}_4(\mathbb D)$
intersects the boundary, there exists $\delta>0$ such that with positive probability all loops
of $\mathrm{CLE}_4(\mathbb D)$ with diameter larger than $\rho$ are contained in $(1-\delta)\D$. See Figure \ref{fig:local-finite}.

\noindent We now use the restriction property with $D = (1-\delta) \D $ and $D' = \mathbb D$.
The fact that $\rho \in B$ implies that, with probability one, $\mathrm{CLE}_4(D^*)$ has a loop
of diameter greater than or equal to $\rho$. This holds almost surely for all domains $D^*$ arising
from realizations in which no loop of diameter larger than $\rho$ intersects
$\D\setminus (1-\delta)\D$. Since $D^* \subseteq (1-\delta)\D$, monotonicity and the fact that
$(1-\delta)^{-1}\mathrm{CLE}_4((1-\delta)\D)$ is equal in law to $\mathrm{CLE}_4(\mathbb D)$
imply that $\rho/(1-\delta) \in B$, from which the claim follows since $\rho>0$.
\end{proof}

\begin{figure}
    \centering 
    \includegraphics[width=0.4\linewidth]{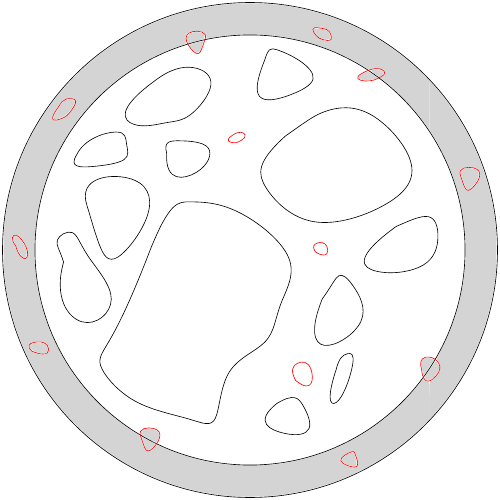}
    \caption{Diagram of the restriction property used in the proof of Lemma \ref{lem:uniform_Bound_diam}. The unshaded region corresponds to a realization of $D^*$ when $D=(1-\de)\D$ and $D'=\D$. The red loops correspond to loops of diameter smaller than $\rho$. The collection of loops satisfies the (positive probability) event that all loops of $\CLE_4(\D)$ with diameter greater than $\rho$ are contained in $(1-\de)\D$.}
    \label{fig:local-finite}
\end{figure}

\begin{proof}[Proof of Proposition \ref{prop:local-finite-tree}] We work in the setup of the algorithm described Section \ref{sec:sin1}.  The result directly follows from these two observations:
\begin{enumerate}
    \item There exists $c>0$ such that for every $n\in \N$, given that the algorithm is still iterating for some $A_{n-1}$, the probability that it stops iterating at level $n$ is at least $c$.
    \item For every $n\in \N$, a.s. the number of excursions discovered up to iteration $n$ of diameter bigger than or equal to $\rho$ is finite.
\end{enumerate}
The first point is a direct consequence of Lemma \ref{lem:uniform_Bound_diam}. The second point directly follows from the local finiteness of $\A_{-2\lambda,2\lambda}$ and $\A_{-2\lambda,2a_c-2\lambda}$.
\end{proof}

\begin{rem}
    For any $\kappa\in(8/3, 4]$, the argument above gives a new proof of the local finiteness of the \emph{nested} $\CLE_\kappa$, as found in {\cite[Lemma 2.1]{aru-papon-powell}}. Indeed, by the standard monotonicity properties of $\CLE_\kappa$ coming from the Brownian loop-soup interpretation, it is clear that \eqref{eq:pos-prob} also holds, and thus the proof remains unchanged.
\end{rem}

\subsection{ Proof of Theorem \ref{thm:cont-XOR} and Theorem \ref{thm:cont-general}} 

We are finally ready to prove our main results in the continuum. We first deal with the proof of the decomposition for $\field{\sin(\alpha\phi)}$ as this is the more general case. Fundamentally, up to iterations, the proof consists in invoking Lemma \ref{lem:CLE4-thin} and then taking the conditional expectation given the collection of critical two-valued sets within each loop. Recall the notion of valid domain from Definition \ref{def:valid}.

\begin{proof}[Proof of Theorem \ref{thm:cont-general} for $\field{\sin (\alpha \phi)}$]
    Fix a valid domain $D\subset\C$ and a continuous, bounded  function $f:\C\to\R$. We begin by sampling $A^{(1)}:=\A_{-2\lambda, 2\lambda}$. Recall that, for concreteness, we choose to index the collection $(\ell_k^{(1)})_{k\geq1}$ of loops of $A^{(1)}$ by decreasing order of diameter, and we denote the corresponding i.i.d. labels by $(\xi_k^{(1)})_{k\geq1}$. By Lemma \ref{lem:CLE4-thin}, the action $A^{(1)}\acts\ \field{\sin(\alpha\phi)}$ gives that
    \begin{equation}\label{eq:main-proof-start}
        (\field{\sin(\alpha\phi)}, f) = \sum_{k=1}^\infty\xi_k^{(1)}(\field{\cos(\alpha(\phi_k^{A^{(1)}}+\xi_k^{(1)}(\pi-a_c)))}, f),
    \end{equation}
    where the sum converges in $L^2(\P)$. For each $k\geq1$, we consider the actions
    \begin{equation}\label{eq:second-action}
         \begin{cases}
             \A_{-2\lambda, 2a_c-2\lambda}\acts\ \field{\cos(\alpha(\phi_k^{A^{(1)}}+\pi-a_c))} \quad\ \text{if}\ \ \xi_k^{(1)}=+1, \\
             \A_{-(2a_c-2\lambda), 2\lambda}\acts\ \field{\cos(\alpha(\phi_k^{A^{(1)}}-\pi+a_c))}\quad\ \text{if}\  \xi_k^{(1)}=-1,
         \end{cases}
    \end{equation}
    and denote the corresponding two-valued set by $A_k^{(2)}$. The now standard bound \eqref{eq:sin-mon-leb} yields that
    \begin{equation}
        (\field{\sin(\alpha\phi^{A^{(2)}_k})}, f) = \sum_{j=1}^\infty c^{(2)}_{k,j}(\field{\sin(\alpha\phi^{A^{(2)}_k}_j)}, f)
    \end{equation}
    is a well-defined limit in $L^2(\P)$, where the index $j\geq1$ runs over the loops of $A^{(2)}_k$ and $c^{(2)}_{k,j}$ is the label of each such loop. In fact, the sum
    \begin{equation}\label{eq:sin-A2}
        (\field{\sin(\alpha\phi^{A^{(2)}})}, f) = \sum_{k=1}^\infty\xi_k^{(1)}(\field{\sin(\alpha\phi^{A^{(2)}_k})}, f)
    \end{equation}
    is also well-defined as a limit in $L^2(\P)$, where $A^{(2)}:= (\cup_{k\geq1}A^{(2)}_k)\cup A^{(1)}$. By Proposition \ref{prop:cos-TVS}, the action \eqref{eq:second-action} is given by
    \begin{align}
        (\field{\cos(\alpha(\phi_k^{A^{(1)}}+\xi_k^{(1)}(\pi-a_c)))}, f)  = (\mu_k^{(1)}, f) + (\field{\sin(\alpha\phi^{A^{(2)}_k})}, f),
    \end{align}
    where the measure $\mu_k^{(1)}$ is supported on $A_k^{(2)}$. Arguing as in Remark \ref{rem:crit-second-mom}, we also see that 
     \begin{equation*}
         \E[(\field{\cos(\alpha(\phi_k^{A^{(1)}}+\xi_k^{(1)}(\pi-a_c)))}, f)^2\mid A^{(1)}, A^{(2)}_k] = (\mu_k^{(1)}, f)^2 + \E[(\field{\sin(\alpha\phi^{A_k^{(2)}})}, f)^2\mid A^{(1)}, A^{(2)}_k].
     \end{equation*} 
     Using the $L^2(\P)$-convergence in \eqref{eq:main-proof-start}, it is immediate that
    \begin{equation}\label{eq:main-sin}
        \E[(\field{\sin(\alpha\phi)}, f)^2 \mid A^{(1)}, A^{(2)}_k] = \sum_{k=1}^\infty (\mu_k^{(1)}, f)^2 + \sum_{k=1}^\infty \E[(\field{\sin(\alpha\phi^{A_k^{(2)}})}, f)^2\mid A^{(1)}, A^{(2)}_k],
    \end{equation}
    where the sums now converge in $L^1(\P)$. In particular, taking expectations and recalling \eqref{eq:sin-A2} ,
    \begin{equation}\label{eq:iterate-bound}
        \sum_{k=1}^\infty\E[(\mu_k^{(1)}, f)^2] = \E[(\field{\sin(\alpha\phi)}, f)^2] -  \E[(\field{\sin(\alpha\phi^{A^{(2)}})}, f)^2].
    \end{equation}
    By Proposition \ref{prop:sign-permute}, putting everything together,
    \begin{equation}\label{eq:key}
         (\field{\sin(\alpha\phi)}, f) = \sum_{k=1}^\infty\xi_k^{(1)}(\mu_k^{(1)}, f) + \sum_{k=1}^\infty\xi_k^{(1)}\left(\sum_{j=1}^\infty c_{k,j}^{(2)}(\field{\sin(\alpha\phi^{A^{(2)}_k}_j)}, f)\right),
    \end{equation}
    where convergence is in $L^2(\P)$ for both sums. Moreover, note that the collection $(\xi_k^{(1)}c_{k,j}^{(2)})_{k, j\geq1}$ still has the law of i.i.d. and symmetric coin tosses. 
    
    \noindent Now, we iterate the above two steps inside every loop of $A^{(2)}$ that has diameter greater than some fixed $\rho>0$. By Proposition \ref{prop:local-finite-tree}, this iteration will stop after a finite number $N\equiv N(\rho)$ of steps. By Proposition \ref{prop:sign-permute}, we can rearrange and relabel the obtained sums after every iteration and successively apply \eqref{eq:iterate-bound} to obtain that
    \begin{equation}
        \sum_{k=1}^\infty\E[(\mu_k^{(N)}, f)^2] = \E[(\field{\sin(\alpha\phi)}, f)^2] -  \E[(\field{\sin(\alpha\phi^{A^{(N+1)}})}, f)^2].
    \end{equation}
    Here we write $(\mu_k^{(N)})_{k\geq1}$ for the collection of measures appearing after the $N$ iterations, indexed under the new relabelling. As we take the diameter cuttoff $\rho\to0$, or equivalently $N\to\infty$, the usual argument using \eqref{eq:sin-mon-leb} gives that
    \begin{equation}
         \E[(\field{\sin(\alpha\phi^{A^{(N+1)}})}, f)^2]\longrightarrow0.
    \end{equation}
    This concludes the proof of the decomposition for a fixed test function $f$, as we have shown that
    \begin{equation}
        \sum_{k=1}^\infty\E[(\mu_k^{(\infty)},f)^2] <\infty,
    \end{equation}
    and thus 
    \begin{equation}
        (\field{\sin(\alpha\phi)}, f) = \sum_{k=1}^\infty\xi_k^{(\infty)}(\mu_k^{(\infty)},f)
    \end{equation}
    converges in $L^2(\P)$.
    
    \noindent It remains to raise the weak convergence above to strong convergence in $H^s(\C)$ for $s<-1$. Noting that $H^{-s}(\C)$ is a subset of the space of continuous and bounded functions $f:\C\to\R$, this is a standard application of the uniform boundedness principle and the Rellich-Kondrachov embedding theorem (since all the fields are supported in a \emph{bounded} domain).
\end{proof}
    
\begin{rem}
    One can directly arrive at \eqref{eq:main-sin} or \eqref{eq:key} by conditioning on both $A_1$ and $A_2$ simultaneously from the start. We prefer to do so in two steps, and state Lemma \ref{lem:CLE4-thin} separately, for the sake of exposition.
\end{rem}

With the decomposition of $\field{\sin(\alpha\phi)}$ at hand, that of $\field{\cos(\alpha\phi)}$ follows readily since the only difference is that one must first explore the boundary cluster.

\begin{proof}[Proof of Theorem \ref{thm:cont-general} for $\field{\cos(\alpha\phi)}$]
    Fix a valid domain $D\subset\C$ and a continuous, bounded  function $f:\C\to\R$. By Proposition \ref{prop:cos-TVS}, we can start by exploring the action $\A_{-a_c, a_c}\acts\cos(\alpha\phi)$ to obtain that
    \begin{align}
        (\field{\cos(\alpha\phi}), f) = (\mu_0, f) + (\field{\sin(\alpha\phi^{\A})}, f).
    \end{align}
    From here, since $(\field{\sin(\alpha\phi^\A)}, f)$ is well-defined and in $L^2(\P)$, we continue as in the proof above.

    \noindent If $D$ is not a valid domain, one still obtains convergence in the local Sobolev space $H^s_{loc}(D)$.
\end{proof}

\begin{proof}[Proof of Theorem \ref{thm:cont-XOR}]  
    Recall that the continuum XOR-Ising field in $D$ is defined as the (appropriately renormalized) scaling limit $\delta\to0$ of the discrete XOR-Ising field on a discrete domain approximation $D_\delta\subset\delta\Z^2$. See Theorem \ref{thm:conv-fields} for a precise statement. Moreover, under $+$ boundary conditions (say), the law of the continuum XOR-Ising field is that of $\mathcal{C}\sqrt{2}\field{\cos(\alpha\phi)}$, for some explicit constant $\mathcal{C}$ and $\alpha=1/\sqrt{2}$. Hence the result follows from Theorem \ref{thm:cont-general}, minding that the measures need to be normalized differently to account for the prefactor $\mathcal{C}\sqrt{2}$.
\end{proof}


\section{Local sets of imaginary chaos and GFF}\label{sec:local}

This section is devoted to proving that any ``reasonable'' local set for $\field{e^{i\alpha\phi}}$ must be a local set for the underlying GFF $\phi$. The heart of the proof is the measurability of the latter with respect to the former as proved in \cite{GFF-IGMC}. This is a result of independent interest that, although purely in the continuum, will be crucial in the proof of the convergence of the discrete decompositions as stated in Theorem \ref{thm:discrete-conv} and Corollary \ref{cor:joint-conv}. 

\noindent \textbf{Notation}: For any set $U\subseteq D$, we write $\restr{(\field{e^{i\alpha\phi}})}{U}$ to mean the collection $(\field{e^{i\alpha\phi}}, f)$ indexed by $f\in C^\infty_c(U)$, i.e. smooth functions compactly supported in $U$. The partial derivative $\partial_x\phi$ in $D$ is defined as the distribution $(\partial_x\phi, f):=-(\phi, \partial_xf)$ for any $f\in C^\infty_c(D)$. The gradient 
$\nabla\phi$ is the vector of distributions given by $(\nabla\phi, F) := -(\phi, \nabla\cdot F)$ for any smooth vector field $F\in C^\infty_c(D, \R^2)$.

\begin{thm}\label{thm:local}
    Let $D\subset\C$ be a proper, simply connected domain. Let $(\phi, A, \phi^A, h_A)$ be a coupling of a GFF $\phi$ with zero boundary conditions in $D$ and a random closed set $A\subset \bar D$ connected to $\partial D$ such that, given $A$, $\phi^A$ is a GFF with zero boundary conditions in $D\setminus A$ and $h_A$ is a harmonic function in $D\setminus A$. Suppose that 
    \begin{enumerate}[label=(\roman*)]
        \item Given $A$, 
        \begin{equation}
            \restr{(\field{e^{i\alpha\phi}})}{D\setminus A} =\ \big(e^{i\alpha h_A}\restr{\field{e^{i\alpha\phi^A}}\big)}{D\setminus A}\quad  \text{a.s.}
        \end{equation}
        \item For every connected component $\mathcal O$ of $D\setminus A$ and any $f\in C^\infty_c(D)$, $$(\field{e^{i\alpha\phi}}, f\1_{[\partial\mathcal O]_n}) \longrightarrow 0$$ in probability as $n\to\infty$.
        \item  For any $N\in\N$, the random variables
        \begin{equation}
            \restr{\big(\field{e^{i\alpha\phi^A}}\big)}{\cup_{i=1}^N\mathcal O_i} \text{ and } \ \big(h_A, \restr{\big(\field{e^{i\alpha\phi}}\big)}{D\setminus\cup_{i=1}^N\mathcal O_i}\big) 
        \end{equation}
        are conditionally independent given $A$.
    \end{enumerate}
    Then, $(\phi,A)$ is a local set coupling. That is, when restricted to $D\backslash A$,
    \begin{equation}
        \phi=\phi^A + h_A + \frac{2\pi}{\alpha} k_{\mathcal O},
    \end{equation}
    where $\phi^A$ is conditionally independent of $h_A+(2k_{\mathcal O}/\alpha)\pi$ given $A$, and $k_{\mathcal O}$ is some integer that may depend on the connected component $\mathcal O$ of $D\setminus A$. 
\end{thm}

\def\O{\mathcal O}
\begin{proof}
    By the assumption $(i)$, given $A$, we have that 
      \begin{equation}\label{eq:local-assumption}
        (\field{e^{i\alpha\phi}}, f) = (e^{i\alpha h_A}\field{e^{i\alpha\phi^A}}, f)
    \end{equation}
    for any $f\in C_c^\infty(D\setminus A)$. By {\cite[Theorem 1]{GFF-IGMC}}, the \emph{same} local\footnote{To mean that the partial derivatives of the underlying fields are measurable with respect to their chaos restricted to any open neighbourhood of the support of the test function. Here, it is crucial that 
    given $A$, the field $e^{i\alpha h_A}\field{e^{i\alpha\phi^A}}$ is absolutely continuous with respect to the imaginary chaos of the GFF in $D\backslash A$. } procedure can be applied to both sides of \eqref{eq:local-assumption} to recover the respective gradients, which must then be equal. As this defines the fields up to a global additive constants in each connected component of $D\backslash A$, it follows that 
    \begin{equation}
        (\phi, f) = (\phi^A+h_A+c_{\mathcal O}, f),
    \end{equation}
    for some constant $c_\O$ that may depend on the connected component $\O$ of $D\setminus A$. Moreover, for \eqref{eq:local-assumption} to hold, it must be the case that
    \begin{equation}
        c_{\mathcal O}=\frac{2\pi}{\alpha}k_{\mathcal O},
    \end{equation}
    for some integer $k_{\mathcal O}$. It remains to show that $\phi^A$ is conditionally independent independent of $h^A+c_\O$ given $A$. Since any harmonic function can be recovered from its values arbitrarily close to the boundary,
    \begin{equation}
        h^A + c_\O \in \bigcap_{n\geq 1}\sigma\left(A, \restr{\phi}{[A]_n}\right) = \bigcap_{n\geq 1}\sigma\left(A, \restr{(\field{e^{i\alpha\phi}})}{[A]_n}\right) =:\F.
    \end{equation}
    \sloppy It thus suffices to show that $\F$ is conditionally independent of the finite-dimensional marginals $(\phi^{\O_1}, \ldots, \phi^{\O_N})$ for any $N\in\N$, where $\phi^{\O}$ denotes a GFF with zero boundary conditions in $\O$. Now, for any fixed $n_0\in\N$ and again thanks to \cite{GFF-IGMC},
    \begin{align}
        \F & \subseteq \bigcap_{n\geq n_0}\sigma\left(A, \restr{(\field{e^{i\alpha\phi}})}{[A]_n\cup(\bigcup_{i=N+1}^\infty \O_i)}\right)\\
        & = \bigcap_{n\geq n_0}\sigma\left(A, \restr{(\field{e^{i\alpha\phi}})}{D\setminus\bigcup_{i=1}^N \O_i}, \restr{\big(e^{i\alpha h_A}\field{e^{i\alpha\phi^A}}\big)}{[A]_n\cap(\bigcup_{i=1}^N \O_i)}\right) \\\label{eq:sigma-algebra}
        & \subseteq \bigcap_{n\geq n_0}\sigma\left(A, \restr{(\field{e^{i\alpha\phi}})}{D\setminus\bigcup_{i=1}^N \O_i}, \restr{\big(h_A\big)}{[A]_n\cap(\bigcup_{i=1}^N \O_i)}, \restr{\big(\field{e^{i\alpha\phi^A}}\big)}{[A]_n\cap(\bigcup_{i=1}^N \O_i)}\right),
    \end{align}
    where the equality follows by the thinness assumption $(ii)$.  For the sake of simplicity, we write $X, Y_n, Z_n$ for the three restrictions of fields appearing in \eqref{eq:sigma-algebra}, respectively. Moreover, we set
    \begin{equation}
        W_{n_0} = \big(\restr{(\phi^A)^{[A]_{n_0}}\big)}{\bigcup_{i=1}^N\O_i},
    \end{equation}
    where we note that $[A]_{n_0}$ is also a local set for $\phi^A$.
    Note that 
    \begin{equation}
        Z_n, W_{n_0} \in \sigma\left(\restr{(\field{e^{i\alpha\phi^A}})}{\cup_{i=1}^N\O_i}\right),
    \end{equation}
    and they are conditionally independent of each other given $A$ whenever $n\geq n_0$. Using these two observations, along with the assumption $(iii)$, we see that for any any continuous, bounded functionals $r, s, t, u$ in the appropriate spaces,
    \begin{align*}
        \E\left[r(X)s(Y_n)t(Z_n)u(W_{n_0})\mid A\right] & = \E\left[r(X)s(Y_n)\ \E[t(Z_n)u(W_{n_0})\mid A, X, Y_n]\mid A\right] \\
        & = \E[r(X)s(Y_n)\mid A]\ \E[t(Z_n)u(W_{n_0})\mid A] \\
        & = \E[r(X)s(Y_n)\mid A]\ \E[t(Z_n)\mid A]\ \E[u(W_{n_0})\mid A] \\
        & = \E[r(X)s(Y_n)t(Z_n)\mid A]\ \E[u(W_{n_0})\mid A].
    \end{align*}
    Since 
    \begin{equation}
        W_{n_0}\longrightarrow \phi^{\cup_{i=1}^N\O_i}\equiv (\phi^{\O_1}, \ldots, \phi^{\O_N})
    \end{equation}
    almost surely as $n_0\to\infty$, and all functionals are bounded, the claimed conditional independence follows.
\end{proof}


\section{Preliminaries in the discrete: XOR-Ising and double random currents} \label{sec:prelim-discrete}
We introduce the (nearest-neighbor, ferromagnetic) XOR-Ising model, focusing on its representation in terms of double random currents. We also present the main results established in \cite{DRC-1, DRC-2} concerning the scaling limit of this percolation model, which will be key for us to be able to show the convergence of the discrete XOR-Ising model towards the one we have defined in Theorem \ref{thm:cont-XOR}. We provide only the necessary facts to prove Theorem \ref{thm:discrete-conv} and Corollary \ref{cor:joint-conv}, referring to \cite{ising-GFF, DRC-1} for the full structure of the coupling\footnote{Note however, that in the former the double random current with wired boundary conditions is defined on the primal graph, while in the latter it is defined on the dual graph. Here, we choose the first of these conventions.}.

\subsection{XOR-Ising model} Let us start with the very basic definition of the XOR-Ising model, to set some notation and conventions.

Let $D\subset \C$ be a Jordan domain. For concreteness, we define a discrete domain approximation $D_\delta\subset\delta\Z^2$ to be the induced subgraph where every vertex corresponds to a dual face intersecting $D$. The vertices corresponding to a dual face not contained in $D$ are denoted by $\partial D_\delta$. See Figure~\ref{fig:discrete-domain}. The weak dual graph $D_\delta^\dagger$ is simply the (strong) dual graph where the vertex corresponding to the unbounded face has been removed.

\begin{figure}[h]
    \centering
    \includegraphics[width=0.7\linewidth]{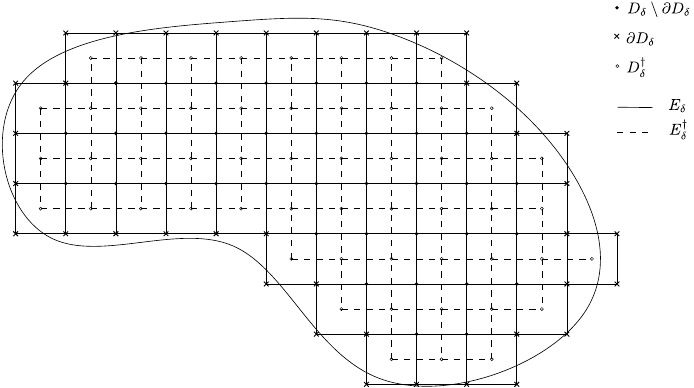}
    \caption{The discrete domain approximation $D_\delta$ chosen above.}
    \label{fig:discrete-domain}
\end{figure}

The XOR-Ising model with $+$ boundary conditions on $D_\delta$, at inverse temperature $\beta>0$, is defined to be the product $$(\tau_\delta(v))_{v\in D_\delta}:=(\sigma_\delta(v)\tilde\sigma_\delta(v))_{v\in D_\delta}\in\{\pm1\}^{D_\delta}$$ of two independent Ising models, i.e. sampled from the probability measure
\begin{equation}
    \P^{+}_{\delta,\beta}((\sigma_\delta, \tilde\sigma_\delta)) := \frac{1}{Z_{\delta,\beta}^+} \exp\left\{\sum_{v\sim v'}\beta(\sigma_\delta(v)\sigma_\delta(v')+\tilde\sigma_\delta(v)\tilde\sigma_\delta(v'))\right\}\1_{\big\{\restr{\sigma_\delta}{\partial D_\delta} \equiv \restr{\tilde\sigma_\delta}{\partial D_\delta} \equiv 1\big\}}
\end{equation}
on $\{\pm1\}^{D_\delta}\times\{\pm1\}^{D_\delta}$, where ${Z_{\delta,\beta}^+}$ is a normalization constant. The XOR-Ising model $\tau_\delta^\dagger$ with free boundary conditions on $D_\delta^\dagger$ is defined analogously, now under the measure
\begin{equation}
    \mathbb\P^{\varnothing}_{\delta,\beta}((\sigma_\delta^\dagger, \tilde\sigma_\delta^\dagger)) := \frac{1}{Z_{\delta,\beta}^\varnothing}  \exp\left\{\sum_{u\sim u'}\beta^\dagger(\sigma_\delta^\dagger(u)\sigma_\delta^\dagger(u')+\tilde\sigma_\delta^\dagger(u)\tilde\sigma_\delta^\dagger(u'))\right\},
\end{equation} 
where $\beta^\dagger$ is the dual inverse temperature satisfying
\begin{equation}
    \tanh(\beta^\dagger) = \exp(-2\beta).
\end{equation}
The critical inverse temperature $\beta_c$ is self-dual, and given by $$\beta_c=\frac{1}{2}\log(1+\sqrt{2}).$$ When considering $\beta_c$ we drop it from the notation.

To take the scaling limit $\delta\to0$ of the critical XOR-Ising models, one must view the spin systems as distributions and renormalize them appropriately. Namely, we set
\begin{equation}\label{eq:XOR-field}
    \tau_\delta(z) :=\delta^{-1/4}\tau_\delta(v(z)),
\end{equation}
where $v(z)$ is the vertex of $D_\delta$ corresponding to the dual face containing $z$. For each $\delta>0$ fixed, this is a well-defined function, and in particular a distribution. The following result identifies the limit of the critical XOR-Ising with $+$ boundary conditions (resp. free boundary conditions) as the real (resp. imaginary) part of the complex multiplicative chaos with parameter $\alpha=1/\sqrt{2}$. We later explain why this identification holds at the discrete level already -- see the identity \eqref{eq:discrete-cos-sin}. Moreover, in Lemma \ref{lem:conv-fields-fractal}, we extend the statement to joint convergence (and drop any assumptions on the compact support of the test functions).

\begin{thm}[\cite{CHI, CHI-2, IGMC}] \label{thm:conv-fields}
Let $D$ be a Jordan domain. For every continuous function $f:D\to\R$ with compact support in $D$, and any $k\geq1$, 
\begin{enumerate}
    \item[(i)] $\E_{\delta}^+[(\tau_\delta, f)^k]\to\mathcal{C}^{k}\ 2^{k/2}\ \E[(\field{\cos(\frac{1}{\sqrt{2}}\phi)}, f)^k]$,
    \item[(ii)] $\E_{\delta}^\varnothing[(\tau_\delta^\dagger, f)^k]\to\mathcal{C}^{k}\ 2^{k/2}\ \E[(\field{\sin(\frac{1}{\sqrt{2}}\phi)}, f)^k]$,
\end{enumerate}
as $\delta\to0$, where $\mathcal{C}$ is an explicit (lattice-dependent) constant. In particular, the fields converge in distribution with respect to the topology of the Sobolev space $H^{s}_{loc}(D)$ for any $s<-1$.
\end{thm}

\begin{rem}
    Both statements are a direct consequence of the results in \cite{CHI}, which determine explicitly the scaling limit of the $k$-point correlation functions of the critical Ising model. Under free boundary conditions, the formula for $k\geq3$ can be found after a short computation using {\cite[Theorem 7.1]{CHI-2}}. The identification of the limiting field under $+$ boundary conditions first appeared in \cite{IGMC}. The optimal Sobolev regularity is $s=-1/4$, as proved in \cite{tightness-fields, IGMC}. 
\end{rem}

\subsection{Double random currents} To introduce the discrete excursion decomposition of the XOR-Ising model, we need to introduce the double random current (DRC) model. A current $\nb_\delta$ on $D_\delta$ is a function $\nb_\delta:E_\delta\to\{0,1,\ldots\}$ defined on the edges of the graph. The set of sources of $\nb_\delta$ is
\begin{align}
    \partial\nb_\delta := \big\{v\in D_\delta : \sum_{v'\sim v}\nb_\delta(vv')\ \text{is odd}\ \big\},
\end{align}
where in order to impose ``wired'' boundary conditions we take the convention that all the vertices in $\partial D_\delta$ are identified into a single ``ghost'' vertex and all edges between two vertices in $\partial D_\delta$ are disregarded (or, equivalently, set to have $\nb_\de(e)=2$). On $D_\delta^\dagger$, the definition is analogous, without needing to specify any way of enforcing boundary conditions. We denote the set of currents $\nb_\delta$ with $\partial\nb_\delta=\emptyset$ by $\Omega_\delta^\emptyset$.

The random current model with wired boundary conditions on $D_\delta$ , at inverse temperature $\beta$, is given by the probability measure
\begin{align}
    \Pb^{\textrm{w}, \curr}_{\delta, \beta}(\nb_\delta) := \cfrac{1}{Z_{\delta}^{\textrm{w}, \curr}}\ \prod_{e\in E}\frac{\beta^{\nb_\delta(e)}}{\nb_\delta(e)!} 
\end{align}
on $\Omega_\delta^\emptyset$, where $Z_{\delta}^{\textrm{w}, \curr}$ is a normalization constant. The double random current model on $D_\delta$ with wired boundary conditions is defined as the sum of two independent random currents with wired boundary conditions. That is, 
\begin{align}
    \Pb^{\textrm{w}, \DRC}_{\delta, \beta}(\nb_\delta) :=  \Pb^{\textrm{w}, \curr}_{\delta, \beta} \otimes \Pb^{\textrm{w}, \curr}_{\delta, \beta}(\{(\nb_\delta^1, \nb_\delta^2)\in\Omega_\delta^\emptyset\times\Omega^\emptyset_\delta : \nb_\delta = \nb_\delta^1+\nb_\delta^2\}).
\end{align}
The respective models with free boundary conditions on $D_\delta^\dagger$ are defined in complete analogy, where the only difference is that we do not identify any set of vertices together. The corresponding measure is denoted by $$\Pb^{\textrm{f}, \DRC}_{\delta, \beta^\dagger}(\nb_\delta^\dagger).$$ 
The trace of a current $\nb_\delta$ is defined to be
\begin{equation}
    \hat\nb_\delta := \{e\in E_\delta: \nb_\delta(e)>0\}.
\end{equation}
It is precisely the trace of a DRC that yields the excursion decomposition of the XOR-Ising model, as first noted in \cite{spin-perc-height, DRC-1} -- see {\cite[Remark 9]{spin-perc-height}}. This readily follows from (a corollary to) the classical switching lemma introduced in \cite{switching}. 
\begin{prop}[Switching Lemma\label{prop:switching}]
    For any set of vertices $A\subset D_\delta\backslash \partial D_{\delta}$,
    \begin{equation}\label{eq:switching}
        \E^+_{\delta, \beta}\left[\prod_{v\in A}\tau_\delta(v)\right] = \Pb^{\textup{w},\DRC}_{\delta, \beta}\left(\hat{\nb}_\delta\in\mathfrak F_A\right)
    \end{equation}
    where $\mathfrak{F}_A$ is the event that every cluster of $\hat\nb_\delta$ intersects $A$ (resp. $A\cup\{\partial D_\delta\}$) an even number of times if $|A|$ is even (resp. $|A|$ is odd). The dual statement with free boundary conditions also holds.
\end{prop}

\noindent The key observation to find the excursion decomposition goes as follows. We define a spin model by assigning symmetric $\{\pm1\}$-valued spins to each cluster of $\hat\nb_\delta$ independently, except for the boundary cluster (i.e. the cluster containing the \emph{vertex} $\{\partial D_\delta\}$) which is always assigned spin $+1$. For any $A\subset D_\delta$, the correlation functions of this model are given by the right-hand side of \eqref{eq:switching}. Since the collection of all such functions determines the law of any $\{\pm1\}$-valued spin model uniquely, it must have the law of the XOR-Ising model. 

Before stating this decomposition rigorously, let us introduce some helpful notation, following \cite{DRC-1}. We write $(\mathcal C^\delta_{k})_{k\geq0}$ to denote the collection of clusters of $\hat\nb_\de$ viewed as \emph{subgraphs} of $D_\delta$, and ordered by decreasing size of diameter. To every cluster $\mathcal {C}^\delta$ we associate a loop configuration made of the dual edges $e^\dagger\in E_\delta^\dagger$ where $e=\{v, v'\}$ is such that $v\in\mathcal{C}^\delta$ and $v\notin\mathcal{C}^\delta$. This procedure gives:
\begin{enumerate}
    \item[(i)] A single loop corresponding to the the unbounded component of $D_\delta \setminus \mathcal C^\delta$, called the \emph{outer boundary of $\mathcal C^\delta$}.
    \item[(ii)] A collection of loops corresponding to the bounded components, called the \emph{inner boundaries of $\mathcal C^\delta$}. 
\end{enumerate}
In particular, to each $\mathcal{C}^\delta$ we can associate a closed set $C^\delta$ given by the region enclosed between the outer boundary and the inner boundaries (equivalently, by gluing together all the $\delta$-squares centred at points in $\mathcal{C}^\delta$).

\begin{cor}[\cite{DRC-1, spin-perc-height}]\label{cor:discrete-exc}
   Let $\tau_\delta$ be the XOR-Ising field with $+$ boundary conditions on $D_\delta$. Let $(C_k^\delta)_{k=0}^N$ be the collection of clusters of the trace $\hat\nb_\delta$ of a DRC with wired boundary conditions on $D_\delta$. Then,
    \begin{align*}
        \tau_\delta(z) = \mu^\delta_0(z)+\sum_{k=1}^N\xi_k^\delta\mu^\delta_k(z), 
    \end{align*}
    where $(\xi_k^\delta)_{k\geq1}^N$ are the i.i.d. symmetric signs and $(\mu_k^\delta)_{k\geq1}^N$ are the renormalized discrete area measures given by 
    \begin{equation}
        \mu^\delta_k(dz):=\delta^{-1/4}\mathbf{1}_{\{z\in C^\delta_k\}}dz.
    \end{equation}
    The dual statement  also holds, without the measure $\mu_0^\delta$ of the boundary cluster.
\end{cor}

\subsection{Master coupling} Finally, let us present the (almost) complete structure of the joint coupling introduced in \cite{DRC-1}, referred to as the \emph{master coupling}, which allows one to couple all the relevant models both in the primal graph and the dual graph in a consistent manner. As we use extensively this coupling in the sequel, we collect the most relevant facts for our purposes. 

 We say that 
\begin{enumerate}
    \item[(i)] A cluster $\mathcal C^\delta$ of $\hat\nb_\delta$ is odd around a face  $u\in D_\delta^\dagger$ if the total current of the cluster flowing across any dual path connecting $u$ to the unbounded face is odd.
     \item[(ii)] An inner boundary $\ell^\de$ of a cluster $\mathcal C^\de$ is odd  if the cluster $\mathcal C^\de$ is odd around every face encircled by $\ell^\de$. We set the parity label to be
     \begin{equation}
         c_\delta(\ell^\delta) := \begin{cases}
             + 1\quad \text{if}\ \ell^\delta\ \text{is odd} \\
             -1\quad \text{otherwise.}
         \end{cases}
     \end{equation}
\end{enumerate}
The dual definitions are analogous. In the next statement, all primal models are at inverse temperature $\beta$, while all dual models are at inverse temperature $\beta^\dagger$.

\begin{thm}[{\cite[Theorem 3.1]{DRC-1}}]\label{thm:master}
    There exists a coupling $\P_{\delta}$ of the following objects
    \begin{enumerate}[label=(\roman*)]
        \item A double random current $\nb_\delta$ with wired boundary conditions on $D_\delta$ and a double random current $\nb_\delta^\dagger$ with free boundary conditions on $D_\delta^\dagger$
        \item A height function $h_\delta$ taking integer values on $D_\delta$ and half-integer values on $D_\delta^\dagger$
        \item A XOR-Ising model $\tau_\delta$ with $+$ boundary conditions on $D_\delta$ and a XOR-Ising model $\tau_\delta^\dagger$ with free boundary conditions on $D_\delta^\dagger$
    \end{enumerate}
    in such a way that
    \begin{enumerate}[label=(\arabic*)]
        \item There exists no pair $(e, e^\dagger)$ of dual edges such that $\nb_\delta(e)>0$ and $\nb_\delta^\dagger(e^\dagger)>0$.
        \item\label{labels} The set
        \begin{equation}
            \nb_\delta^\odd :=\{e\in E_\de: \nb_\delta(e)\ \text{is odd}\}
        \end{equation}
        is equal to the collection of interfaces of $\tau_\delta^\dagger$. In particular, the spin $\tau_\delta^\dagger(u)$ is given by a product of  $-c_\delta$ over the inner boundaries surrounding $u$.
        The same holds for the dual models.
        \item\label{eq:cos-sin-height} The spins are given by
        \begin{align}
            &\tau_\delta(v) = \cos(\pi h_\delta(v)), \\ \label{eq:discrete-cos-sin}
            &\tau_\delta^\dagger(u) = \sin(\pi h_\delta(u)).
        \end{align}
        In particular, 
        \begin{equation}\label{eq:d-grad}
            h_\de(u)-h_\de(v) = \frac{1}{2}\ \tau_\de(v)\tau_\de^\dagger(u).
        \end{equation}
        \item The conditional law of $\tau_\delta$ given $\nb_\delta$ is that of i.i.d. symmetric sign assignments $(\xi_k^\delta)_{k\geq1}$ to the clusters of $\nb_\delta$. The same holds for the dual models.
        \item\label{boson} \cite{ dimer-boson, nesting-field} The bosonisation identity holds, i.e.
        \begin{equation}\label{eq:bosonisation}
            \E_\delta\Big[\prod_{i=1}^n\cos(\pi h_\delta(v_i))\prod_{j=1}^k\sin(\pi h_\delta(u_j))\Big] = \E_{\delta, \beta}^+\Big[ \prod_{i=1}^n\sigma_\delta(v_i)\prod_{j=1}^k\mu_\delta(u_j)\Big]^2
        \end{equation}
        for any primal vertices $v_1, \ldots, v_n\in D_\delta$ and any dual vertices $u_1, \ldots, u_k\in D_\delta^\dagger$. Here, the so-called disorder operators are such that
        \begin{align}\label{eq:disorder}
            \E_{\delta, \beta}^+\Big[\prod_{j=1}^k \mu_\delta(u_j)\Big] := \E_{\delta, \beta}^+\Big[\exp\big\{-2\beta\sum_{xx'\in E_\delta^\gamma}\sigma_\delta(x)\sigma_\delta(x')\big\}\Big] = \E_{\delta, \beta^\dagger}^\varnothing\Big[\prod_{j=1}^k\sigma_\delta^\dagger(u_j)\Big],
        \end{align}
        where $\gamma$ is any\footnote{The second equality is the well-known Kramers--Wannier duality of the Ising model. In particular, the expression on the left-hand side of \eqref{eq:disorder} does not depend on the choice of the disorder line $\gamma$, but only on the choice of dual vertices $u_1, \ldots, u_k$. The expectation on the right-hand side of \eqref{eq:bosonisation} does depend on the homology class of $\gamma$ in $D_\delta\setminus\{v_1, \ldots, v_n\}$, but only up to a sign that is cancelled by the square.} collection of dual edges such that $\{u_1, \ldots, u_k\}$ is the set of dual vertices incident to an odd number of edges in $\gamma$,  and $E_\delta^\gamma$ is the subset of edges of $E_\delta$ crossed by $\gamma$.
        \item\label{MP-free} When exploring a cluster of $\nb_\delta$ from the outside, the law of $\nb_\delta$ inside each inner boundary is that of a DRC with free boundary conditions. The same holds for $\nb_\delta^\dagger$.
        \item\label{MP-wired} When exploring a cluster of $\nb_\delta$ from the outside, the law of $\nb_\delta^\dagger$ inside each inner boundary is that of a DRC with wired boundary conditions (and thus the inner boundary itself is identified into a single vertex). The dual statement, swapping the roles of $\nb_\delta$ and $\nb_\delta^\dagger$, also holds. 
    \end{enumerate}
\end{thm}

We are now ready to state the complete convergence results from \cite{DRC-1, DRC-2}. 
The precise definition of the topology for the convergence of the inner and outer boundaries of the currents can be found in \cite[Section 1.1]{DRC-1}, while for the topology of convergence of the height function one may take $H^s(\C)$ with $s<-1$. Here, we abuse notation and identify $\hat\nb$ or $\hat\nb^\dagger$ with a collection of inner and outer boundaries.

\begin{thm}\label{thm:DRC}
    Let $D$ be a Jordan domain and $D_\delta$ be as above. For the (critical) coupling $\P_\delta$, 
    \begin{equation}
        (h_\delta, \hat\nb_\delta, \hat\nb_\delta^\dagger, c_\delta, c_\delta^\dagger, \xi_\delta, \xi^\dagger_\delta) \overset{(d)}{\longrightarrow} (\frac{1}{2\sqrt{2}\lambda}\ h, \hat\nb, \hat\nb^\dagger, c, c^\dagger, \xi, \xi^\dagger),
    \end{equation}
    where
    \begin{enumerate}[label=(\roman*)]
        \item h is a GFF with zero boundary conditions in $D$, 
        \item The inner boundaries of the boundary cluster of $\hat\nb_\de$ converge to $\mathcal L_{-\sqrt{2}\lambda, \sqrt{2}\lambda}$. The label $c(\ell)$ of an inner boundary $\ell$ is such that the boundary values of $h$ in $\mathcal O(\ell)$ are $c(\ell)\sqrt{2}\lambda$.
        \item The outer boundaries of the outermost clusters of $\hat\nb^\dagger_\delta$ converge to $\mathcal L_{-2\lambda, 2\lambda}$. The sign $\xi^\dagger(\gamma)$ of an outer boundary $\gamma$ is such that the boundary values of $h$ in $\mathcal O(\gamma)$ are $\xi^\dagger(\gamma)2\lambda$. Moreover, the inner boundaries of the cluster corresponding to $\gamma$ converge to
        \begin{enumerate}
            \item [(iii.a)] $\mathcal L_{-2\lambda, (2\sqrt{2}-2)\lambda}$ if $\xi^\dagger(\gamma)=+1$,
            \item [(iii.b)] $\mathcal L_{-(2\sqrt{2}-2)\lambda, 2\lambda}$ if $\xi^\dagger(\gamma)=-1$,
            \item [(iii.c)] where the label $c^\dagger(\ell)$ of an inner boundary $\ell$ is such that the boundary values of $h$ in $\mathcal O(\ell)$ are $\xi(\gamma)(c(\ell)+1)\sqrt{2}\lambda$. 
        \end{enumerate}
        \item The structure repeats iteratively by applying the Markov property \ref{MP-free} of Theorem \ref{thm:master}, which gives the nesting of two-valued sets described in Section \ref{sec:decomp}.
    \end{enumerate}
\end{thm}

\begin{rem}\label{rem:labels}
    We stress that the signs $\xi_\de$ are a measurable function of $h_\de$, and only their conditional law given $\hat\nb_\de$ is that of independent coin tosses. This is completely analogous to the structure in the continuum. In particular, the above result shows that our choice of the labels of $\A_{-2\lambda, 2\lambda}$ as the coin tosses in the continuum decomposition is the correct one from the discrete point of view, and that one should \emph{not} take them to be independent of everything else. While the two-valued sets with gap $2\sqrt{2}\lambda$ corresponds to jumps of $1/2$ in the discrete height function, and as such are \emph{purely geometrical}, one should think of $\A_{-2\lambda, 2\lambda}$ \emph{not} as a geometric object, but rather as generating the independence properties of the decomposition.
\end{rem}



\section{Convergence of the discrete decomposition: Proof of Theorem \ref{thm:discrete-conv}} \label{sec:conv}

The goal of this final section is to prove the convergence of the discrete excursion decomposition of the XOR-Ising model. We recall that, as discussed in Section \ref{sec:intro-conv}, the first step is to identify the limit of the XOR-Ising fields as the real and imaginary parts of the chaos of the limit of the discrete height function. We do so by translating the strong Markov properties of the discrete models to the continuum and invoking Theorem \ref{thm:local}. We believe that this approach should generalize nicely to AT models, as opposed to proving convergence of the joint moments of the height function and the polarisation fields. It also allows us to prove Corollary \ref{cor:joint-conv} readily, without computing the joint moments along with the four coupled Ising fields.

\subsection{Convergence of renewed fields and local set couplings}

We start by using the bosonisation equation \eqref{eq:bosonisation} to upgrade Theorem \ref{thm:conv-fields} into a joint convergence statement, while also allowing test functions to intersect the boundary of the domain. The identification of joint moments was already noted in {\cite[Remark 7.3]{CHI-2}}, and further bosonisation identities in the continuum were established in {\cite[Theorem 1.3]{cont-boson}}. 

\begin{lemma}\label{lem:conv-fields-fractal}
    Let $D$ be a Jordan domain such that $\dim_M(\partial D)<2-1/4$. Let $\E_\de$ denote the expectation under the (critical) coupling of Theorem \ref{thm:master} in the discrete approximation $D_\delta$. For any continuous, bounded function $f:\C\to\R$ and any $n,k\geq1$, 
    \begin{equation}
        \E_{\delta}\big[(\tau_\delta, f)^n\ (\tau_\delta^\dagger, f)^k\big]\longrightarrow\mathcal{C}^{n+k}\ 2^{(n+k)/2}\ \E\big[(\field{\cos(\frac{1}{\sqrt{2}}\phi)}, f)^n\ (\field{\sin(\frac{1}{\sqrt{2}}\phi)}, f)^k\big]
    \end{equation}
    as $\delta\to0$. In particular, the fields converge jointly in distribution with respect to the (product) topology of the Sobolev space $H^{s}(\C)$ for any $s<-1$.    
\end{lemma}
\begin{proof} 
    By Properties \ref{eq:cos-sin-height} and \ref{boson} of Theorem \ref{thm:master}, 
    \begin{equation}\label{eq:nk-functions}
        \E_{\delta}[(\tau_\delta, f)^n\ (\tau_\delta^\dagger, f)^k] = \int_D\ldots\int_D\ \E_\delta^+\Big[ \prod_{i=1}^n\sigma_\delta(v_i)\prod_{j=1}^k\mu_\delta(u_j)\Big]^2 f(v_1)\cdots f(u_j)\ dv_1\cdots du_j.
    \end{equation}
    By {\cite[Theorems 1.1-7.1]{CHI-2}}, for any $\eps>0$,
    \begin{equation}
        \E_\delta^+\Big[ \prod_{i=1}^n\sigma_\delta(v_i)\prod_{j=1}^k\mu_\delta(u_j)\Big]^2 \longrightarrow \mathcal{C}^{k+n}\ 2^{(k+n)/2}\ \left\langle\prod_{i=1}^n\field{\cos\left(\frac{1}{\sqrt{2}}\phi(v_i)\right)}\prod_{j=1}^k\field{\sin\left(\frac{1}{\sqrt{2}}\phi(u_j)\right)}\right\rangle
    \end{equation}
     as $\de\to0$, uniformly over all points at distance at least $\eps$ away from each other and the boundary. Thus it suffices to show that \eqref{eq:nk-functions} is uniformly integrable.
    
    \noindent Using the Kramers-Wannier duality in \eqref{eq:disorder},
    \begin{align}
         \E_{\delta}[(\tau_\delta, f)^n\ (\tau_\delta^\dagger, f)^k]^2 &\leq \E_{\delta}[(\tau_\delta, f)^{2n}]\ \E_{\delta}[(\tau_\delta^\dagger, f)^{2k}] \\
         &\leq \left(\int_D\ldots\int_D\ \E_\delta^+\Big[ \prod_{i=1}^{2n}\sigma_\delta(v_i)\Big] f(v_1)\cdots f(v_{2n})\ dv_1\cdots dv_{2n}\right) \\\label{eq:DOM}
         & \hspace{4mm} \times \left(\int_D\ldots\int_D\ \E_\delta^\varnothing\Big[ \prod_{j=1}^{2k}\sigma_\delta^\dagger(u_j)\Big] f(u_1)\cdots f(u_{2k})\ du_1\cdots du_{2k}\right).
     \end{align}
     Now, one can use e.g. {\cite[Proposition 3.10]{tightness-fields}} or {\cite[Lemma 4.4]{IGMC}} to see that
     \begin{equation}\label{eq:onsager-discrete}
        \E_\delta^+\Big[ \prod_{i=1}^{2n}\sigma_\delta(v_i)\Big] \leq C_D\prod_{i=1}^{2n} \left(\min_{j\neq i}|v_i-v_j|\wedge d(v_i, \partial D)\right)^{-1/8},
    \end{equation}
    where $C_D$ is a constant that only depends on the domain (and hence the right-hand side is uniform in $\de$). A similar bound holds under  $\E_\delta^\varnothing$. Using our assumption on the dimension of $\partial D$, a standard computation shows that \eqref{eq:DOM}-\eqref{eq:onsager-discrete} give the uniform integrability required.

    \noindent Convergence in distribution follows from the fact that the moments uniquely characterize the joint law by Proposition \ref{prop:unique}. 
\end{proof}

\begin{rem}\label{rem:boson}
    A natural question is whether one can prove Theorem \ref{thm:conv-fields} without the results of \cite{CHI-2}, but instead use the dimer correspondence and the convergence of the dimer height function. This has indeed been done in the infinite volume case in \cite{dimer-electric}. A further natural question is whether one can identify the limit of $\cos(\pi h_\de)$ and $\sin(\pi h_\de)$ without a need for correlation functions. If such were the case, given the argument in this paper, the only missing ingredient in the proof of Conjecture \ref{conj:AT} would be the convergence of the AT currents.
\end{rem}

The next step in the proof is to argue that the local set coupling given by property \ref{MP-wired} of Theorem \ref{thm:master} converges to a local set coupling of the continuum XOR-Ising models. Similar results concerning the convergence of local set couplings of the GFF can be found in \cite{SchSh} and {\cite[Lemma 4.9]{ALS2}}. The discrete setup can be seen in Figure \ref{fig:MP}.
\medskip

\begin{figure}[b]
    \centering
    \includegraphics[width=0.55\linewidth]{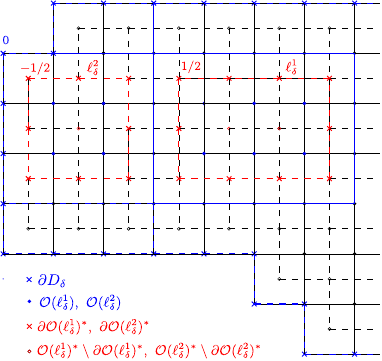}
    \caption{A portion of the boundary cluster of $\nb_\de$ in $D_\de$. Solid blue lines represent $\nb_\de^\odd$, dashed blue lines represent $\hat\nb_\de\setminus \nb_\de^\odd$. Only two inner boundaries $\ell_\de^1$ and $\ell_\de^2$ are highlighted with dashed red lines, with respective labels $c_\de^1=+1$ and $c_\de^2=-1$. The subgraphs of $D_\de$ and $D_\de^\dagger$ induced by these inner boundaries are as marked, where all vertices in an inner boundary are identified into a single one. The numbers represent the values of the height function $h_\de$, which must be constant across the boundary cluster of $\hat\nb_\de$ and across each of the inner boundaries.}
    \label{fig:MP}
\end{figure}

\begin{lemma}\label{lem:local-set-1}
    Let $D$ and $\P_\de$ be as in Lemma \ref{lem:conv-fields-fractal}. Let $\ell_\de$ be an inner boundary of the boundary cluster of $\hat\nb_\de$ on $D_\de$, with parity label $c_\de$. Let $\mathcal O(\ell_\de)$ be the induced subgraph of $D_\de$ whose vertices are enclosed by $\ell_\de$, and let $\mathcal{O}(\ell_\de)^*$ be its \emph{strong} dual graph, where all vertices of $D_\de^\dagger$ belonging to $\ell_\de$ are identified. Then,
    \begin{enumerate}
        \item[$\bullet$]  the law of the restriction of $\tau_\de$ to $\mathcal O(\ell_\de)$ is that of $-c_\de$ times a XOR-Ising model $\tau_\de^{\ell_\de}$ with free boundary conditions, 
        \item[$\bullet$]  the law of the restriction of $\tau_\de^\dagger$ to $\mathcal O(\ell_\de)^*$ is that of $c_\de$ times a XOR-Ising model  $(\tau_\de^{\ell_\de})^\dagger$  with $+$ boundary conditions.
    \end{enumerate}
    Moreover, every subsequential limit as $\de\to0$ of the joint law
    \begin{equation}
        \left( h_\de,\ \ell_\de,\ c_\de,\ \tau_\de,\ \tau_\de^\dagger,\ \tau_\de^{\ell_\de},\ (\tau_\de^{\ell_\de})^\dagger \right) 
    \end{equation}
    is given by 
    \begin{equation}
        \left(\frac{1}{2\sqrt{2}\lambda}\ h,\ \ell,\ c,\ \field{\cos(\frac{1}{\sqrt{2}}\phi)},\ \field{\sin(\frac{1}{\sqrt{2}}\phi)},\ \field{\sin(\frac{1}{\sqrt{2}}\phi^\ell)},\ \field{\cos(\frac{1}{\sqrt{2}}\phi^\ell)} \right),
    \end{equation}
    where 
    \begin{enumerate}
        \item[$\bullet$] the loop $\ell$ is a loop of the two-valued set $\A_{-\sqrt{2}\lambda, \sqrt{2}\lambda}$ of $h$,
        \item[$\bullet$] the label $c$ is the label of the above two-valued set,
        \item[$\bullet$] given $\ell$, the conditional law of $\phi^\ell$ is that of a GFF with zero boundary conditions in $\mathcal O(\ell)$,
        \item[$\bullet$] given $\ell$, the label $c$ and the pair $\big(\field{\cos(\frac{1}{\sqrt{2}}\phi^\ell)}, \ \field{\sin(\frac{1}{\sqrt{2}}\phi^\ell)} \big)$ are conditionally independent.
    \end{enumerate}
\end{lemma}
\begin{proof}
    The claimed law of $\tau_\de^{\ell_\de}$ and $(\tau_\de^{\ell_\de})^\dagger$ is an immediate consequence of properties \ref{labels} and \ref{MP-wired} of Theorem \ref{thm:master}. Equivalently, this can be rephrased in terms of the height function $h_\de$ as follows: given $\ell_\de$, when restricted to both $\mathcal O(\ell_\de)$ and $\mathcal O(\ell_\de)^*$,
    \begin{equation}\label{eq:MP-height}
        h_\de = h_\de^{\ell_\de} + c_\de/2,
    \end{equation}
    where $h_\de^{\ell_\de}$ has the law of a renewed height function with zero boundary conditions on $\mathcal O(\ell_\de)^*$ and is conditionally independent of the label $c_\de$. In particular, when restricted to $\mathcal O(\ell_\de)$ and $\mathcal O(\ell_\de)^*$ respectively,
    \begin{align}\label{eq:XOR-MP1}
        & \tau_\de = \cos(\pi (h_\de^{\ell_\de}+\frac{c_\de}{2})) = -c_\de\sin(\pi h_\de^{\ell_\de})\overset{(d)}{=}-c_\de\tau_\de^{\ell_\de}, \\ \label{eq:XOR-MP2}
        & \tau_\de^\dagger = \sin(\pi (h_\de^{\ell_\de}+\frac{c_\de}{2})) = c_\de\cos(\pi h_\de^{\ell_\de})\overset{(d)}{=}c_\de(\tau_\de^{\ell_\de})^\dagger.
    \end{align}

    \noindent For the second part, note first that the existence of a subsequential limit is immediate from tightness of each marginal law. The convergence of the triple $(h_\de, \ell_\de, c_\de)$ to the described limit is the content of Theorem \ref{thm:DRC}. The convergence of the pair $(\tau_\de, \tau_\de^\dagger)$  is the content of Lemma \ref{lem:conv-fields-fractal}, and in particular implies tightness of the pair $(\tau_\de^{\ell_\de}, (\tau_\de^{\ell_\de})^\dagger)$. Any subsequential limit is thus of the form
    \begin{equation}
        \Big(\frac{1}{2\sqrt{2}\lambda}\ h,\ \ell,\ c,\ \field{\cos(\frac{1}{\sqrt{2}}\phi)},\ \field{\sin(\frac{1}{\sqrt{2}}\phi)},\ \tau_*^{\ell},\  (\tau_*^{\ell})^\dagger \Big).
    \end{equation}
    To show the remaining two bullet points, we argue as in {\cite[Lemma 4.9]{ALS2}}, see also {\cite[Lemma B.2]{quant-SLE}} for the same result in terms of sequences of arbitrary random variables.  By Lemma \ref{lem:conv-fields-fractal}, noting that $\dim_M(\ell)=3/2$, we have that
    \begin{equation*}
        \E_{\delta}[(\tau_\delta^{\ell_\de}, f)^n\ ((\tau_\delta^{\ell_\de})^\dagger, f)^k \mid \ell_\de\ ]\longrightarrow\mathcal{C}^{n+k}\ 2^{(n+k)/2}\ \E[(\field{\sin(\frac{1}{\sqrt{2}}\phi^\ell)}, f)^n\ (\field{\cos(\frac{1}{\sqrt{2}}\phi^\ell)}, f)^k\mid \ell\ ]
    \end{equation*}
    almost surely for any $n,k\geq1$ and any continuous, bounded function $f:\C\to\R$. Equivalently, for any continuous, bounded functional $r$,
    \begin{equation}\label{eq:conv-given-loop}
        \E_\de[\ r(\tau_\delta^{\ell_\de},(\tau_\delta^{\ell_\de})^\dagger) \mid \ell_\de\ ]\longrightarrow  \E[\ r(\field{\sin(\frac{1}{\sqrt{2}}\phi^\ell)}, \field{\cos(\frac{1}{\sqrt{2}}\phi^\ell)})\mid \ell\ ].
    \end{equation}
    Moreover, for any continuous, bounded function $s$, 
    \begin{equation}\label{eq:CE-lim1}
        \E_\de[\ r(\tau_\delta^{\ell_\de},(\tau_\delta^{\ell_\de})^\dagger)\ s(\ell_\de, c_\de)\ ]\longrightarrow  \E[\ r(\tau_*^{\ell},\  (\tau_*^{\ell})^\dagger)\ s(\ell, c)\ ].
    \end{equation}
    But using the conditional independence of $(\tau_\de^{\ell_\de}, (\tau_\de^{\ell_\de})^\dagger)$ and $c_\de$ given $\ell_\de$, the left-hand side above can be written as
    \begin{align}
        \E_\de[\ r(\tau_\delta^{\ell_\de},(\tau_\delta^{\ell_\de})^\dagger)\ s(\ell_\de, c_\de)\ ]  &= \E_\de[\ \E_\de[\ r(\tau_\delta^{\ell_\de},(\tau_\delta^{\ell_\de})^\dagger) \mid \ell_\de\ ]\  \ s(\ell_\de, c_\de)\ ]  \\ \label{eq:CE-lim2}
        &\longrightarrow \E[\ \E[\ r(\field{\sin(\frac{1}{\sqrt{2}}\phi^\ell)}, \field{\cos(\frac{1}{\sqrt{2}}\phi^\ell)})\mid \ell\ ]\ s(\ell, c)\ ],
    \end{align}
    where the convergence follows from \eqref{eq:conv-given-loop} and the bounded convergence theorem. Comparing \eqref{eq:CE-lim1} and \eqref{eq:CE-lim2}, the last two bullet points follow.
\end{proof}

We explain now how to extend this result beyond the case of a single inner boundary. Let $A_\de$ be the union of the inner boundaries of the boundary cluster of $\hat\nb_\de$. Using the same notation, it holds that
\begin{enumerate}
    \item[$\bullet$] the law of the restriction of $\tau_\de$ to $\mathcal O(A_\de)$ is that of 
        \begin{equation}
            \sum_{\ell_\de\in A_\de}-c_\de(\ell_\de)\tau_\de^{\ell_\de} =: -c_\de\tau_\de^{A_\de},
        \end{equation}
    \item[$\bullet$]  the law of the restriction of $\tau_\de^\dagger$ to $\mathcal O(A_\de)^*$ is that of 
         \begin{equation}
            \sum_{\ell_\de\in A_\de}c_\de(\ell_\de)(\tau_\de^{\ell_\de})^\dagger =: c_\de(\tau_\de^{A_\de})^\dagger.
        \end{equation}
\end{enumerate}
Moreover, the tightness of $(\tau_\de, \tau_\de^\dagger)$ implies the tightness\footnote{Note, however, that it does not imply tightness of the unsigned sum $\tau_\de^{A_\de}$ or $(\tau_\de^{A_\de})^\dagger$.} of $(-c_\de\tau_\de^{A_\de}, c_\de(\tau_\de^{A_\de})^\dagger)$. Hence there exists a subsequential limit as $\de\to0$ of the joint law
    \begin{equation}\label{eq:subseq-joint}
        \left( h_\de,\ A_\de,\ c_\de,\ \tau_\de,\ \tau_\de^\dagger,\ -c_\de\tau_\de^{A_\de},\ c_\de(\tau_\de^{A_\de})^\dagger \right)
    \end{equation}
given by
  \begin{equation}
        \Big(\frac{1}{2\sqrt{2}\lambda}\ h,\ A,\ c,\ \field{\cos(\frac{1}{\sqrt{2}}\phi)},\ \field{\sin(\frac{1}{\sqrt{2}}\phi)},\ \tau_*^{A},\ (\tau_*^{A})^\dagger \Big),
    \end{equation}
where $(h, A, c)$ are as in Theorem \ref{thm:DRC} and $\phi$ is a GFF with zero boundary conditions in $D$. By Lemma \ref{lem:local-set-1}, we also know that for each $\ell_\de\in A_\de$ we have
    \begin{equation}
          \left(\tau_\de^{\ell_\de},\ (\tau_\de^{\ell_\de})^\dagger \right) \overset{(d)}{\longrightarrow}\left(\field{\sin(\frac{1}{\sqrt{2}}\phi^\ell)},\ \field{\cos(\frac{1}{\sqrt{2}}\phi^\ell)} \right),
    \end{equation}
where $\phi^\ell$ is a GFF with zero boundary conditions in $\mathcal{O}(\ell)$. In particular, we can include the convergence of \emph{all} renewed fields in the joint subsequential limit of \eqref{eq:subseq-joint}. Applying Skorokhod's representation theorem, let us pass to a common probability space in which convergence holds almost surely.

In the continuum coupling, we first note that, given $A$, $$\phi^A := \sum_{\ell\in A}\phi^\ell$$ has the law of a GFF with zero boundary conditions in $D\setminus A$, since every GFF appearing in the sum is (conditionally) independent of the others. We can also set $$h_A := c(\ell)\sqrt{2}\ \frac{\pi}{2}
=c(\ell)\sqrt{2}\lambda$$
to be a harmonic function in $D\setminus A$. This defines a coupling
\begin{equation}\label{eq:local-couple}
    (\phi, A, \phi^A, h_A)
\end{equation}
where the marginals laws are as in Theorem \ref{thm:local}. 

\begin{lemma}\label{lem:local-set-2}
    For any subsequential limit, the coupling $(\phi, A, \phi^A, h_A)$ above satisfies the hypotheses of Theorem \ref{thm:local}.
\end{lemma}
\begin{proof}
We check the hypotheses (i)-(iii) of Theorem \ref{thm:local} separately:
\begin{enumerate}[label=(\roman*)]
    \item Given a loop $\ell$ of $A$, for any continuous, bounded function $f:\C\to\R$,
        \begin{align}
            & (\field{\cos(\frac{1}{\sqrt{2}}\phi)}, f) = -c(\ell)\  (\field{\sin(\frac{1}{\sqrt{2}}\phi^\ell)}, f), \\
            & (\field{\sin(\frac{1}{\sqrt{2}}\phi)}, f) = c(\ell)\ (\field{\cos(\frac{1}{\sqrt{2}}\phi^\ell)}, f).
        \end{align}
    Indeed, this follows from Lemma \ref{lem:local-set-1} recalling \eqref{eq:XOR-MP1}-\eqref{eq:XOR-MP2}. By restricting to $f\in C_c^\infty(D\setminus A)$, which in particular removes any boundary issues, this is equivalent to
        \begin{equation}
            (\field{\exp(i\frac{1}{\sqrt{2}}\phi)}, f) =  (\exp(i\frac{1}{\sqrt{2}}h_A) \field{\exp(i\frac{1}{\sqrt{2}}\phi^A)}, f)
        \end{equation}
    by our choice of $h_A$, as required. 
    \item Immediate from the fact that $\dim_M(\ell)=3/2$ for every loop $\ell$ of $A$. 
    \item Fix $N\in\N$ and any collection $(\ell^1, \ldots, \ell^N)$ of loops of $A$. By Lemma \ref{lem:local-set-1}, the conditional law of 
        \begin{equation}
           \big(\field{\exp({i\pi h_\de^{\ell_\de^1}})}\ , \ldots,\ \field{\exp({i\pi h_\de^{\ell_\de^N}})}\big)
        \end{equation}
    given $A_\de$ converges to the conditional law of
        \begin{equation}
           \big(\field{\exp({i\frac{1}{\sqrt{2}}\phi^{\ell^1}})}\ , \ldots,\ \field{\exp({i\frac{1}{\sqrt{2}}\phi^{\ell^N}})} \big)
        \end{equation}
    given $A$. Using the same argument as in the proof of Lemma \ref{lem:local-set-1}, and noting that the information encoded by the harmonic function is equivalent to that of the labels, it is enough to show that, for any $\de>0$, the laws
        \begin{equation}
             \big(\field{\exp({i\pi h_\de^{\ell_\de^1}})}\ , \ldots,\ \field{\exp({i\pi h_\de^{\ell_\de^N}})}\big), \quad \big( (c_\de^k)_{k\geq1},\ \big(\field{\exp({i\pi h_\de^{\ell^k_\de}})}\big)_{k>N}\big)
        \end{equation}
        are conditionally independent given $A_\de$. This is a slight generalization of the Markov property stated in \eqref{eq:MP-height}, which follows readily by noting that it is always possible to explore the boundary cluster of $\nb_\de$ in such a way that the loops $\ell_\de^1,\ldots, \ell_\de^N$ are discovered from the outside.
    \end{enumerate}
\end{proof}

\subsection{Convergence of the decomposition}

We are now ready to put everything together and use Theorem \ref{thm:local} to identify the limiting XOR-Ising fields as the multiplicative chaos of the limiting height function.

\begin{thm}\label{thm:full-joint}
    Let $D$ and $\P_\de$ be as in Lemma \ref{lem:conv-fields-fractal}. As $\de\to0$, 
    \begin{equation}
        \left( h_\de,\ \tau_\de,\ \tau_\de^\dagger \right)  \overset{(d)}{\longrightarrow}\left(\frac{1}{2\sqrt{2}\lambda}\ h,\ \field{\cos(\frac{1}{\sqrt{2}}h)},\ \field{\sin(\frac{1}{\sqrt{2}}h)} \right).
    \end{equation}
\end{thm}
\begin{proof}
    Take any subsequential limit as described in Lemma \ref{lem:local-set-2}. By Theorem \ref{thm:local}, we know that the two-valued set $A$ obtained from $h$ must also be a local set for $\phi$. Moreover, the boundary values of $\phi$ on a loop $\ell$ of $A$ are given by
    \begin{equation}
       (c(\ell) + 4k(\ell)) \sqrt{2}\lambda
    \end{equation}
    for some integer $k(\ell)$. Let $\tilde A$ be the two-valued set $\A_{-\sqrt{2}\lambda, \sqrt{2}\lambda}$ now obtained from $\phi$. By Proposition \ref{prop:TVS-unique}, we must have that
    \begin{equation}
        \tilde A \subseteq A\quad \text{a.s.}
    \end{equation}
    Since they have the same distribution, it must be the case that
    \begin{equation}
        \tilde A = A\quad \text{a.s.}
    \end{equation}
    By iterating this argument, we obtain a.s. equality of all nested two-valued sets. Since the labels under $h$ and under $\phi$ agree, we conclude that
    \begin{equation}
        h= \phi\quad \text{a.s.}
    \end{equation}
  Indeed, if we define $h_n$ and $\phi_n$ to be the harmonic functions obtained at level $n$, we have that $h_n=\phi_n$, but also $h_n\to h$ and $\phi_n \to \phi$ as $n\to \infty$. This concludes the proof, as we have uniquely characterized any subsequential limit in terms of $h$ and measurable functions of it.
\end{proof}

Once the joint convergence in Theorem \ref{thm:full-joint} has been established, identifying any subsequential limit of the discrete excursion decomposition as the (unique) excursion decomposition in the continuum is almost immediate. We stress that the proof below corresponds to a stronger statement than that in Theorem \ref{thm:discrete-conv}, see Remark \ref{rem:strong-conv}.

\begin{proof}[Proof of Theorem \ref{thm:discrete-conv}]
    We prove the claim under free boundary conditions. For a valid domain $D$, let 
   \begin{equation}
       (h_\delta,\ \hat\nb_\delta^\dagger,\ \xi^\dagger_\delta,\ \tau_\de^\dagger)
   \end{equation}
   be the joint law under the coupling $\P_\de$ of Theorem \ref{thm:master}. We identify $\hat\nb_\de^\dagger$ with the collection $((C_k^\de)^\dagger)_{k\geq1}$ of its clusters, which we order by decreasing size of diameter. By Corollary \ref{cor:discrete-exc}, we know that
    \begin{align*}
        (\tau_\delta^\dagger,f) = \sum_{k\geq 1}(\xi_k^\delta)^\dagger((\mu^\delta_k)^\dagger,f) 
    \end{align*}
    for any continuous, bounded function $f:\C\to\R$. Tightness of the collection of measures, with respect to the topology of weak convergence, readily follows from tightness of the field $\tau_\de^\dagger$, since e.g.
    \begin{equation}
        \sum_{k\in I}\E[((\mu^\delta_k)^\dagger, f)^2] \leq \E_\de[(\tau_\de^\dagger, f)^2] 
    \end{equation}
    for any (possibly infinite) index set $I\subset\N$.  This is an immediate consequence of the orthogonality of the signs, see also {\cite[Lemma 29]{excursion-GFF}}. 

    \noindent Combining the above tightness with Theorem \ref{thm:DRC} and Theorem \ref{thm:full-joint}, we have that there exists a subsequence of $\de\to0$, which denote the same way, along which\footnote{One can check that convergence of the inner and outer boundaries of $\hat\nb_\de^\dagger$ implies convergence of its clusters with respect to the standard Hausdorff metric on closed subsets of $D$.} 
    \begin{align}\label{eq:joint-conv-proof}
       \big(h_\delta,\ \tau_\de^\dagger,\ ((\mu^\delta_k)^\dagger)_{k\geq1},&\ ((C_k^\de)^\dagger)_{k\geq1},\ ((\xi_k^\de)^\dagger)_{k\geq1}\big)\\
       & \longrightarrow \left(\frac{1}{2\sqrt{2}\lambda}\ h,\ \field{\sin(\frac{1}{\sqrt{2}}h)}, \  ((\mu_k^*)^\dagger)_{k\geq1},\ (C_k^\dagger)_{k\geq1},\ (\xi_k^\dagger)_{k\geq1} \right).
   \end{align}
    Moreover, with the exception of the collection $((\mu_k^*)^\dagger)_{k\geq1}$, every element of the limiting coupling is measurable with respect to $h$. As usual, we pass to a common probability space in which convergence holds almost surely.

    \noindent It is now a standard argument, see e.g. {\cite[Section 4]{CME-review}}, to show that a decomposition of $\field{\sin((1/\sqrt{2})h)}$ holds in terms of $((\mu_k^*)^\dagger)_{k\geq1}$. To do so, fix some diameter cutoff $\rho>0$ and its associated number $N=N(\rho)$ of clusters with diameter greater than $\rho$. Recall that this number is a.s. finite by Proposition \ref{prop:local-finite-tree}. Since
   \begin{equation}
       (\tau_\delta^\dagger,f) = \sum_{k=1}^N (\xi_k^\delta)^\dagger((\mu^\delta_k)^\dagger,f) + (R_\de^\rho,f)
   \end{equation}
   for some remainder term $R_\de^\rho$, and as a finite sum
   \begin{equation}
       \sum_{k=1}^N (\xi_k^\delta)^\dagger((\mu^\delta_k)^\dagger,f) \longrightarrow \sum_{k=1}^N \xi_k^\dagger((\mu^*_k)^\dagger,f)
   \end{equation}
  almost surely and in $L^2(\P)$ as $\de\to0$, it follows that 
    \begin{equation}
        R_\de^\rho\longrightarrow\ R^\rho_*
    \end{equation}
  exists as an almost sure limit. A simple computation using \eqref{eq:onsager-discrete} yields that
  \begin{align}
      \E[(R^\rho_*, f)^2] &\leq \limsup_{\de\to0}\E_\de\Big[\sum_{k=N+1}^\infty((\mu^\delta_k)^\dagger, f)^2\Big],\\
      & \leq \norm{f}_\infty^2\limsup_{\de\to0}\E_\de\Big[\sum_{\substack{ u, u'\in D_\de^\dagger\\ |u-u'|\leq\rho}} \tau_\de^\dagger(u)\tau_\de^\dagger(u')\Big] \\
      & \leq \norm{f}_\infty^2\rho^{2-1/2} \longrightarrow0
  \end{align}
  as $\rho\to0$ (or equivalently $N\to\infty$). By Proposition \ref{prop:sign-permute}, 
  \begin{equation}\label{eq:decomp-star}
      (\field{\sin((1/\sqrt{2})h)},f)  = \sum_{k=1}^\infty\xi_k^\dagger((\mu^*_k)^\dagger,f)
  \end{equation}
    exists as a limit in $L^2(\P)$.
  
   \noindent We can now compare this to the continuum decomposition we have built. Namely, by Theorem \ref{thm:cont-general}, we have that
    \begin{equation}\label{eq:decomp-good}
        (\field{\sin((1/\sqrt{2})h)},f) = \sum_{k=1}^\infty \xi_k^\dagger(\mu_k^\dagger,f)
    \end{equation}
    We stress that it is crucial that, by Theorem \ref{thm:DRC} and Theorem \ref{thm:full-joint}, we know that
    \begin{enumerate}[label=(\roman*)]
        \item Both decompositions \eqref{eq:decomp-star}-\eqref{eq:decomp-good} are of the chaos of the \emph{same} underlying GFF.
        \item The supports of both $((\mu^*_k)^\dagger)_{k\geq1}$ and $(\mu_k^\dagger)_{k\geq 1}$ are the \emph{same} clusters $(C_k^\dagger)_{k\geq1}$, which in particular are coupled to $h$.
        \item The limit of the discrete signs $((\xi_k^\delta)^\dagger)_{k\geq1}$ are the \emph{same} labels $(\xi_k^\dagger)_{k\geq1}$ of the outer boundaries of the clusters that we use to build the continuum decomposition.
    \end{enumerate}
    For convenience, we rearrange the sums in \eqref{eq:decomp-star}-\eqref{eq:decomp-good} by decreasing level of nesting and then by decreasing size of diameter within each level of nesting. We proceed to identify the measures $(\mu^*_k)^\dagger$ as $\mu_k^\dagger$ one by one. By taking conditional expectation in both \eqref{eq:decomp-star} and \eqref{eq:decomp-good}, 
    \begin{equation}\label{eq:CE-meas}
        \xi_1^\dagger\ (\mu_1^\dagger,f) = \xi_1^\dagger\ \E\big[ ((\mu_1^*)^\dagger,f) \mid \xi_1^\dagger, C_1^\dagger\big],
    \end{equation}
    where we have used that $\mu_1^\dagger$ is measurable with respect to $C_1^\dagger$. Similarly, by taking conditional variances, 
    \begin{equation}
        \Var\Big[ \sum_{k=1}^\infty\xi_k^\dagger\mu_k^\dagger \mid \xi_1^\dagger, C_1^\dagger \Big] =   \Var\Big[ \sum_{k=1}^\infty\xi_k^\dagger(\mu^*_k)^\dagger \mid \xi_1^\dagger, C_1^\dagger \Big].
    \end{equation}
    On the one hand,
    \begin{equation}\label{eq:CV-good}
         \Var\Big[ \sum_{k=1}^\infty\xi_k^\dagger\mu_k^\dagger \mid \xi_1^\dagger, C_1^\dagger \Big] = \Var\Big[ \sum_{k=2}^\infty\xi_k^\dagger\mu_k^\dagger \mid \xi_1^\dagger, C_1^\dagger \Big],
    \end{equation}
    since the sum over $k\geq2$ is conditionally independent of the term $k=1$ and $\mu_1^\dagger$ is measurable with respect to $C_1^\dagger$. On the other hand, we have that
    \begin{equation}\label{eq:CV-star}
        \Var\Big[ \sum_{k=1}^\infty\xi_k^\dagger(\mu^*_k)^\dagger \mid \xi_1^\dagger, C_1^\dagger \Big] =  \Var[((\mu_1^*)^\dagger,f) \mid \xi_1^\dagger, C_1^\dagger] + \Var\Big[ \sum_{k=2}^\infty\xi_k^\dagger(\mu_k^*)^\dagger \mid \xi_1^\dagger, C_1^\dagger \Big].
    \end{equation}
    Here, we have used that the conditional independence between the sum over $k\geq2$ and the term $k=1$ holds for their discrete counterparts, by the Markov property \ref{MP-free} of Theorem \ref{thm:master}, and we can argue as in the proof of Lemmas \ref{lem:local-set-1}-\ref{lem:local-set-2} to show that the conditional independence carries over to the continuum. Moreover, this same Markov property and its continuum counterpart show that the sums over $k\geq2$ in \eqref{eq:CV-good} and \eqref{eq:CV-star} have the same conditional law, from which it follows that
    \begin{equation}\label{eq:CV-meas}
        \Var[((\mu_1^*)^\dagger,f) \mid \xi_1^\dagger, C_1^\dagger] = 0.
    \end{equation}
    Combining \eqref{eq:CV-meas} with \eqref{eq:CE-meas}, we see that
    \begin{equation}
        \mu_1^\dagger = (\mu_1^*)^\dagger\quad \text{a.s.}
    \end{equation}
    Iterating this argument for each subsequent term in the sums \eqref{eq:decomp-star}-\eqref{eq:decomp-good} concludes the proof.
\end{proof}

\begin{rem}\label{rem:strong-conv}
    The proof above has shown a stronger statement than the one in Theorem \ref{thm:discrete-conv}. For one, the convergence in \eqref{eq:joint-conv-proof} is jointly with $h_\de$ and identifies the limit of the XOR-Ising field as the chaos of its limit $h$. Moreover, thanks to Theorem \ref{thm:DRC}, it also follows that the discrete signs signs $((\xi_k^\delta)^\dagger)_{k\geq1}$, which are measurable functions of $h_\de$, converge to the labels $(\xi_k^\dagger)_{k\geq1}$, which are measurable functions of $h$. This differs from other known excursion decompositions, like those for the Ising model in \cite{mag-field, CME, ising-GFF}, where the coin tosses are independent of everything else in the coupling and their ``convergence'' is thus trivial.
\end{rem}

\begin{proof}[Proof of Corollary \ref{cor:joint-conv}]
    For joint convergence, consider the coupling in \cite{ising-GFF} and, up to taking a subsequence, add the joint convergence of both XOR-Ising models. The same local set argument as in Theorem \ref{thm:full-joint} characterizes any such subsequential limit uniquely. 

    \noindent It remains to prove \eqref{eq:wick-prod}. Let $\rho^\eps$ be a standard $\eps$-mollifier given by
    \begin{equation}
        \rho^\eps(w) := \eps^{-2}\rho(w/\eps),
    \end{equation}
    for a non-negative, radially symmetric function $\rho$ with unit mass. We define
    \begin{equation}
        \sigma^\eps(z) := (\sigma, \rho^\eps(z-\cdot)),
    \end{equation}
    and note this is a smooth function. 
    
    \noindent To see that $\field{\sigma\tilde\sigma}$ is well-defined, one may check via a direct computation that
    \begin{equation}
        \E[((\sigma^\eps\tilde\sigma^\eps, f)-(\sigma^{\eps'}\tilde\sigma^{\eps'}))^2]\longrightarrow0
    \end{equation}
    as $\eps, \eps'\to0$, for any continuous, bounded $f:\C\to\R$ in a valid domain. Moreover, one can similarly see that all moments of $\field{\sigma\tilde\sigma}$ match those of the XOR-Ising, e.g. 
    \begin{align}
        \E[(\sigma_\eps\tilde\sigma_\eps, f)] &= \int_D \E[(\sigma, \rho^\eps(z-\cdot))]\ \E[(\tilde\sigma, \rho^\eps(z-\cdot))]f(z)dz \\
        &= \int_D \left(\int_D \CR(x, D)^{-1/8}\rho^\eps(z-x)dx\right)\left(\int_D\CR(y, D)^{-1/8}\rho^\eps(z-y)dy\right)f(z)dz\\
        & = \int_D \CR(x, D)^{-1/8} \left(\int_D \CR(y, D)^{-1/8}\left( \int_D \rho^\eps(z-x)\rho^\eps(z-y)f(z)dz\right)dy\right)dx\\
        & \longrightarrow \int_D \CR(x, D)^{-1/4}f(x)dx \label{eq:one-point}
    \end{align}
    as $\eps\to0$. Note that, crucially, the field $\field{\sigma\tilde\sigma}$ is measurable with respect to the pair $(\sigma, \tilde\sigma)$. 

    \noindent Now, it suffices to show that
    \begin{equation}\label{eq:triple-XOR}
        (\sigma_\de,\ \tilde\sigma_\de,\ \tau_\de) \overset{(d)}{\longrightarrow} (\sigma,\ \tilde\sigma,\ \field{\sigma\tilde\sigma})
    \end{equation}
    as $\de\to0$. Indeed, we already know that
    \begin{equation}
        \big(h_\de,\ \sigma_\de,\ \tilde\sigma_\de,\ \tau_\delta,\ \sigma_\de^\dagger,\ \tilde\sigma_\de^\dagger,\ \tau_\delta^\dagger\big)
    \end{equation}
    converges in law as $\delta\to0$ to
    \begin{equation}\label{eq:full-joint-law}
        \big(1/(\pi\sqrt{2})h,\ \sigma,\ \tilde\sigma,\ \field{\cos((1/\sqrt{2})h)},\ \sigma^\dagger,\ \tilde\sigma^\dagger,\ \field{\sin((1/\sqrt{2})h)}\big).
    \end{equation}
    But then by \eqref{eq:triple-XOR}, and its free boundary condition counterpart, we know that the marginal laws in the coupling \eqref{eq:full-joint-law} are such that
    \begin{equation}
        \field{\cos((1/\sqrt{2})h)}\ \overset{(d)}{=}\ \field{\sigma\tilde\sigma}\ , \quad  \ \field{\sin((1/\sqrt{2})h)}\ \overset{(d)}{=}\ \field{\sigma^\dagger\tilde\sigma^\dagger}\ .
    \end{equation}
    But, in particular, it follows that
    \begin{equation}
        \field{\cos((1/\sqrt{2})h)}\ , \quad   \field{\sin((1/\sqrt{2})h)}
    \end{equation}
    are measurable functions of the pairs $(\sigma, \tilde\sigma)$ and $(\sigma^\dagger, \tilde\sigma^\dagger)$, respectively, and a.s. equality must hold. 

    \noindent The proof of \eqref{eq:triple-XOR} relies on a simple (especially using the independence of the two Ising models) joint moment computation. Firstly, by expanding into the integral forms, we see that
    \begin{equation}
        \E_\de[(\sigma_\de, f)^n(\tilde\sigma_\de, f)^m(\tau_\de,f)^k] = \E_\de[(\sigma_\de, f)^{n+k}]\ \E[(\tilde\sigma_\de, f)^{m+k}].
    \end{equation}
    On the other hand, since $(\sigma^\eps\tilde\sigma^\eps, f)$ has uniformly bounded moments of all orders, direct computations similar to that in \eqref{eq:one-point} show that 
    \begin{equation}
         \E[(\sigma, f)^n(\sigma, f)^m(\field{\sigma\tilde\sigma},f)^k] = \E[(\sigma, f)^{n+k}]\ \E[(\tilde\sigma, f)^{m+k}].
    \end{equation}
    The claim now follows from the convergence of Ising correlation functions \cite[Theorem~1.1]{CHI}.
\end{proof}

\begin{rem}\label{rem:meas-relation}
    As already explained in the introduction, the identification \eqref{eq:wick-prod} combined with \cite[Theorem 1.1]{GFF-IGMC} implies the measurability of the limit of the height function with respect to the four Ising magnetisation fields. Interestingly enough, both in the discrete and the continuum, given the XOR-Ising fields one only recovers the \emph{gradient} of the height function (recall \eqref{eq:d-grad}), and then imposes boundary conditions to fix the additive constant. In \cite{GFF-IGMC}, it is discussed whether the GFF should be measurable with respect to \emph{only} its sine or \emph{only} its cosine. The correspondence with discrete models suggests that the answer should be negative, and perhaps one can even use the continuum excursion decomposition to prove this claim.
\end{rem}
\bigskip


\printbibliography

@article{dim-TVS,
author = {Schoug, Lukas and Sep\'ulveda, Avelio and Viklund, Fredrik},
    title = {Dimensions of Two-Valued Sets via Imaginary Chaos},
    journal = {International Mathematics Research Notices},
    volume = {2022},
    number = {5},
    pages = {3219--3261},
    year = {2020},
}

@phdthesis{Aru,
  title={{The geometry of the Gaussian free field combined with SLE processes and the KPZ relation}},
  author={Aru, Juhan},
  year={2015},
  school={{\'E}cole Normale Sup{\'e}rieure de lyon}
}

@article{TVS,
  TITLE = {{Two-valued local sets of the 2D continuum Gaussian free field: connectivity, labels, and induced metrics}},
  AUTHOR = {Aru, Juhan and Sep{\'u}lveda, Avelio},
  JOURNAL = {{Electronic Journal of Probability}},
  VOLUME = {23},
  YEAR = {2018},
}

@article{CGS,
  title={{The critical Ising magnetisation field can be reconstructed from its {\(+/-\)} interfaces}},
  author={Cahen, Paul and Garban, Christophe and Sep{\'u}lveda, Avelio},
  year={2026+},
notes={To appear}, 
}

@article{ASZ,
  title={{Non-existence of the level sets of the Gaussian free field in higher dimensions}},
  author={Araya, Pablo and Sep{\'u}lveda, Avelio and Z{\'u}{\~n}iga, Pablo},
  year={2026+},
note={To appear}
}

@article{BTLS,
  title={{On bounded-type thin local sets of the two-dimensional Gaussian free field}},
  author={Aru, Juhan and Sep{\'u}lveda, Avelio and Werner, Wendelin},
  journal={Journal of the Institute of Mathematics of Jussieu},
  volume={18},
  number={3},
  pages={591--618},
  year={2019},
  publisher={Cambridge University Press}
}

@article{thin-GFF,
author = {Avelio Sep{\'u}lveda},
title = {{On thin local sets of the Gaussian free field}},
volume = {55},
journal = {Annales de l'Institut Henri Poincar\'e, Probabilit\'es et Statistiques},
number = {3},
pages = {1797 -- 1813},
year = {2019},
}

@article{DRC-1,
  title={{Conformal invariance of double random currents I: identification of the limit}},
  author={Duminil-Copin, Hugo and Lis, Marcin and Qian, Wei},
  journal={Proceedings of the London Mathematical Society},
  volume={130},
  number={1},
  pages={e70022},
  year={2025},
  publisher={Wiley Online Library}
}

@article{DRC-2,
  title={{Conformal invariance of double random currents II: tightness and properties in the discrete}},
  author={Duminil-Copin, Hugo and Lis, Marcin and Qian, Wei},
  journal={ArXiv preprint arXiv:2107.12880},
  year={2021}
}

@article{IGMC,
  title={{Imaginary Multiplicative Chaos}},
  author={Junnila, Janne and Saksman, Eero and Webb, Christian},
  journal={The Annals of Applied Probability},
  volume={30},
  number={5},
  pages={2099--2164},
  year={2020},
  publisher={JSTOR}
}

@article{fields-loo-soups,
  title={{Conformally invariant fields out of Brownian loop soups}},
  author={Jego, Antoine and Lupu, Titus and Qian, Wei},
  journal={ArXiv preprint arXiv:2307.10740},
  year={2023}
}

@article{Lupu,
  title={{From loop clusters and random interlacements to the free field}},
  author={Lupu, Titus},
  journal={Annals of Probability},
  volume={44},
  number={3},
  pages={2117--2146},
  year={2016}
}

@article{CHI-2,
  title={{Correlations of primary fields in the critical Ising model}},
  author={Chelkak, Dmitry and Hongler, Cl{\'e}ment and Izyurov, Konstantin},
  journal={ArXiv preprint arXiv:2103.10263},
  year={2021}
}

@article{CHI,
  title={{Conformal invariance of spin correlations in the planar Ising model}},
  author={Chelkak, Dmitry and Hongler, Cl{\'e}ment and Izyurov, Konstantin},
  journal={Annals of Mathematics},
  pages={1087--1138},
  year={2015},
  publisher={JSTOR}
}

@article{spin-perc-height,
  title={{Spins, percolation and height functions}},
  author={Lis, Marcin},
  journal={Electronic Journal of Probability},
  volume={27},
  pages={1--21},
  year={2022},
  publisher={The Institute of Mathematical Statistics and the Bernoulli Society}
}

@article{mag-field,
  title={{Planar Ising magnetization field I. Uniqueness of the critical scaling limit}},
  author={Camia, Federico and Garban, Christophe and Newman, Charles M},
  journal={The Annals of Probability},
  pages={528--571},
  year={2015},
  publisher={JSTOR}
}

@inproceedings{CME,
  title={{Conformal Measure Ensembles for Percolation and the FK--Ising Model}},
  author={Camia, Federico and Conijn, Ren{\'e} and Kiss, Demeter},
  booktitle={{Sojourns in Probability Theory and Statistical Physics-II}},
  pages={44--89},
  year={2019},
  organization={Springer}
}

@article{CME-review,
  title={{Conformal measure ensembles and planar Ising magnetization: A review}},
  author={Camia, Federico and Jiang, Jianping and Newman, Charles M},
  journal={ArXiv preprint arXiv:2009.08129},
  year={2020}
}

@article{switching,
  title={{Concavity of magnetization of an Ising ferromagnet in a positive external field}},
  author={Griffiths, Robert and Hurst, Charles and Sherman, Seymour},
  journal={Journal of Mathematical Physics},
  volume={11},
  number={3},
  pages={790--795},
  year={1970},
  publisher={AIP Publishing}
}

@article{notes-GFF,
  title={{Lecture notes on the Gaussian free field}},
  author={Werner, Wendelin and Powell, Ellen},
  journal={ArXiv preprint arXiv:2004.04720},
  year={2020}
}

@article{quant-SLE,
  title={{An elementary approach to quantum length of SLE}},
  author={Powell, Ellen and Sep{\'u}lveda, Avelio},
  journal={ArXiv preprint arXiv:2403.03902},
  year={2024}
}

@article{curr-prob,
  title={{The planar Ising model and total positivity}},
  author={Lis, Marcin},
  journal={Journal of Statistical Physics},
  volume={166},
  pages={72--89},
  year={2017},
  publisher={Springer}
}

@article{tightness-fields,
  author = {Marco Furlan and Jean-Christophe Mourrat},
title = {{A tightness criterion for random fields, with application to the Ising model}},
volume = {22},
journal = {Electronic Journal of Probability},
pages = {1 -- 29},
year = {2017},
}

@article{excursion-GFF,
  title={{Excursion decomposition of the 2D continuum GFF}},
  author={Aru, Juhan and Lupu, Titus and Sep{\'u}lveda, Avelio},
  journal={ArXiv preprint arXiv:2304.03150},
  year={2023}
}

@article{ALS2,
  title={{The first passage sets of the 2D Gaussian free field: convergence and isomorphisms}},
  author={Aru, Juhan and Lupu, Titus and Sep{\'u}lveda, Avelio},
  journal={Communications in Mathematical Physics},
  volume={375},
  number={3},
  pages={1885--1929},
  year={2020},
  publisher={Springer}
}

@article{ALS1,
  title={{First passage sets of the 2D continuum Gaussian free field}},
  author={Aru, Juhan and Lupu, Titus and Sep{\'u}lveda, Avelio},
  journal={Probability Theory and Related Fields},
  volume={176},
  number={3},
  pages={1303--1355},
  year={2020},
  publisher={Springer}
}

@article{nesting-field,
  title={{On the double random current nesting field}},
  author={Duminil-Copin, Hugo and Lis, Marcin},
  journal={Probability Theory and Related Fields},
  volume={175},
  pages={937--955},
  year={2019},
  publisher={Springer}
}

@article{dimer-boson,
  title={{Exact bosonization of the Ising model}},
  author={Dub{\'e}dat, Julien},
  journal={ArXiv preprint arXiv:1112.4399},
  year={2011}
}

@article{dimer-electric,
  title={{Dimers and families of Cauchy-Riemann operators I}},
  author={Dub{\'e}dat, Julien},
  journal={Journal of the American Mathematical Society},
  volume={28},
  number={4},
  pages={1063--1167},
  year={2015}
}

@article{cont-boson,
  title={{Bosonization of Primary Fields for the Critical Ising Model on Multiply Connected Planar Domains}},
  author={Bayraktaroglu, Baran and Izyurov, Konstantin and Virtanen, Tuomas and Webb, Christian},
  journal={Communications in Mathematical Physics},
  volume={406},
  number={9},
  pages={222},
  year={2025},
  publisher={Springer}
}

@article{IGMC2,
  title={{Complex Gaussian multiplicative chaos}},
  author={Lacoin, Hubert and Rhodes, R{\'e}mi and Vargas, Vincent},
  journal={Communications in Mathematical Physics},
  volume={337},
  pages={569--632},
  year={2015},
  publisher={Springer}
}

@article{noise-IGMC,
  title={{Noise-like analytic properties of imaginary chaos}},
  author={Aru, Juhan and Baverez, Guillaume and Jego, Antoine and Junnila, Janne},
  journal={Electronic Journal of Probability},
  volume={30},
  pages={1--43},
  year={2025},
  publisher={The Institute of Mathematical Statistics and the Bernoulli Society}
}

@article{plasma,
  title={{Large deviations for the two-dimensional two-component plasma}},
  author={Lebl{\'e}, Thomas and Serfaty, Sylvia and Zeitouni, Ofer},
  journal={Communications in Mathematical Physics},
  volume={350},
  pages={301--360},
  year={2017},
  publisher={Springer}
}

@article{LK-hadamard,
  title={{The Loewner--Kufarev energy and foliations by Weil--Petersson quasicircles}},
  author={Viklund, Fredrik and Wang, Yilin},
  journal={Proceedings of the London Mathematical Society},
  volume={128},
  number={2},
  year={2024},
  publisher={Wiley Online Library}
}

@article{other-hadamard,
  title={{The Loewner and Hadamard variations}},
  author={Roth, Oliver and Schippers, Eric},
  journal={Illinois Journal of Mathematics},
  volume={52},
  number={4},
  pages={1399--1415},
  year={2008},
  publisher={Duke University Press}
}

@article{miller-schoug,
  author = {Jason Miller and Lukas Schoug},
title = {{Existence and uniqueness of the conformally covariant volume measure on conformal loop ensembles}},
volume = {60},
journal = {Annales de l'Institut Henri Poincar\'e, Probabilit\'es et Statistiques},
number = {4},
pages = {2267 -- 2296},
year = {2024},
}

@article{ising-GFF,
  title={{The Ising magnetisation field and the Gaussian free field}},
  author={Alcalde L{\'o}pez, Tom{\'a}s and Heeney, Lorca and Lis, Marcin},
  journal={ArXiv preprint  	arXiv:2602.05886},
  year={2026}
}

@article{kada-brown,
  title={{Correlation functions on the critical lines of the Baxter and Ashkin-Teller models}},
  author={Kadanoff, Leo and Brown, Alan},
  journal={Annals of Physics},
  volume={121},
  number={1-2},
  pages={318--342},
  year={1979},
  publisher={Elsevier}
}

@article{GFF-IGMC,
  title={{Reconstructing the base field from imaginary multiplicative chaos}},
  author={Aru, Juhan and Junnila, Janne},
  journal={Bulletin of the London Mathematical Society},
  volume={53},
  number={3},
  pages={861--870},
  year={2021},
  publisher={Wiley Online Library}
}

@article{camia-excur,
  title={{Conformal covariance of connection probabilities and fields in 2D critical percolation}},
  author={Camia, Federico},
  journal={Communications on Pure and Applied Mathematics},
  volume={77},
  number={3},
  pages={2138--2176},
  year={2024},
  publisher={Wiley Online Library}
}

@inproceedings{ito,
  title={{Poisson point processes attached to Markov processes}},
  author={It{\^o}, Kiyosi},
  booktitle={{Proceedings of the Sixth Berkeley Symposium on Mathematical Statistics and Probability}},
  volume={3},
  pages={225--239},
  year={1972}
}

@article{ES,
  title={{Generalization of the Fortuin-Kasteleyn-Swendsen-Wang representation and Monte Carlo algorithm}},
  author={Edwards, Robert and Sokal, Alan },
  journal={Physical review D},
  volume={38},
  number={6},
  pages={2009},
  year={1988},
  publisher={APS}
}

@article{aru-papon-powell,
  title={{Thick points of the planar GFF are totally disconnected for all {\(\gamma\ne 0\)}}},
  author={Aru, Juhan and Papon, L{\'e}onie and Powell, Ellen},
  journal={Electronic Journal of Probability},
  volume={28},
  pages={1--24},
  year={2023},
  publisher={The Institute of Mathematical Statistics and the Bernoulli Society}
}

@article{AT-crit,
  title={{Dual Transformations in Many-Component Ising Models}},
  author={Mittag, L and Stephen, MJ},
  journal={Journal of Mathematical Physics},
  volume={12},
  number={3},
  pages={441--450},
  year={1971},
  publisher={American Institute of Physics}
}

@article{AT-curr,
  title={{On Boundary Correlations in Planar Ashkin--Teller Models}},
  author={Lis, Marcin},
  journal={International Mathematics Research Notices},
  volume={2022},
  number={13},
  pages={9909--9940},
  year={2022},
  publisher={Oxford University Press}
}

@article{DSS,
  title={{A new correlation inequality for Ising models with external fields}},
  author={Ding, Jian and Song, Jian and Sun, Rongfeng},
  journal={Probability Theory and Related Fields},
  volume={186},
  number={1},
  pages={477--492},
  year={2023},
  publisher={Springer}
}

@book{lawler,
  title={{Conformally invariant processes in the plane}},
  author={Lawler, Gregory F},
  number={114},
  year={2008},
  publisher={American Mathematical Soc.}
}

@book{evans,
  title={{Partial differential equations}},
  author={Evans, Lawrence },
  volume={19},
  year={2022},
  publisher={American Mathematical Society}
}

@article{ber-SLE,
  title={{Lectures on {Schramm--Loewner} evolution}},
  author={Berestycki, Nathanaël and Norris, James},
  journal={Available on the webpages of the authors},
  year={2023}
}

@book{BP, place={Cambridge}, series={Cambridge Studies in Advanced Mathematics}, title={{Gaussian free field and Liouville quantum gravity}},
publisher={Cambridge University Press}, author={Berestycki, Nathanaël and Powell, Ellen}, year={2025}, collection={Cambridge Studies in Advanced Mathematics}}

@article{Wil,
  title={{XOR-Ising loops and the Gaussian free field}},
  author={Wilson, David B},
  journal={arXiv:1102.3782},
  year={2011}
}

@article{Cara,
  title={{Uniform convergence of Green's functions}},
  author={Kalmykov, Sergei and Kovalev, Leonid V},
  journal={Complex Variables and Elliptic Equations},
  volume={64},
  number={4},
  pages={557--562},
  year={2019},
  publisher={Taylor \& Francis}
}

@article{CS,
  title={{Discrete complex analysis on isoradial graphs}},
  author={Chelkak, Dmitry and Smirnov, Stanislav},
  journal={Advances in Mathematics},
  volume={228},
  number={3},
  pages={1590--1630},
  year={2011},
  publisher={Elsevier}
}

@article{AT,
  title={{Statistics of two-dimensional lattices with four components}},
  author={Ashkin, Julius and Teller, Edward},
  journal={Physical Review},
  volume={64},
  number={5-6},
  pages={178},
  year={1943},
  publisher={APS}
}

@book{doob,
  title={{Classical potential theory and its probabilistic counterpart}},
  author={Doob, Joseph L and Doob, JI},
  volume={262},
  year={1984},
  publisher={Springer}
}

@article{She,
  title={{Gaussian free fields for mathematicians}},
  author={Sheffield, Scott},
  journal={Probability theory and related fields},
  volume={139},
  number={3},
  pages={521--541},
  year={2007},
  publisher={Springer}
}

@article{SchSh,
  title={{A contour line of the continuum Gaussian free field}},
  author={Schramm, Oded and Sheffield, Scott},
  journal={Probability Theory and Related Fields},
  volume={157},
  number={1},
  pages={47--80},
  year={2013},
  publisher={Springer}
}

@article{AL26,
  title={{CLE{\(_4\)} two-point correlation functions in the unit disk}},
  author={Aru, Juhan and Lupu, Titus},
  year={2026+},
  note={To appear}, 
}

@article{FK,
title = {{On the random-cluster model: I. Introduction and relation to other models}},
journal = {Physica},
volume = {57},
number = {4},
pages = {536-564},
year = {1972},
issn = {0031-8914},
author = {Cees Fortuin and Pieter Kasteleyn},
}

@article{XOR-dimer,
 author = {C{\'e}dric Boutillier and B{\'e}atrice de Tili{\`e}re},
title = {{Height representation of XOR-Ising loops via bipartite dimers}},
volume = {19},
journal = {Electronic Journal of Probability},
pages = {1--33},
year = {2014},
}

@article{boson-phys,
  title = {{Quantum field theory and the two-dimensional Ising model}},
  author = {Zuber, Jean-Bernard and Itzykson, Claude},
  journal = {Phys. Rev. D},
  volume = {15},
  issue = {10},
  pages = {2875--2884},
  numpages = {0},
  year = {1977},
  publisher = {American Physical Society}
}

@article{nienhuis,
  title={{Critical behavior of two-dimensional spin models and charge asymmetry in the Coulomb gas}},
  author={Nienhuis, Bernard},
  journal={Journal of Statistical Physics},
  volume={34},
  number={5},
  pages={731--761},
  year={1984},
  publisher={Springer}
}

@article{c1,
  title={{c{\(=\)}1 conformal field theories on Riemann surfaces}},
  author={Dijkgraaf, Robbert and Verlinde, Erik and Verlinde, Herman},
  journal={Communications in Mathematical Physics},
  volume={115},
  number={4},
  pages={649--690},
  year={1988},
  publisher={Springer}
}

@article{ShW,
  title={{Conformal loop ensembles: the Markovian characterization and the loop-soup construction}},
  author={Sheffield, Scott and Werner, Wendelin},
  journal={Annals of Mathematics},
  pages={1827--1917},
  year={2012},
  publisher={JSTOR}
}
\end{document}